\documentclass[reqno]{article}

\usepackage{mathrsfs}
\usepackage[utf8]{inputenc}
\usepackage{authblk}
\usepackage{amsthm}
\usepackage{amsmath}
\usepackage{amssymb}
\usepackage{graphicx}
\usepackage{xcolor}
\definecolor{dblue}{rgb}{0,0,0.45}
\usepackage{bbm}
\usepackage[top=2.5cm, bottom=2.5cm, left=2.5cm, right=2.5cm]{geometry}
\usepackage{enumerate}
\usepackage{faktor}
\usepackage{nicefrac}
\usepackage[nottoc]{tocbibind}
\usepackage{slashed}
\usepackage{authblk}
\usepackage[normalem]{ulem}
\usepackage{bm}
\usepackage{comment}
\usepackage{todonotes}
\usepackage{ulem}
\usepackage[colorlinks=true, allcolors=dblue, breaklinks=true]{hyperref}
\usepackage{upgreek}

\usepackage{upgreek}
\usepackage{tensor}
\usepackage{cancel}
\usepackage[capitalise]{cleveref}

\usepackage[
  backend=biber,
  style=numeric,
  sorting=none,
  doi=true,
  maxnames= 6,
  giveninits=true,
  citestyle=numeric-comp,
  url=false,
  isbn=false
]{biblatex}

\renewbibmacro{in:}{}
\DeclareFieldFormat[article]{volume}{\mkbibbold{#1}}

\DeclareFieldFormat[article]{title}{#1}

\DeclareFieldFormat{journaltitle}{#1}

\DeclareFieldFormat{volume}{\mkbibbold{#1}}

\DeclareFieldFormat{pages}{#1}
\DeclareFieldFormat{eid}{#1}

\renewbibmacro*{note+pages}{%
  \printfield{note}%
}

\newbibmacro*{linkedjournalblock}{%
  \printfield{journaltitle}%
  \setunit{\addspace}%
  \printfield{volume}%
  \setunit{\addcomma\space}%
  \iffieldundef{eid}
    {\printfield{pages}}
    {\printfield{eid}}%
  \setunit{\addspace}%
  \printtext{\mkbibparens{\printfield{year}}}%
}

\renewbibmacro*{journal+issuetitle}{%
  \iffieldundef{journaltitle}
    {}
    {%
      \iffieldundef{doi}
        {\usebibmacro{linkedjournalblock}}
        {\href{https://doi.org/\thefield{doi}}{\usebibmacro{linkedjournalblock}}}%
      \newunit
    }%
}

\DeclareFieldFormat[book]{title}{%
  \iffieldundef{doi}
    {#1}
    {\href{https://doi.org/\thefield{doi}}{#1}}%
}

\DeclareFieldFormat[inbook]{title}{%
  \iffieldundef{doi}
    {#1}
    {\href{https://doi.org/\thefield{doi}}{#1}}%
}
\DeclareFieldFormat[incollection]{title}{%
  \iffieldundef{doi}
    {#1}
    {\href{https://doi.org/\thefield{doi}}{#1}}%
}

\newcommand{\N}{\mathbb{N}}
\newcommand{\R}{\mathbb{R}}
\newcommand{\C}{\mathbb{C}}
\newcommand{\rd}{\partial}
\newcommand{\bD}{{D}}
\newcommand{\bB}{{B}}
\newcommand{\bE}{{E}}
\newcommand{\bH}{{H}}
\newcommand{\cE}{\mathcal{E}}
\newcommand{\be}{{e}}
\newcommand{\bh}{{h}}
\newcommand{\bF}{{F}}
\newcommand{\bG}{{G}}
\newcommand{\bS}{\mathcal{S}}
\newcommand{\bbE}{\mathbb{E}}
\newcommand{\bbX}{\mathbb{X}}
\newcommand{\bbP}{\mathbb{P}}
\newcommand{\bbJ}{\mathbb{J}}
\newcommand{\bbQ}{\mathbb{Q}}
\newcommand{\n}{\mathfrak{n}}

\def\ii{\mathrm i}

\newcommand{\phu}{\bm{\upphi}}
\newcommand{\eu}{\mathbf{e}}
\newcommand{\hu}{\mathbf{h}}
\newcommand{\Eu}{\mathbf{E}}
\newcommand{\Hu}{\mathbf{H}}
\newcommand{\x}{x}

\newcommand{\gammu}{\underline{\gamma}}

\renewcommand{\baselinestretch}{1.2}

\numberwithin{equation}{section}
\newtheorem{theorem}[equation]{Theorem}
\newtheorem{remark}[equation]{Remark}

\newtheorem{lemma}[equation]{Lemma}
\newtheorem{definition}[equation]{Definition}
\newtheorem{corollary}[equation]{Corollary}
\newtheorem{proposition}[equation]{Proposition}

\renewcommand{\div}{\mathrm{div}}

\newcommand{\eps}{\varepsilon}

\newcommand{\E}{{E}}
\renewcommand{\H}{{H}}
\newcommand{\Eh}{\hat{\E}}
\newcommand{\Hh}{\hat{\H}}
\newcommand{\e}{e}
\newcommand{\h}{h}
\renewcommand{\Im}{\mathfrak{Im}}
\renewcommand{\Re}{\mathfrak{Re}}
\newcommand{\bC}{{C}}
\newcommand{\bK}{{K}}
\newcommand{\cl}{\mathrm{cl}}
\newcommand{\uu}{\mathbf{u}}
\newcommand{\vu}{\mathbf{v}}
\newcommand{\Ab}{\mathbf{A}}

\title{A mathematical description of the spin Hall effect of light in inhomogeneous media}

\usepackage{authblk}

\newcommand{\mailto}[1]{\thanks{\href{mailto:#1}{#1}}}

\author[1]{Sam C. Collingbourne\mailto{sc.collingbourne@ed.ac.uk}}
\author[2]{Marius A. Oancea\mailto{marius.oancea@univie.ac.at}}
\author[1]{Jan Sbierski\mailto{jan.sbierski@ed.ac.uk}}
\affil[1]{School of Mathematics and Maxwell Institute for Mathematical Sciences, University of Edinburgh, James Clerk Maxwell Building, Peter Guthrie Tait Road, Edinburgh, EH9 3FD, United Kingdom}
\affil[2]{University of Vienna, Faculty of Physics, Boltzmanngasse~5, 1090 Vienna, Austria}

\date{}

\begin{document}

\maketitle

\begin{abstract}
We study Gaussian wave packet solutions for Maxwell's equations in an isotropic, inhomogeneous medium and derive a system of ordinary differential equations that captures the leading-order correction to geodesic motion. The dynamical quantities in this system are the energy centroid, the linear and angular momentum, and the quadrupole moment. Furthermore, the system is closed to first order in the inverse frequency. As an immediate consequence, the energy centroids of Gaussian wave packets with opposite circular polarisations generally propagate in different directions, thereby providing a mathematical proof of the spin Hall effect of light in an inhomogeneous medium.
\end{abstract}

\begingroup
\small
\renewcommand{\baselinestretch}{0.90}\normalsize
\tableofcontents
\endgroup

\section{Introduction}

Spin Hall effects are present in many areas of physics, such as condensed matter physics \cite{SHE_review,RevModPhys.82.1959,PhysRevLett.92.126603,originalSHE1,originalSHE2,originalSHE3,originalSHE4,PhysRevLett.95.137204}, optics \cite{SOI_review,SHEL_review,SHE-L_original,SHE_original,Bliokh2009,Duval2006,Duval2007,Bliokh2008,Hosten2008,deNittis2017,OpticalMagnus}, general relativity \cite{GSHE_reviewCQG,SHE_QM1,SHE_Dirac,Oancea_2020,O2,PhysRevD.102.084013,PhysRevD.110.064020,PhysRevD.109.064020,PhysRevD.111.024034,SHE_GW,GSHE_GW,Frolov_2024,GSHE_lensing,GSHE_lensing2,GSHE_Dirac,rudiger,audretsch}, and high energy physics \cite{PhysRevD.104.054043,HIDAKA2022103989,KHARZEEV20161}. The characteristic property of these effects is that localised wave packets carrying spin angular momentum are scattered or propagated in a spin-dependent way. This behaviour is due to spin-orbit coupling mechanisms \cite{SHE_review,SOI_review} represented by mutual interactions between the external (average position and average momentum) and the internal (spin angular momentum) degrees of freedom of a wave packet. While on a fundamental level the description of such wave packets is given by a partial differential equation (such as Schrödinger, Dirac, Maxwell, linearised gravity), spin Hall effects are usually derived using semiclassical methods. This leads to an approximate description in terms of a system of ordinary differential equations, where the propagation of the wave packet is approximately represented as a point particle that follows a spin-dependent trajectory. 

\begin{figure}[t!]
    \centering
    \includegraphics[width=.8\textwidth]{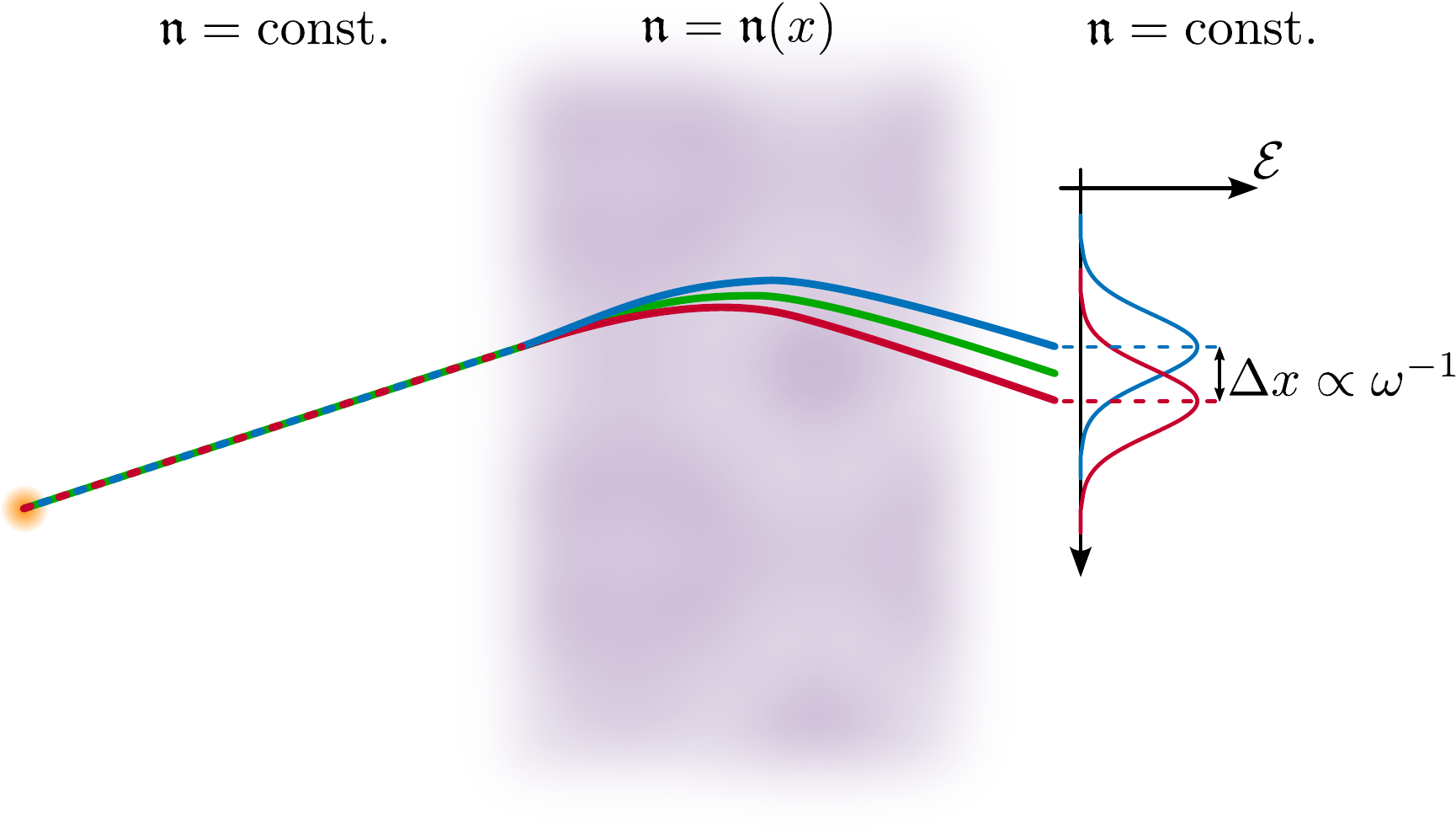}
    \label{fig:SHE}
    \caption{A sketch of the spin Hall effect of light in an inhomogeneous medium. The geometric optics geodesic is represented by the green ray, which is independent of the frequency and of the polarisation. When higher-order spin Hall corrections to the geometric optics approximation are included, we obtain frequency- and polarisation-dependent rays. This leads to the spin Hall effect of light, where the propagation of the energy centroids of wave packets with opposite circular polarisation is then represented by the blue and red rays. On the right part of the figure, the spin Hall ray separation is also represented in terms of the shifted energy densities of the two wave packets with opposite circular polarisations.}
\end{figure}

In optics, the spin Hall effect of light typically arises during the propagation of electromagnetic wave packets in inhomogeneous media with a smoothly varying refractive index \cite{SHE-L_original,SHE_original,BLIOKH2004181,Duval2006,PhysRevE.70.026605,Bliokh2009,PhysRevA.92.043805}, or in association with reflection and refraction processes at the interface between two distinct media where there is a discontinuous jump of the refractive index \cite{Fedorov_2013,PhysRevD.5.787,PhysRevE.75.066609,PhysRevLett.96.073903,Aiello:08,Bliokh_2013} (similar effects are also present for light propagating in optical fibres \cite{OpticalMagnus,PhysRevResearch.5.023140,PhysRevResearch.7.013162}). Most importantly, spin Hall effects of light have been observed in many experiments, where polarisation-dependent shifts of the energy centroids of electromagnetic wave packets have been measured \cite{Hosten2008,Bliokh2008,Qin:09,PradeepChakravarthy:19,LOFFLER20103367}. An overview of these effects and their applications can be found in \cite{SOI_review,SHEL_review}. 

A sketch of the spin Hall effect of light in an inhomogeneous medium is presented in Figure \ref{fig:SHE} (see also \cite{SOI_review} for other similar illustrations of the effect). Here,   we consider a medium with a smoothly varying refractive index, where a central region is sandwiched between two regions of constant refractive index. In particular, we assume that there are no sharp interfaces and that the refractive index varies smoothly between these regions. In the region of constant refractive index to the left of the figure, we prescribe initial data representing localised wave packets, where the only difference between the considered wave packets is the state of circular polarisation. For reference, we include the geometric optics geodesic ray represented in green, which is independent of the state of polarisation and the frequency. However, if we include
higher-order spin Hall corrections to the geometric optics approximation, the propagation of the wave packets becomes frequency- and polarisation-dependent. In this case, the rays followed by the centre of energy of the wave packets of opposite circular polarisation (blue and red rays in Figure \ref{fig:SHE}) coincide with the geometric optics ray (green ray in Figure \ref{fig:SHE}) in the initial region of constant refractive index but drift apart in a frequency- and polarisation-dependent way as soon as the inhomogeneous region is reached. This polarisation-dependent propagation of electromagnetic wave packets represents the spin Hall effect of light. In comparison to the geometric optics ray, the spin Hall rays will generally drift in a direction orthogonal to the direction of propagation and to the gradient of the refractive index, and the magnitude of the drift is proportional to the wavelength. The geometric optics rays are recovered in the limit of zero wavelength, or equivalently, infinitely high frequency. In other words, the spin Hall correction term for geometric optics geodesics is proportional to \cite{SOI_review}
\begin{equation} \label{EqDisplacementSpinHall}
    \frac{s}{\omega} p \times \nabla \n,
\end{equation}
where $s = \pm 1$ is a constant determined by the state of circular polarisation, $\omega$ is the frequency of the wave, $p$ is the linear momentum which represents (to leading order in $\frac{1}{\omega}$) the direction of propagation of the wave packet and $\n$ is the refractive index of the medium. One may call the correction \eqref{EqDisplacementSpinHall} to  geodesic motion due to different circular polarisations the `classical' spin Hall effect. Already here it is worthwhile to point out that further correction terms to  geodesic motion which are proportional to $\frac{1}{\omega}$ are in general present (see Section \ref{SecDisc}).

The derivations of the (classical) spin Hall effect of light present in the physics literature generally start by considering Maxwell's equations and then applying certain approximations or high-frequency asymptotic expansions. A common route is represented by extensions of geometric optics and WKB-type expansions \cite{PhysRevE.70.026605,PhysRevA.92.043805}, sometimes combined with paraxial approximations \cite{SHE-L_original,Bliokh2009}. In this case, a geometric optics ansatz of the form $\psi e^{\ii \omega S}$, where $S$ is a real phase function and $\psi$ is a vector-valued amplitude, is inserted into Maxwell's equations and the resulting equations at each order in $\omega$ are individually set to zero. This recovers the well-known geometric optics results represented by a Hamilton-Jacobi equation for $S$ at the leading order in $\omega$ and a transport equation for the amplitude $\psi$ at the next-to-leading order in $\omega$. The Hamilton-Jacobi equation can be solved by the method of characteristics, leading to the geodesic rays of geometric optics. The transport equation determines the evolution of the shape of the wave packet, as well as the evolution of the polarisation, along the geometric optics rays. In particular, it is convenient to express the part of the transport equation for the polarisation in terms of a Berry connection, which can be integrated to represent the evolution of the polarisation in terms of a Berry phase (see \cite{Emmrich1996,QGT2026} for a geometric definition of transport equations in terms of Berry connections). Within this framework, the spin Hall equations can be obtained by noting that the total phase of the ansatz $\psi e^{\ii \omega S}$ is not only given by the eikonal phase $S$, but that there is also a sub-leading Berry phase contribution $S_B$ that comes from the complex vector amplitude $\psi$. Thus, the total phase function is $S + \omega^{-1} S_B$, and a modified Hamilton-Jacobi equation for it can be derived by combining the Hamilton-Jacobi equation for $S$ and the transport equation for $\psi$. The spin Hall equations represented by the leading-order geodesic rays of geometric optics together with the spin-dependent correction term in \cref{EqDisplacementSpinHall}, are then obtained by applying the method of characteristics to the modified Hamilton-Jacobi equation. We emphasise here that this approach only focuses on the sub-leading correction originating from the dynamics of the polarisation, through the Berry phase and the Berry connection. In particular, there are no contributions related to the shape of the wave packet (e.g., general angular momentum or quadrupole moments). Furthermore, we note that the correction to the geometrical optics rays is captured at the level of `corrected characteristics' instead of the level of the trajectory of the energy centroid. Those aspects will be important for comparison with the results that we present in the following.

Different derivations of the spin Hall equations, based on other methods, have also been given. For example, in Refs. \cite{SHE_original,PhysRevE.74.066610} the authors used semiclassical methods (adapted from quantum mechanics and condensed matter physics \cite{PhysRevA.44.5239,PhysRevB.53.7010,PhysRevB.59.14915}) to describe the dynamics of wave packets with Berry curvature corrections. Here also, similar to the previously discussed approach, the main geometric objects describing the effect are the Berry connection and the associated Berry curvature. Another derivation has also been given in \cite{Duval2006,Duval2007}, where instead of starting from Maxwell's equations the authors introduce a geometrically motivated formulation of photons as classical particles. On a more mathematical level, a spin Hall effect of light has been described in \cite{deNittis2017}, where the authors considered electromagnetic wave packets in photonic crystals (optical materials with periodic structure). This result uses semiclassical methods based on the theory of pseudodifferential operators \cite{teufel2003adiabatic,Panati2003,Teufel2013} to prove Egorov-type theorems for the dynamics of certain observables.

\subsection{Discussion of main results} \label{SecDisc}

In this paper, we present a precise mathematical theory of the propagation of Gaussian beam\footnote{We emphasise here a difference in terminology compared to the optics literature. In our work, a Gaussian beam represents a wave packet of finite energy that, at each time $t$, is localised in space in the sense that it decays exponentially in all three spatial directions away from a reference point. On the other hand, in the optics literature the term Gaussian beam is generally used to describe exact or approximate solutions of a wave equation or a paraxial equation that are exponentially localised only in two spatial direction transverse to the spatial direction of propagation and which have infinite energy \cite{siegman1986lasers,Kiselev2007,Levy:19}.} solutions to Maxwell's equations in an inhomogeneous medium with refractive index $\n$. This theory is based on the Gaussian beam approximation\footnote{In the literature, this is also known as the complex WKB approximation \cite{maslov2012}.} for hyperbolic partial differential equations, as, for example, presented in \cite{ralston1982,Sbierski2013}: profile functions for the Gaussian beam can be freely chosen as part of the initial data, which then determines a one-parameter family of Gaussian beam solutions, where the parameter is the frequency $\omega$ of the beam. As the frequency $\omega$ goes to infinity, the spatial width of  the Gaussian envelope of the beam scales to zero like $\omega^{- \frac{1}{2}}$. We show that for any given time $T>0$, the following ODE system is satisfied by this one-parameter family of Gaussian beam solutions for times $0 \leq t \leq T$ (see \cref{MainThm}):
\begin{subequations} \label{EqODESystemIntro}
\begin{align}
    \dot{\mathbb{X}}^i  &= \frac{1}{\mathbb{E} \n^2}
    \mathbb{P}^i - \frac{1}{\mathbb{E} \n^2} \epsilon^{i j k} \mathbb{J}_j \nabla_k \ln \n - \frac{1}{\mathbb{E}} \dot{\mathbb{Q}}^{i j} \nabla_j \ln \n \nonumber \\
    &\qquad- \frac{1}{\mathbb{E}^2 \n^2} \Big[ \mathbb{P}^i \mathbb{Q}^{j k} \nabla_j \nabla_k \ln \n + 2 \mathbb{P}^j \mathbb{Q}^{i k} (\nabla_j \ln \n) (\nabla_k \ln \n)  \Big] + \mathcal{O}(\omega^{-2}),  \label{EqIntroXDot}\\
    \dot{\mathbb{P}}_i &= \mathbb{E} \nabla_i \ln \n + \bbQ^{j k} \nabla_i \nabla_j \nabla_k \ln \n + \mathcal{O}(\omega^{-2}),    \label{EqIntroPDot}\\
    \dot{\mathbb{J}}_i &= \epsilon_{i j k} \Big( \mathbb{P}^{j} \dot{\mathbb{X}}^k + \bbQ^{j l} \nabla_l \nabla^k \ln \n \Big) + \mathcal{O}(\omega^{-2}),   \label{EqIntroJDot} \\
    \dot{\bbQ}^{i j} &= \ii \omega^{-1} \mathbb{E} \frac{d}{dt} (A^{-1})^{i j} + \mathcal{O}(\omega^{-2}).  \label{EqIntroQDot}  
\end{align}    
\end{subequations}
Here, the constants implicit in the $\mathcal{O}$-notation depend in particular on the chosen profile functions and the time of approximation $T$. In principle, these constants can be computed explicitly. Furthermore, $\mathbb{E}$ denotes the total energy (which is conserved in time), $\mathbb{X}$ the centre of energy, $\mathbb{P}$ the Minkowski linear momentum, $\mathbb{J}$ the Minkowski angular momentum, and $\mathbb{Q}$ the quadrupole moment. Both the angular momentum and the quadrupole moment are defined with respect to the energy centroid $\mathbb{X}$. The definitions are given in \cref{sec:def_average_quantities}. Furthermore, in the above system, $\n$ and its derivatives are all evaluated at $\mathbb{X}(t)$ and $A(t)$ denotes an invertible matrix which is purely imaginary and is uniquely determined for all time by the chosen profile functions. In particular, this closes the ODE system, modulo the error terms $\mathcal{O}(\omega^{-2})$. If one normalises the energy so that it is of order $1$ (with respect to $\omega \to \infty$), then one can show that $\mathbb{P}$ is also of order $1$ (see \cref{RemarkNullGeodesicMotion}) and $\mathbb{J}$ and $\mathbb{Q}$ are of order $\omega^{-1}$ (see \cref{PropJQLateTime}). As a consequence, the error terms are indeed negligible for large $\omega$ and the ODE system determines the evolution of $\mathbb{X}$, $\mathbb{P}$, $\mathbb{J}$, $\mathbb{Q}$ up to and including order $\omega^{-1}$. At leading order $\sim 1$, the solution $(\mathbb{X}(t), \mathbb{P}(t))$ of the above ODE system is determined by geodesic motion with respect to the optical metric, as discussed in \cref{RemarkNullGeodesicMotion}. This recovers the propagation of the energy centroid according to the laws of geometric optics in the limit $\omega \to \infty$, and is represented by the green ray in Figure \ref{fig:SHE}. 

In addition to the precise mathematical theory, our approach directly gives a system involving the energy centroid as a variable (which is accessible to experiments \cite{Hosten2008,Bliokh2008,PradeepChakravarthy:19}) and, moreover, the system captures \emph{all} $\omega^{-1}$ corrections to null geodesic motion: not just the internal spin angular momentum, but indeed the total angular momentum and also the quadrupole moment. The displacement \eqref{EqDisplacementSpinHall} of the `classical' spin Hall effect arises from the second term on the right-hand side of \eqref{EqIntroXDot} if one makes the idealised assumption that the wave only carries internal spin angular momentum (which is proportional to $\mathbb{P}$). This is presented in \cref{PropJQLateTime} and \cref{def:circ_polarisation} with \cref{eq:ID_circ_pol}.

To the best of our knowledge, this work represents the first mathematical theory of the spin Hall effect of light in an inhomogeneous medium which is based on the Gaussian beam approximation. However, for the Schr\"odinger equation the Gaussian beam approximation has already been used to capture subleading effects on the propagation depending on the shape of the wave packet \cite{Ohsawa_2013} as well as the anomalous Hall effect in inhomogeneous periodic media \cite{Watson2017}, which has structural similarities to the spin Hall effect of light. Note that in \cite{Watson2017} an effective particle-\emph{field} system is derived, while our ODE system \eqref{EqODESystemIntro} corresponds to an effective particle system.

\subsection{Overview of the proof}

We work at the level of the electric field $\bE$ and the magnetising field $\bH$ and construct approximate Gaussian beam solutions of the form
\begin{align} 
        &\hat{\bE} = \omega^{\nicefrac{3}{4}} \Re \big[(\be_0 + \omega^{-1} \be_1  + \omega^{-2} \be_{2} )e^{\ii \omega \phi} \big], \qquad\hat{\bH} = \omega^{\nicefrac{3}{4}} \Re \big[(\bh_0 + \omega^{-1} \bh_1  + \omega^{-2} \bh_{2} )e^{\ii \omega \phi} \big].
\end{align}
Here, $\be_j$, $\bh_j$, and $\phi$ are the profile functions of the beam. In contrast to unconstrained hyperbolic PDEs (see, for example, \cite{ralston1982,Sbierski2013}), one cannot just take the induced initial data of the approximate solution as initial data for Maxwell's equations, since in general it does not satisfy the constraint equations. Here, we use Bogovskii's operator \cite{Bogovskii,Galdi} to solve for a compactly supported error term, which we add to the approximate initial data to obtain exact Gaussian beam initial data for Maxwell's equations. Our main theorem is phrased in terms of such compactly supported exact Gaussian beam initial data. 

Using energy estimates, we get quantitative upper bounds on the difference between exact and approximate solutions in a finite time range. This gives us on the one hand the a priori estimate 
\begin{equation} \label{EqAPrioriIntro}
|\mathbb{X}(t) - \underline{\gamma}(t)|_{\R^3} = \mathcal{O}(\omega^{-1}),
\end{equation}
where $\underline{\gamma}$ is the geodesic determined by the optical geometry and $\underline{\gamma}(t)$ is the point of stationary phase for the energy and momentum density of the approximate solution at time $t$. On the other hand, we obtain that the energy centroid  $\hat{\mathbb{X}}$ of the approximate solution is $\mathcal{O}(\omega^{-2})$-close to $\mathbb{X}$ -- and similarly for the linear and angular momentum and the quadrupole moment. In principle $\hat{\mathbb{X}}(t)$ can be computed to order $\omega^{-1}$ directly from the profile functions of the approximate solution (see \cref{prop:approx_average_quantities}), which may be of use for numerical evaluation in concrete situations. However, theoretically it hides the explicit dynamical dependence of the  energy centroid on the first few multipole moments. To derive the closed ODE system \eqref{EqODESystemIntro} for these moments, we proceed as follows: Maxwell's equations give us
\begin{equation}
    \dot{\mathbb{X}}^i = \frac{1}{\mathbb{E}} \int_{\R^3} \n^{-2} \bS^i \, d^3 x,
\end{equation}
where $\bS$ is the momentum density and $\mathbb{P}_i(t) := \int_{\R^3} \bS_i(t,x) \, d^3x $. We now Taylor-expand $\n^{-2}(x)$ around $\mathbb{X}(t)$. The first (constant) term can be pulled out of the integral so that we end up with $\frac{1}{\mathbb{E}} \n^{-2}|_{\mathbb{X}(t)} \mathbb{P}^i(t)$, which is the first term on the right-hand side of \eqref{EqIntroXDot}. The remaining terms in the Taylor expansion yield higher multipole moments. The antisymmetric first moment of $\bS$ gives us the angular momentum term in \cref{EqIntroXDot}, while the symmetric first moment can be related to the time derivative of the quadrupole moment. To treat the second moments, we first use the fact that they are $\mathcal{O}(\omega^{-2})$-close to those of the approximate solution. We then use the structure of the approximate solution together with a stationary phase expansion and \cref{EqAPrioriIntro} to write them, to leading order, in terms of linear momentum and quadrupole moment. The third moments can be shown to be negligible -- again using a stationary phase expansion for the approximate solution. 

The evolution equations for $\bbP$, $\bbJ$, and $\bbQ$ can be dealt with in a similar way.
Finally, the determination of the correct order in $\omega$ of the different moments again follows from the stationary phase expansion.

\subsection{Outline of the paper}

We start with some preliminaries in \cref{SecPrelim}: in \cref{SecConventions} we lay out conventions used throughout this paper, in \cref{SecMaxwell} we recall Maxwell's equations in an inhomogeneous medium and define the notions of total energy, linear momentum, angular momentum, dipole moment, and quadrupole moment. The evolution equations of those quantities are also collated here. In \cref{SecOpticalGeom} the optical geometry is defined and the equations of ray optics (geodesics) recalled. Our main results are stated and discussed in detail in \cref{SecMainResults}. The proof of our main results is spread across \cref{SecApproxSol,SecConstructionID,SecEnergyEst,sec:approx_exact_and_main_results}: in \cref{SecApproxSol} the construction of approximate Gaussian beam solutions is carried out, and in \cref{SecConstructionID} exact Maxwell initial data satisfying the constraint equations is constructed from the induced initial data of the approximate solution by adding a suitable small perturbation. The fundamental energy estimate for Maxwell's equations is recalled in \cref{SecEnergyEst}, which is used to control the error between the exact and approximate Gaussian beam solution. The derivation of our ODE system is then concluded in \cref{sec:approx_exact_and_main_results}. Four appendices are provided: \cref{app:Stationary} recalls the stationary phase approximation for the convenience of the reader, while \cref{AppAdditionalResults,AppDerGB,app:D} provide more details on and auxiliary computations for the Gaussian beam approximation for Maxwell's equations.

\section*{Acknowledgements}
Sam C. Collingbourne and Jan Sbierski acknowledge support through the Royal Society University Research Fellowship URF\textbackslash R1\textbackslash 211216. This research was funded in whole or in part by the Austrian Science Fund (FWF) \href{https://doi.org/10.55776/PIN9589124}{10.55776/PIN9589124}. Jan Sbierski also thanks Sung-Jin Oh for discussions about Bogovskii's operator.

\section{Preliminaries} \label{SecPrelim}

In this section, we review our notation and introduce the basic definitions to be used in the rest of the paper. The starting point for the formulation of our results is represented by Maxwell's equations in an inhomogeneous medium, together with the associated observable quantities, such as total energy and centre of energy, linear and angular momentum, and quadrupole moment. We also review the optical geometry associated with Gordon's optical metric, which serves as a geodesic formulation of ray optics.

\subsection{Notation and conventions} \label{SecConventions}

\begin{itemize}
    \item We work in Minkowski spacetime $\R^{1+3}$ using global Cartesian coordinates $(t, x)$, where the spatial coordinates are denoted as $x = x^i = (x^1, x^2, x^3)$. 
    \item We use lowercase Latin indices  that range from 1 to 3 to denote coordinate components $x^i$, as well as vector components $v^i$. The summation convention $v^i w_i = \sum_{i=1}^3 v^i w_i$ is used for repeated indices. To avoid possible confusion between indices and the imaginary unit, we use the notation $\ii^2 = -1$. We also fix the complex square root by imposing a branch cut along the negative real axis.
    \item We denote spatial partial derivatives by $\nabla_i$ and spacetime partial derivatives by $\partial_\nu$. Indices $\nu, \kappa, \ldots$ denote spacetime indices running from $0$ to $3$.
    \item The electric permittivity $\varepsilon:\R^3\rightarrow \R$ and the magnetic permeability $\mu :\R^3\rightarrow \R$ are positive smooth real scalar functions on $\R^3$ with $0 < c_m \leq \varepsilon \leq C_m$ and $0 < c_m \leq \mu \leq C_m$ for some positive constants $c_m$, $C_m$. The refractive index of the inhomogeneous medium is defined as $\mathfrak{n} = \sqrt{ \varepsilon \mu }$, and  we assume that $\n \geq 1$. 
    \item For $v\in \R^3$, define $(\star v)_{ij}=\epsilon_{ijk}v^k$, where $\epsilon_{ijk}$ is the Levi-Civita symbol. 
    \item Given a spacetime vector or vector field $v$, then $\underline{v}$ denotes the purely spatial part with respect to the standard basis.
    \item Our convention for raising and lowering spatial indices is with respect to the Euclidean metric $\delta$. If a raising or lowering of an index is accompanied by $\sharp$ or $\flat$, then the raising or lowering is with respect to $\underline{g} =\n^2 \delta$. 
    \item Similarly, the norm $| X | := \sqrt{X^iX_i}$ of a vector in $\R^3$ is always with respect to the Euclidean metric. For a complex vector $X \in \C^3$ the norm is defined by $|X| := \sqrt{X^i\overline{X_i}}$. The dot product for real as well as for complex vectors $X,Y$ is defined by $X \cdot Y = X^iY_i$. Note that for a complex vector  $X$ we have $|X|^2 = X \cdot \overline{X}$ and $X \cdot X$ is not necessarily real. 
    \item We denote the closure of a set $K$ by $\cl(K)$.
    \item Given a quantity $\mathcal{Q}$, which in particular depends on $\omega >1$, we employ the notation $\mathcal{Q}=\mathcal{O}(\omega^{-p})$ if there exists a positive constant $C$, whose dependence will be made explicit, such that $|\mathcal{Q}|\leq C\omega^{-p}$.
    \item Given a function $\mathcal{Q}:\R^3\rightarrow \C$, which furthermore depends on $\omega >1$, we employ the notation $\mathcal{Q}(\x)=\mathcal{O}_{L^q(\R^3)}(\omega^{-p})$ if there exists a positive constant $C$, whose dependence will be made explicit, such that $\|\mathcal{Q}\|_{L^q(\R^3)}\leq C\omega^{-p}$.
    \item Let $f:\R^{1+3}\rightarrow \C$ or $f:\R\rightarrow \C$ be a smooth function. We use the notation $\dot{f}=\partial_tf$ and $\ddot{f}=\partial_t^2f$.
    \item Let $\alpha=(\alpha_1,\alpha_2,\alpha_3)\in \N^3$ and $D^{\alpha}$ be defined by $D^{\alpha}:= (\nabla_{x_1})^{\alpha_1}(\nabla_{x_2})^{\alpha_2}(\nabla_{x_3})^{\alpha_3}$.
\end{itemize}

\subsection{Maxwell's equations in an inhomogeneous medium} \label{SecMaxwell}

This section recalls the basic theory of Maxwell's equations in an inhomogeneous medium, which is at rest in Minkowski spacetime with respect to the timelike Killing vector field $\partial_t$. In a general medium, the equations can be written as \cite{born1980}
\begin{subequations}
\begin{align}
    \nabla \cdot \bD &= 4 \pi \rho, \\
    \nabla \cdot \bB &=0, \\
    \nabla \times \bE + \frac{1}{c} \dot{\bB} &= 0, \\
    \nabla \times \bH - \frac{1}{c} \dot{\bD} &= \frac{4 \pi}{c} j,
\end{align}
\end{subequations}
where $\bE$ is the electric field, $\bD$ the displacement field, $\bB$ the magnetic field, and $\bH$ the magnetising field. All these fields are smooth functions from (subsets of) $\R^{1+3}$ to $\R^3$.
We work in units where the speed of light is $c = 1$, and we restrict our attention to the case where there are no charges or currents ($\rho = 0$ and $j = 0$). Furthermore, we consider inhomogeneous media described by the constitutive relations
\begin{equation} \label{eq:constitutive}
    \bD = \varepsilon \bE, \qquad \bB = \mu \bH,
\end{equation}
where $\varepsilon = \varepsilon(x)$ and $\mu = \mu(x)$ are positive smooth real scalar functions on $\R^3$ with $0 < c_m \leq \varepsilon \leq C_m$ and $0 < c_m \leq \mu \leq C_m$ for some positive constants $c_m$, $C_m$. The refractive index is then defined as $\mathfrak{n} = \sqrt{ \varepsilon \mu }$, and to simplify the presentation, we also assume the physically realistic assumption that $\n \geq 1$. In this case, Maxwell's equations can be written as
\begin{subequations} \label{EqMax}
\begin{align}
    \nabla \cdot \bE + \bE \cdot \nabla \ln \varepsilon &= 0, \label{EqMaxCE}\\
    \nabla \cdot \bH + \bH \cdot \nabla \ln \mu &= 0, \label{EqMaxCH} \\
    \nabla \times \bE + \mu \dot{\bH} &= 0,  \label{EqMaxH}\\
    \nabla \times \bH - \varepsilon \dot{\bE} &= 0. \label{EqMaxE}
\end{align}
\end{subequations}
All solutions to \eqref{EqMax} considered in this paper will have compact spatial support.

\subsection{Energy, momentum, and multipole moments}
\label{sec:def_average_quantities}

The energy density of the electromagnetic field is defined as
\begin{equation}
    \mathcal{E} := \frac{1}{2} (\bE \cdot \bD + \bH \cdot \bB) = \frac{1}{2} (\varepsilon \bE \cdot \bE + \mu \bH \cdot \bH),
\end{equation}
and the momentum density (Minkowski's Poynting vector) is defined as
\begin{equation}
    \bS := \bD \times \bB = \n^2 \bE \times \bH .
\end{equation} 
Using Maxwell's equations \eqref{EqMax}, the energy density and momentum density are related through the continuity equation
\begin{equation}
    \dot{\mathcal{E}} + \nabla \cdot \left( \n^{-2} \bS \right) = 0.
\end{equation}
The total energy of the electromagnetic field is defined as
\begin{equation}
    \mathbb{E}(t) := \int_{\R^3} \cE(t,x) \, d^3x.
\end{equation}
If we assume that the electromagnetic field vanishes sufficiently fast at infinity -- for example, if the field is of compact spatial support, as is the case in this paper -- it follows from the continuity equation that the total energy is conserved:
\begin{equation}
    \dot{\mathbb{E}} = 0.
\end{equation}
We define the energy centroid (or also centre of energy) of the electromagnetic field as 
\begin{equation} \label{DefX}
    \mathbb{X}^i(t) := \frac{1}{\mathbb{E}} \int_{\R^3} x^i \cE(t,x) \, d^3x.
\end{equation}
Taking the time derivative of the above equation and using Maxwell's equations \eqref{EqMax} in combination with sufficiently fast decay towards infinity, we obtain
\begin{equation} \label{EqXDot}
    \dot{\mathbb{X}}^i = \frac{1}{\mathbb{E}} \int_{\R^3} \n^{-2} \bS^i \, d^3 x.
\end{equation}
The total linear momentum is defined as
\begin{equation} \label{DefP}
    \mathbb{P}_i(t) := \int_{\R^3} \bS_i(t,x) \, d^3x.
\end{equation}
Taking the time derivative and using Maxwell's equations \eqref{EqMax} in combination with sufficiently fast decay towards infinity we obtain
\begin{equation} \label{EqPDot}
    \dot{\mathbb{P}}_i = \frac{1}{2} \int_{\R^3} \left( \varepsilon \bE \cdot \bE \nabla_i \ln \varepsilon + \mu \bH \cdot \bH \nabla_i \ln \mu \right)  \, d^3x.
\end{equation}
The total angular momentum with respect to the energy centroid $\bbX(t)$ is defined as
\begin{equation} \label{DefJ}
    \mathbb{J}_i(t) := \int_{\R^3} \epsilon_{ijk} r^j(t,x) \bS^k(t,x) \, d^3x,
\end{equation}
where $r^j(t,x) := x^j - \mathbb{X}^j(t)$. Taking the time derivative, we obtain
\begin{align} \label{EqJDot}
    \dot{\mathbb{J}}_i &= \epsilon_{i j k} \mathbb{P}^j \dot{\mathbb{X}}^k + \int_{\R^3} \epsilon_{ijk} r^j \dot{\bS}^k \, d^3x \nonumber \\
    &= \epsilon_{i j k} \mathbb{P}^j \dot{\mathbb{X}}^k + \frac{1}{2} \int_{\R^3} \epsilon_{ijk} r^j \left( \varepsilon \bE \cdot \bE \nabla^k \ln \varepsilon + \mu \bH \cdot \bH \nabla^k \ln \mu \right) \, d^3x .
\end{align}
The dipole and quadrupole moments of the energy density with respect to the energy centroid $\bbX(t)$ are defined as
\begin{subequations}
\begin{align}
    \mathbb{D}^i(t) &:= \int_{\R^3} r^i(t,x) \cE(t,x) \, d^3x , \\
    \mathbb{Q}^{i j}(t) &:= \int_{\R^3} r^i(t,x) r^j(t,x) \cE(t,x) \, d^3x .
\end{align}
\end{subequations}
It follows from the definition \eqref{DefX} of the energy centroid that 
\begin{equation}
    \mathbb{D}^i = \int_{\R^3} x^i \cE \, d^3x\; - \mathbb{X}^i \int_{\R^3}  \cE \, d^3x = 0.
\end{equation}
Using Maxwell's equations \eqref{EqMax} and again sufficiently fast decay at infinity, the time derivative of the quadrupole moment is
\begin{equation} \label{eq:dotQ}
    \dot{\mathbb{Q}}^{i j} = 2 \int_{\R^3} \n^{-2} r^{(i} \bS^{j)}  \, d^3x.
\end{equation}

\subsection{Optical geometry} \label{SecOpticalGeom}

Recall that the optical metric on $\R \times \R^3= \R^{1+3}$ as defined by Gordon \cite{1923AnP...377..421G} is given by $-\n^{-2} dt^2 + \delta$, where $\delta = dx^1 \otimes dx^1 + dx^2 \otimes dx^2 + dx^3 \otimes dx^3$ is the Euclidean metric on $\mathbb{R}^3$. In this paper, it will be useful to work with the conformally rescaled optical metric
\begin{align} \label{eq:optical_metric}
    g:=-dt^2+ \n^2\delta =: -dt \otimes dt + \underline{g}
\end{align}
on $\R^{1+3}$, where we have defined the Riemannian metric $\underline{g} := \n^2 \delta$ on $\R^3$. Note that $g$ is a Lorentzian metric.

We denote the Christoffel symbols with respect to $g$ by $\Gamma$ and those with respect to $\underline{g}$ by $\underline{\Gamma}$. A straightforward investigation then gives $\Gamma^{i}_{jk} = \underline{\Gamma}^i_{jk}$ for $i,j,k \in \{1,2,3\}$ and all other Christoffel symbols of $g$ with at least one time-component vanish. Thus, the geodesic equation on $(\R^{1+3}, g)$ takes the form
\begin{subequations}
\begin{align}
    \ddot{\gamma}^t &= 0, \\
    \ddot{\gamma}^i &= - \underline{\Gamma}^i_{jk} \dot{\gamma}^j \dot{\gamma}^k.
\end{align}
\end{subequations}
Thus, it follows that $t \mapsto \gamma(t) = \big(t, \underline{\gamma}(t)\big)$ is a $g$-null geodesic, if and only if, $t \mapsto \underline{\gamma}(t)$ is a  geodesic in $(\R^3, \underline{g})$ parametrised by $\underline{g}$-arclength.

We now focus on the spatial geometry of $(\R^3, \underline{g})$.
Consider the Hamiltonian $\mathcal{H}(x,p) := \frac{1}{2} \underline{g}^{-1}|_x(p,p)$ on phase space $T^*\R^3 \simeq \R^6$. Then the Hamiltonian equations
\begin{subequations}
\begin{align}
    \dot{x}^i &= \frac{\rd \mathcal{H}}{\rd p_i} = \underline{g}^{ij} p_j, \\
    \dot{p}_i &= - \frac{\rd \mathcal{H}}{\rd x^i} = 2 \mathcal{H} \nabla_i \ln \n ,
\end{align}
\end{subequations}
generate the geodesic flow on phase space. Consider now a geodesic $\underline{\gamma}$ on $(\R^3, \underline{g})$ which is parametrised by $\underline{g}$-arclength. If we set $p_i := \dot{\underline{\gamma}}^\flat_i := \underline{g}_{ij} \dot{\underline{\gamma}}^j$,  we then have $H(\underline{\gamma}, \dot{\underline{\gamma}}^\flat) \equiv \frac{1}{2}$ and thus the equations
\begin{subequations} \label{EqHamiltonianEq}
\begin{align}
    \dot{\underline{\gamma}}^i &=  \underline{g}^{ij} p_j = \frac{1}{\n^2} \delta^{ij} p_j, \\
    \dot{p}_i &=  \nabla_i \ln \n 
\end{align}
\end{subequations}
are satisfied.

\section{Main results} \label{SecMainResults}

There are various localised high-frequency solutions of Maxwell's equations whose energy centroids propagate, to leading order in one over frequency, according to null geodesic motion. In this paper, we restrict ourselves to a special class of such localised high-frequency solutions, namely those arising from Gaussian beam initial data (defined below). For such solutions, we describe the sub-leading correction to the equation of motion of the energy centroid. We begin by defining the class of initial data that we consider in this paper.

\begin{definition} \label{DefGBID}
    $\mathcal{K}$-supported Gaussian beam initial data of order $2$ for Maxwell's equations \eqref{EqMax} is a one-parameter family $\big(\Eu( \cdot; \omega), \Hu( \cdot; \omega)\big) \in C^\infty_0(\mathcal{K}, \R^3) \times C^\infty_0(\mathcal{K}, \R^3)$ with $\omega >1$ of the form
    \begin{subequations} \label{EqDefID}
    \begin{align}
        \Eu(x; \omega) &= \omega^{\nicefrac{3}{4}} \Re \Big\{ \big[\eu_0(x) + \omega^{-1} \eu_1(x)  \big] e^{\ii \omega \phu(x)} \Big\} + \mathcal{O}_{L^2(\R^3)}(\omega^{-2}), \\
        \Hu(x; \omega) &= \omega^{\nicefrac{3}{4}} \Re \Big\{ \big[\hu_0(x) + \omega^{-1} \hu_1(x)  \big] e^{\ii \omega \phu(x)} \Big\} + \mathcal{O}_{L^2(\R^3)}(\omega^{-2}),
        \end{align}
    \end{subequations}
   such that for $\x_0 \in \R^3$ and a pre-compact open neighbourhood $\mathcal{K} \subseteq \R^3$ of $\x_0$ we have
    \begin{enumerate}
        \item \label{ID_def_point1} $\phu \in C^\infty (\R^3, \C)$ with $\Im \phu \geq 0$ and $\Im \phu|_{\x_0} = 0$, $\nabla \Im\phu|_{\x_0} = 0$, $\nabla \Re \phu|_{\x_0} \neq 0$, $ \Im \nabla_i \nabla_j \phu|_{\x_0}$ is a positive definite matrix\footnote{Note that this implies that $\Im \nabla_i \nabla_j \phu|_{\x_0}$  is invertible.}, and $ \nabla \Im  \phu \neq 0$ in $\cl(\mathcal{K})\setminus \{\x_0\}$.
        \item \label{ID_def_point2} $\eu_A$, $\hu_A \in C^\infty_0(\mathcal{K}, \C^3)$ for $A \in \{0,1\}$ with $\eu_0|_{\x_0} \neq 0$ and 
        \begin{subequations} \label{eq:ID_constraints}
        \begin{align}
            D^{\alpha} \Big( \eu_0^k \nabla_k\phu \Big) \Big|_{\x_0} &= 0 \qquad \forall |\alpha|\leq 5, \\
            D^{\alpha} \Big( \eu_1^k \nabla_k\phu - \ii \div\eu_0 -\ii \eu_0^k\nabla_k\ln\eps \Big) \Big|_{\x_0} &= 0 \qquad \forall |\alpha|\leq 3.
        \end{align}
        \end{subequations}       
        Moreover, the first terms in the Taylor expansion of $\hu_A$ around $\x_0$ are related to those of $\eu_A$ and $\phu$ by 
        \begin{subequations} \label{eq:hdefID}
        \begin{align}
            D^{\alpha}\hu^i_{0}\bigg|_{\x_0} &= D^{\alpha}\bigg(-\frac{1}{\mu \dot{\phu}} \epsilon\indices{^{i j}_k} \eu^k_{0}\nabla_j\phu  \bigg)\bigg|_{\x_0} &\forall |\alpha|\leq 5,\label{eq:h0defID}\\ 
            D^{\alpha}\hu^i_{1}\bigg|_{\x_0} &= D^{\alpha} \bigg[ -\frac{1}{\mu \dot{\phu}} \epsilon\indices{^{i j}_k} \eu^k_{1} \nabla_j\phu + \frac{\ii}{\mu \dot{\phu}}\epsilon\indices{^{i j}_k} \nabla_j\eu^k_{0} + \frac{\ii}{\n^2\dot{\phu}^2} \nabla^j\phu \nabla_{j}\hu^i_{0} \nonumber\\
            &\qquad+\frac{\ii}{2\n^2 \dot{\phu}^2} \Big( \hu^i_{0}\Delta\phu-\n^2 \hu^i_{0}\ddot{\phu} + \nabla^{i}\phu \hu^{m}_{0} \nabla_m \ln\n^2 - \hu^{i}_{0} \nabla^{m}\phu \nabla_m\ln\eps \Big) \bigg]\bigg|_{\x_0} &\forall |\alpha|\leq 3, \label{eq:h1defID}   
        \end{align}
        \end{subequations}
        where $\dot{\phu}$ and $\ddot{\phu}$ can be computed naively at $\x_0$, up to order $5$ and $3$ respectively, from the formulas\footnote{Recall that $\sqrt{{\quad}}$ denotes the complex square root (with a branch cut along the negative real axis).}
        \begin{align}
            \dot{\phu}=-\frac{1}{\n}\sqrt{\nabla_i\phu\nabla^i\phu},\qquad \ddot{\phu} =\frac{\nabla^j\phu}{\n \sqrt{\nabla_i\phu\nabla^i\phu}}\nabla_j\left( \frac{\sqrt{\nabla_i\phu\nabla^i\phu}}{\n} \right).
        \end{align}
        \item \label{ID_def_point3} $\Eu(\cdot; \omega)$ and $\Hu(\cdot ; \omega)$ satisfy the constraint equations \eqref{EqMaxCE} and \eqref{EqMaxCH} for each $\omega >1$.
    \end{enumerate}
\end{definition}
A priori it might not be obvious that the class of $\mathcal{K}$-supported Gaussian beam initial data of order $2$ is non-empty. That it is indeed not just non-empty, but a very rich class of initial data is shown in \cref{thm:construction_theorem,ThmConstructionMaxwellID}. The construction of approximate Gaussian beam solutions to hyperbolic PDEs is well known and is given, for example, in \cite{ralston1982}, \cite{Sbierski2013}. Here, we carry out this construction in \cref{thm:construction_theorem} for Maxwell's equations \eqref{EqMax}, which constitute a \emph{constrained} hyperbolic system. As a consequence, it remains to show that one can perturb the compactly supported approximate initial data to obtain compactly supported data which identically satisfies the constraint equations. This is done in \cref{ThmConstructionMaxwellID}.

Given such  $\mathcal{K}$-supported Gaussian beam initial data of order $2$ together with the point $\x_0 \in \R^3$, there are the following associated dynamical structures which naturally enter into the ODE system described in the main \cref{MainThm}. These structures are determined purely by the optical geometry and the Gaussian beam initial data. For this we define the constant 
\begin{equation}
   c_\gamma := \bigg( \frac{1}{\n} \big|\nabla \phu|\bigg)\bigg|_{\x_0} = - \dot{\phu}|_{\x_0} >0\;. 
\end{equation} Firstly, we consider the (real and null) spacetime vector $\rd_t + \frac{\nabla^i \phu}{\n |\nabla\phu|} \rd_i = \rd_t + \frac{1}{c_\gamma} \frac{\nabla^i \phu}{\n^2} \in T_{(0, \x_0)}\R^{1+3}$ and the affinely parametrised $g$-null geodesic $\gamma$ generated by these initial data, which is of the form $t \overset{\gamma}{\mapsto} (t, \underline{\gamma}(t))$. Recall from \cref{SecOpticalGeom} that $\underline{\gamma}$ is a Riemannian geodesic in $(\R^3, \underline{g})$ emanating from $x_0$ with tangent $\frac{\nabla^i \phu}{\n |\nabla\phu|} \rd_i$ that is parametrised by $\underline{g}$-arclength and satisfies \cref{EqHamiltonianEq}. 

Secondly, we then consider the following time-dependent real $3\times 3$-matrices along $\gammu$
\begin{subequations}
\begin{align}
    L_{i j}(t) &= -\frac{1}{\n^2 c_\gamma } \Big( \n^2 \dot{\underline{\gamma}}^i \dot{\gammu}^j  -  \delta^{i j} \Big) \Big|_{\gammu(t)}, \\
    N_{i j}(t) &= -\big( \nabla_i \ln\n \big) \dot{\gammu}^j \Big|_{\gammu(t)},\\
    R_{i j}(t) &= -c_\gamma \Big[ \nabla_i \nabla_j \ln\n - \big(\nabla_i \ln\n \big) \big( \nabla_j \ln\n \big) \Big] \Big|_{\gammu(t)}
\end{align}
\end{subequations}
and we solve the following matrix Riccati equation\footnote{Here, $\cdot$ stands for matrix multiplication and ${}^T$ for matrix transpose.}
\begin{equation}
    \frac{d}{dt}M+M\cdot L\cdot M+N\cdot M+M\cdot N^T+R=0
\end{equation}
with initial data $M_{ij}(0) := \nabla_i \nabla_j \phu|_{\x_0}$. In the proof of \cref{prop:construction_phase_function} it is shown that due to $\Im \nabla_i \nabla_j \phu|_{\x_0}$ being positive definite, this ODE has a unique smooth solution $t \mapsto M(t) \in Mat(3 \times 3; \C)$ for all $t \geq 0$ and 
\begin{equation} \label{EqDefA}
A_{ij}(t) := 2 \ii\cdot  \Im M_{ij}(t)
\end{equation}
is invertible for all $t \geq 0$. The inverse of the matrix $A$ is the dynamical structure that enters into the ODE system below.

The following is the main theorem of this paper:
\begin{theorem} \label{MainThm}
    Let $\mathcal{K}$-supported Gaussian beam initial data of order $2$ as in \cref{DefGBID} be given and consider the corresponding solution $(\bE, \bH)$ to Maxwell's equations \eqref{EqMax}. Construct the null geodesic $\gamma$ and time-dependent purely imaginary and invertible matrix $A_{ij}(t)$ as defined above. Let $T>0$ be given.
    Then for all $0 \leq t \leq T$ the following system of ODEs is satisfied  by the solution $(\bE, \bH)$:
\begin{subequations} \label{EqMainThm}
\begin{align}
    \dot{\mathbb{X}}^i  &= \frac{1}{\mathbb{E} \n^2}
    \mathbb{P}^i - \frac{1}{\mathbb{E} \n^2} \epsilon^{i j k} \mathbb{J}_j \nabla_k \ln \n - \frac{1}{\mathbb{E}} \dot{\mathbb{Q}}^{i j} \nabla_j \ln \n \nonumber \\
    &\qquad- \frac{1}{\mathbb{E}^2 \n^2} \Big[ \mathbb{P}^i \mathbb{Q}^{j k} \nabla_j \nabla_k \ln \n + 2 \mathbb{P}^j \mathbb{Q}^{i k} (\nabla_j \ln \n) (\nabla_k \ln \n)  \Big] + \mathcal{O}(\omega^{-2}),  \label{EqMainThmXDot}\\
    \dot{\mathbb{P}}_i &= \mathbb{E} \nabla_i \ln \n + \bbQ^{j k} \nabla_i \nabla_j \nabla_k \ln \n + \mathcal{O}(\omega^{-2}),    \label{EqMainThmPDot}\\
    \dot{\mathbb{J}}_i &= \epsilon_{i j k} \Big( \mathbb{P}^{j} \dot{\mathbb{X}}^k + \bbQ^{j l} \nabla_l \nabla^k \ln \n \Big) + \mathcal{O}(\omega^{-2}),   \label{EqMainThmJDot} \\
    \dot{\bbQ}^{i j} &= \ii \omega^{-1} \mathbb{E} \frac{d}{dt} (A^{-1})^{i j} + \mathcal{O}(\omega^{-2}).  \label{EqMainThmQDot}  
\end{align}    
\end{subequations}
Here, $\n$ and its derivatives are all evaluated at $\mathbb{X}(t)$.
The constant implicit in the $\mathcal{O}$-notation depends only on the constants implicit in the $\mathcal{O}_{L^2(\R^3)}$-terms in \cref{EqDefID}, the profile functions $\phu$, $\eu_A$, and $\hu_A$ for $A = 0,1$, the functions $\varepsilon$ and $\mu$, and the time of approximation $T$.
\end{theorem}

Note that by definition \eqref{EqDefA}, the matrix $A$ is purely imaginary, so the first term on the right-hand side of \eqref{EqMainThmQDot} is indeed real.

We emphasise that in this theorem, if $T>0$ is fixed, $\omega>1$ has to be chosen large enough for the error terms to be small enough. Our proof relies on the validity of the Gaussian beam approximation. If the electric permittivity $\varepsilon$ and magnetic permeability $\mu$ are given, and if the profile functions $\phu$, $\eu_A$, and $\hu_A$ for $A = 0,1$ of the initial data are known together with a bound on the error terms in \cref{EqDefID}, then an explicit, $T$-dependent bound on the error terms in \cref{EqMainThm} is, in principle, computable from our proof, and thus an estimate on how big $\omega$ has to be chosen for a satisfactory approximation.

To complement \cref{MainThm}, the initial values of the average quantities and multipole moments can be computed directly from the Gaussian beam initial data.

\begin{proposition} \label{PropInitialValueAveragedQuantities}
Consider initial data as in \cref{DefGBID}. Then, the corresponding energy $\bbE$ and initial data for the system of ODEs \eqref{EqMainThm} are
\begin{subequations} \label{eq:initial_quantities}
\begin{align}
    \bbE &= \frac{1}{\sqrt{\det\left( \frac{1}{2 \pi \ii} \Ab  \right)}} \Big( \uu + \omega^{-1} L_1 \uu \Big) \Big|_{\x_0} + \mathcal{O}(\omega^{-2}),  \\ 
    \bbX^i(0) &= \x_0^i + \frac{\ii \omega^{-1}}{\mathbb{E} \sqrt{\det\left( \frac{1}{2 \pi \ii} \Ab  \right)}} (\Ab ^{-1})^{i a} \Big[ \nabla_a \uu  - \ii \uu  (\Ab ^{-1})^{b c} \nabla_a \nabla_b \nabla_c \Im \phu \Big] \Big|_{\x_0} + \mathcal{O}(\omega^{-2}), \label{EqInCondX}\\
    \bbP_i(0) &= \frac{1}{\sqrt{\det\left( \frac{1}{2 \pi \ii} \Ab  \right)}} \Big( \vu _i + \omega^{-1} L_1 \vu _i \Big) \Big|_{\x_0} + \mathcal{O}(\omega^{-2}), \label{EqInCondP}\\
    \bbJ_i (0) &= \epsilon_{i j k} r^j \bbP^k(0) + \frac{\ii \omega^{-1}}{\sqrt{\det\left( \frac{1}{2 \pi \ii} \Ab  \right)}} \epsilon_{i j k} (\Ab ^{-1})^{j a} \Big[ \nabla_a \vu ^k - \ii \vu ^k (\Ab ^{-1})^{b c} \nabla_a \nabla_b \nabla_c \Im \phu \Big] \Big|_{\x_0} + \mathcal{O}(\omega^{-2}), \label{EqInCondJ} \\
    \bbQ^{i j}(0) &= \ii \omega^{-1} \mathbb{E} (\Ab ^{-1})^{i j} + \mathcal{O}(\omega^{-2}), \label{EqInCondQ}
\end{align}
\end{subequations}
where $L_1$ is the differential operator defined in \cref{eq:L_def} with $f(x) = 2\ii \Im \phu$, $\Ab _{i j} = A_{i j}(0) = 2\ii \nabla_i \nabla_j \Im \phu|_{\x_0}$, $r^j = \x_0^j - \bbX^j(0)$ and  
\begin{subequations}
\begin{align}
    \uu  &= \frac{1}{4} \big( \varepsilon \eu_0 \cdot \overline{\eu}_0 + \mu \hu_0 \cdot \overline{\hu}_0 \big) + \frac{\omega^{-1}}{2} \Re \big(\varepsilon \eu_0 \cdot \overline{\eu}_1 + \mu \hu_0 \cdot \overline{\hu}_1 \big)  , \\
    \vu  &= \frac{\n^2}{2} \Re \big( \eu_0 \times \overline{\hu}_0 \big) + \frac{\omega^{-1} \n^2}{2} \Re \big( \eu_0 \times \overline{\hu}_1 + \eu_1 \times \overline{\hu}_0 \big) .
\end{align}
\end{subequations}
\end{proposition}

\begin{proof}
    The proof can be found in~\cref{sec:compute_initial_quantities_GB_ID}.
\end{proof}

\begin{proposition} \label{PropJQLateTime}
Given initial data as in \cref{DefGBID}, the total angular momentum and the quadrupole moment for all $t \in [0, T]$ are
\begin{subequations} \label{eq:J_Q}
\begin{align} 
    \bbJ_i(t) &= \omega^{-1} \frac{\bbE}{\dot{\phu} |_{\x_0}} \bigg[ s - \ii \epsilon_{a b c} \frac{\bbP^a(t)}{|\bbP(t)|} (A^{-1})^{b d} (t) B\indices{_d^c}(t) \bigg] \frac{\bbP_i(t)}{|\bbP(t)|} \nonumber \\
    &\qquad + \ii \omega^{-1} \frac{\bbE}{\dot{\phu} |_{\x_0}} \epsilon_{i a b} \bigg[ \frac{\bbP_c(t)}{|\bbP(t)|} (A^{-1})^{c d}(t) B\indices{_d^a}(t) - \frac{\dot{\phu} |_{\x_0}}{\bbE} |\bbP(t)| (A^{-1})^{a c} \nabla_c \ln \n \bigg] \frac{\bbP^b(t)}{|\bbP(t)|}  + \mathcal{O}(\omega^{-2}), \label{EqValueJ} \\ 
    \bbQ^{i j}(t) &= \ii \omega^{-1} \bbE (A^{-1})^{i j}(t) + \mathcal{O}(\omega^{-2}), \label{EqValueQ}
\end{align}    
\end{subequations}
where $B_{i j}(t) = \Re M_{i j}(t)$ and the constant $s \in [-1, 1]$ is determined by the state of polarisation of the initial data (with $s = \pm 1$ for circular polarisation) as 
\begin{equation} \label{EqS}
    s = \frac{\ii \n \eu_0 \cdot \overline{\hu}_0}{\eps \eu_0 \cdot \overline{\eu}_0} \bigg|_{\x_0} .
\end{equation}
\end{proposition}

\begin{proof}
See Section~\ref{sec:approx_exact_and_main_results}.
\end{proof}

Based on \cref{EqValueJ}, we note that there are several contributions to the total angular momentum carried by the wave packet. The term proportional to $s$ and aligned with the longitudinal direction of $\bbP$ is called spin angular momentum and is determined by the state of polarisation. In other words, this is an intrinsic angular momentum contribution related to the spin-$1$ nature of the electromagnetic field. When combined with \cref{EqMainThmXDot}, the spin angular momentum gives the spin Hall correction term \eqref{EqDisplacementSpinHall} that is commonly discussed in the literature \cite{SOI_review}. The other terms in \cref{EqValueJ} provide additional transverse and longitudinal angular momentum contributions that directly depend on the shape and phase profile of the wave packet through the $A$ and $B$ matrices. We emphasise that these additional terms are not related to any vortex-type structures (such as in the case of Laguerre-Gauss beams) \cite{AM_Light,BLIOKH20151}, but rather to the asymmetry or astigmatism that can be present in the Gaussian profile of the wave packet \cite{PhysRevA.70.013809,Plachenov:23}. In any case, all these additional terms provide contributions to the spin Hall effect that have not been previously discussed in the literature.

\begin{remark}[On the form of the ODE system \eqref{EqMainThm}]
\begin{enumerate}
\item Equation \eqref{EqMainThmXDot} can be inserted into the right-hand side of equation \eqref{EqMainThmJDot} to obtain an ODE system in canonical form.
    \item Equation \eqref{EqMainThmQDot} can be solved directly to give \eqref{EqValueQ}.
     This can be inserted in the first three equations to get a closed system for $\mathbb{X}, \mathbb{P}$, $\mathbb{J}$ alone together with the forcing term $(A^{-1})^{ij}(t)$:
\begin{subequations} \label{EqMainThm2}
\begin{align}
    \dot{\mathbb{X}}^i(t)  &= \frac{1}{\mathbb{E} \n^2} \mathbb{P}^i(t) - \frac{1}{\mathbb{E} \n^2} \epsilon^{i j k} \mathbb{J}_j (t) \nabla_k \ln \n - \ii \omega^{-1} \frac{d}{dt} (A^{-1})^{i j}(t)  \nabla_j \ln \n \nonumber \\
    &\qquad- \frac{\ii \omega^{-1}}{\mathbb{E} \n^2} \Big[ \mathbb{P}^i(t) (A^{-1})^{j k}(t)  \nabla_j \nabla_k \ln \n + 2 \mathbb{P}^j(t) (A^{-1})^{i k}(t)  (\nabla_j \ln \n) (\nabla_k \ln \n) \Big]+ \mathcal{O}(\omega^{-2}), \label{EqNewXDot} \\
    \dot{\mathbb{P}}_i(t) &= \mathbb{E} \nabla_i \ln \n + \ii \omega^{-1}  \mathbb{E} (A^{-1})^{j k}(t) \nabla_i \nabla_j \nabla_k \ln \n + \mathcal{O}(\omega^{-2}),   \label{EqNewPDot} \\
    \dot{\mathbb{J}}_i(t) &= \epsilon_{i j k} \Big[ \mathbb{P}^{j}(t) \dot{\mathbb{X}}^k(t) + \ii \omega^{-1} \mathbb{E} (A^{-1})^{j l}(t)  \nabla_l \nabla^k \ln \n \Big] + \mathcal{O}(\omega^{-2})\,. \label{EqNewJDot}
\end{align}   
\end{subequations}

\end{enumerate}
\end{remark}

\begin{remark}[Null geodesic motion at leading order] \label{RemarkNullGeodesicMotion}
    We investigate the leading order behaviour of the solution to \cref{EqMainThm2}. For $0 \leq t \leq T$  we directly obtain from \cref{EqNewPDot} that 
    \begin{equation} \label{EqAPrioriBoundP}
    \mathbb{P}_i(t) = \mathcal{O}(1).
    \end{equation}
    Inserting \eqref{EqNewXDot} into \eqref{EqNewJDot} and only keeping leading order terms gives
    \begin{equation}
        \dot{\mathbb{J}}_i(t) = \frac{\mathbb{P}^m(t) \mathbb{J}_m(t) }{\mathbb{E} \n^2} \nabla_i \ln \n - \frac{\mathbb{P}^m(t) \nabla_m \ln \n}{\mathbb{E}\n^2} \mathbb{J}_i(t) + \mathcal{O}(\omega^{-1}) .
    \end{equation}
Now, from \cref{EqInCondJ,EqAPrioriBoundP} it follows that $\mathbb{J}_i(t) = \mathcal{O}(\omega^{-1})$ for $0 \leq t \leq T$ (or indeed this follows from \eqref{eq:J_Q}). Putting this information back into \cref{EqNewXDot}, the leading order behaviour of the system \eqref{EqMainThm2} reduces to 
\begin{subequations} \label{Eq_leading_order}
\begin{align}
    \dot{\mathbb{X}}^i(t)  &= \frac{1}{\mathbb{E} \n^2}
    \mathbb{P}^i(t)  + \mathcal{O}(\omega^{-1}), \label{EqXLeadingOrder} \\
    \dot{\mathbb{P}}_i(t) &= \mathbb{E} \nabla_i \ln \n + \mathcal{O}(\omega^{-1}).  \label{EqPLeadingOrder}
\end{align} 
\end{subequations}
Now, with 
\begin{equation} \label{EqSolGuessed}\mathbb{X}^i(t) = \underline{\gamma}^i(t) \qquad  \textnormal{ and } \qquad  \mathbb{P}_i(t) = \mathbb{E} \cdot p_i(t) = \mathbb{E} \cdot \underline{\dot{\gamma}}^\flat_i ,
\end{equation}
we see that \cref{Eq_leading_order} is equivalent to leading order to \cref{EqHamiltonianEq} -- and by \eqref{EqInCondX}, \eqref{EqInCondP} the initial values agree. Hence, due to the uniqueness of the initial value problem, the solution $\big(\mathbb{X}(t), \mathbb{P}(t) \big)$ is given by \eqref{EqSolGuessed} to leading order. We have thus shown that, to leading order, \cref{EqNewJDot} for the angular momentum drops out and the evolution of $\big(\mathbb{X}(t), \mathbb{P}(t)\big)$ is determined by null geodesic motion. 

The content of \cref{MainThm}, or of \cref{EqMainThm2}, is to give an ODE system which determines the correction to null geodesic motion for the energy centroid to the first subleading order in $\omega^{-1}$. This deviation from null geodesic motion depends not only on the initial position and initial momentum, but also on the initial angular momentum and the initial quadrupole moment. However, it can still be described in terms of \emph{ordinary} differential equations as a particle system!
\end{remark}

Finally, we discuss how the ODE system \eqref{EqMainThm} (or \eqref{EqMainThm2}) can be used to describe the spin Hall effect of light in an inhomogeneous medium. For this we consider a point $\x_0 \in \R^3$ that lies outside the inhomogeneous medium or at least in  a region where the medium is nearly homogeneous in the sense that $D^\alpha \varepsilon |_{\x_0} = 0 = D^\alpha \mu |_{\x_0}$ for all $1 \leq |\alpha| \leq 3$. We now prepare left and right circularly polarised Gaussian beam initial data in the vicinity of this point which, to leading order, is `identical up to polarisation'. Mathematically, this is captured as follows:

\begin{definition}[A class of circularly polarised initial data] \label{def:circ_polarisation}
Assume that the medium is nearly homogeneous near $\x_0 \in \R^3$ in the sense that  $D^\alpha \varepsilon |_{\x_0} = 0 = D^\alpha \mu |_{\x_0}$ for all $1 \leq |\alpha| \leq 3$. 

Then $\mathcal{K}$-supported Gaussian beam initial data of order $2$ as in Definition \ref{DefGBID} is called circularly polarised if, for $\Big( \frac{\nabla \phu}{|\nabla \phu|} \Big|_{\x_0}, X, Y \Big)$ being a positively oriented orthonormal frame at $\x_0$, where $X$ and $Y$ are real vectors, we have  
\begin{enumerate}
    \item $\eu_0|_{\x_0} = \frac{\mathfrak{a}}{\sqrt{2}} (X - \ii s Y)$, where $s = \pm 1$ and $\mathfrak{a}$ is a strictly positive real constant.
    \item $\Re \nabla_a \nabla_b \phu |_{\x_0} = 0$ (or equivalently $\nabla_a \nabla_b \phu |_{\x_0} = \frac{1}{2} \Ab _{a b}$) and $D^\alpha \phu |_{\x_0} = 0$ for all $3 \leq |\alpha| \leq 4$.
    \item \begin{subequations} \label{EqIDCircCon}
    \begin{align}
    \nabla_a \eu_0^i \big|_{\x_0} &= - \frac{1}{|\nabla \phu|^2} \Big( \eu_0^b \nabla_a \nabla_b \phu \Big) \nabla^i \phu \Big|_{\x_0},\\
    \nabla_a \nabla_b \eu_0^i \big|_{\x_0} &= - \frac{2}{|\nabla \phu|^2} \Big( \nabla_{(a} \eu_0^c \nabla_{b)} \nabla_c \phu \Big) \nabla^i \phu \Big|_{\x_0},\\
    \eu_1^i \big|_{\x_0} &= \frac{\ii }{|\nabla \phu|^2} \div \eu_0 \nabla^i \phu \Big|_{\x_0}.
    \end{align}
\end{subequations}
    \end{enumerate}
\end{definition}
In the above definition, the sign of $s$ defines the state of circular polarisation and agrees with the value of $s$ from \eqref{EqS}. 

The existence of such circularly polarised initial data follows again from Theorems \ref{thm:construction_theorem} and \ref{ThmConstructionMaxwellID}: in Theorem \ref{thm:construction_theorem} the above initial conditions on $D^\alpha \phu|_{\x_0}$ for $2 \leq |\alpha| \leq 4$ can be freely specified and the value of $\eu_0|_{\x_0}$ can also be freely prescribed. Furthermore, the components of $\nabla_a \eu_0^i |_{\x_0}$, $\nabla_a \nabla_b \eu_0^i |_{\x_0}$, and $\eu_1^i |_{\x_0}$ that are not constrained by \cref{eq:ID_constraints} are set to zero, which directly gives \eqref{EqIDCircCon}. In other words, in our definition we capture that, in particular, $\eu_0$ changes as little as possible compared to its value at $\x_0$. Broader classes of circularly polarised Gaussian beam initial data may be defined. The advantage of the above definition is that the initial data for the ODE system \eqref{EqMainThm} can be relatively easily computed (see \cref{Appendix:ComputePCircID} for the computation of $\bbP_i(0)$, which is more involved):
\begin{subequations} \label{eq:ID_circ_pol}
\begin{align}
    \bbE &= \frac{\varepsilon \mathfrak{a}^2}{ 2 \sqrt{\det\left( \frac{1}{2 \pi \ii} \Ab  \right)}} \bigg[ 1 - \frac{\ii \omega^{-1}}{8 |\nabla \phu|^2} \Ab _{i j} \big(X^i X^j + Y^i Y^j \big)  \bigg] \bigg|_{\x_0} + \mathcal{O}(\omega^{-2}),  \\ 
    \bbX^i(0) &= \x_0^i + \mathcal{O}(\omega^{-2}),\\
    \bbP_i(0) &= \frac{\varepsilon \n \mathfrak{a}^2}{2\sqrt{\det\left( \frac{1}{2 \pi \ii} \Ab  \right)}} \bigg[ \frac{\nabla_i \phu}{|\nabla \phu|} + \frac{\ii \omega^{-1}}{4 |\nabla \phu|^3} \big( X_i X_j + Y_i Y_j \big) \Ab ^{j k} \nabla_k \phu \bigg] \bigg|_{\x_0} + \mathcal{O}(\omega^{-2}), \\
    \bbJ_i (0) &=  \frac{\omega^{-1} s \varepsilon \n \mathfrak{a}^2}{2\sqrt{\det\left( \frac{1}{2 \pi \ii} \Ab  \right)}} \frac{\nabla_i \phu}{|\nabla \phu|^2} \bigg|_{\x_0} + \mathcal{O}(\omega^{-2}), \\
    \bbQ^{i j}(0) &= \ii \omega^{-1} \mathbb{E} (\Ab ^{-1})^{i j} + \mathcal{O}(\omega^{-2}).
\end{align}
\end{subequations}
Note that only the initial angular momentum $\bbJ_i(0)$ depends on $s$, while the other quantities $\bbE$, $\bbX^i(0)$, $\bbP_i(0)$ and $\bbQ^{i j}(0)$ are independent of $s$. 

Now assume that the two circularly polarised Gaussian beams, one with $s=+1$, the other with $s=-1$, propagate into the inhomogeneous medium where in particular $\nabla \ln \n \neq 0$. Since the two angular momenta $\mathbb{J}(t)$ are of order $\omega^{-1}$ and different, it follows from the second term on the right-hand side of \eqref{EqNewXDot} that the two trajectories of the energy centroids in general differ at fixed time $t$ by an amount of order $\omega^{-1}$. Given a particular inhomogeneous medium described by functions $\varepsilon$ and $\mu$ one may solve the ODE system \eqref{EqMainThm} or \eqref{EqMainThm2} to obtain the precise description of the two trajectories.

\section{The approximate solutions} \label{SecApproxSol}

In this section, we use the Gaussian beam approximation \cite{ralston1982,Sbierski2013} to construct high-frequency approximate solutions $(\Eh, \Hh)$ for Maxwell's equations~\eqref{EqMax}. For clarity, and because the geometric optics approximation is more widely used and appears in a much broader body of work, we begin by briefly contrasting it with the Gaussian beam approach. A brief historical account of the Gaussian beam approximation can be found in \cite{Sbierski2013} and references therein. 

The geometric optics and Gaussian beam approximations for Maxwell's equations~\eqref{EqMax} both start from a highly oscillatory ansatz of the form
\begin{align}
    \mathfrak{E}:=\omega^{\nicefrac{3}{4}}\sum_{A=0}^{N-1} \omega^{-A} \e_{A}e^{\ii \omega \phi},\qquad \mathfrak{H}:=\omega^{\nicefrac{3}{4}}\sum_{A=0}^{N-1} \omega^{-A} \h_{A}e^{\ii \omega \phi},\qquad \Eh:= \Re(\mathfrak{E}),\qquad \Hh:= \Re(\mathfrak{H}), \label{eq:series_approx}
\end{align}
where $N\in \N$, $A=0,...,N-1$, $(\e_{A},\h_{A})\in C^{\infty}(\R^4,\C)\times  C^{\infty}(\R^4,\C)$, and $\omega>1$ is large. The function $\phi$ is called the eikonal function and is where the first difference between the approximations lies:
$\phi$ is a smooth real-valued spacetime function in the geometric optics approximation and in the Gaussian beam approximation $\phi$ is a smooth complex-valued spacetime function with the requirement that along a chosen curve $\gamma$ (see below for restrictions):
\begin{enumerate}
     \item $\phi|_{\upgamma}$ and $\nabla_a\phi|_{\upgamma}$ are real-valued,
    \item $\Im\phi$ is chosen so that $\Im \nabla_a \nabla_b \phi \big|_{\gamma}$ is positive-definite.
\end{enumerate}
This means that $ |\Eh(x)|$ and $ |\Hh(x)|$ resemble Gaussian distributions centred on $\gamma$. If we then cut-off with a smooth function, we obtain a localised beam around the curve $\gamma$. 

In both cases, we wish to build approximate solutions to~\eqref{EqMax} with~\eqref{eq:series_approx}. More precisely, if we define
\begin{subequations} \label{eq:GB_approx}
\begin{align}
    C &:= \omega^{-\nicefrac{3}{4}}e^{-\ii \omega \phi} \Big( \div \mathfrak{E} + \mathfrak{E}^i\nabla_i \ln \varepsilon \Big), \label{eq:C_GB_approx} \\
    K  &:=  \omega^{-\nicefrac{3}{4}}e^{-\ii \omega \phi} \Big( \div \mathfrak{H} + \mathfrak{H}^i\nabla_i \ln \mu \Big), \label{eq:K_GB_approx} \\
    \bF &:= \omega^{-\nicefrac{3}{4}}e^{-\ii \omega \phi} \Big( \nabla\times\mathfrak{E} + \mu \dot{\mathfrak{H}} \Big), \label{eq:F_GB_approx} \\
    \bG &:= \omega^{-\nicefrac{3}{4}}e^{-\ii \omega \phi} \Big( \nabla\times\mathfrak{H} - \eps \dot{\mathfrak{E}} \Big), \label{eq:G_GB_approx}
\end{align}   
\end{subequations}
we want 
\begin{subequations}\label{eq:approximate_solution_PDE}
\begin{align}
\omega^{\nicefrac{3}{4}} \bC e^{i\omega\phi}&=\mathcal{O}_{L^2(\R^3)}(\omega^{-M}), \quad \omega^{\nicefrac{3}{4}} \bK e^{i\omega\phi}=\mathcal{O}_{L^2(\R^3)}(\omega^{-M}),\\
\omega^{\nicefrac{3}{4}} \bF e^{i\omega\phi}&=\mathcal{O}_{L^2(\R^3)}(\omega^{-M}),\quad  \omega^{\nicefrac{3}{4}} \bG e^{i\omega\phi}=\mathcal{O}_{L^2(\R^3)}(\omega^{-M}),
\end{align}
\end{subequations}
for some $M\in \N$ and for all $0\leq t\leq T<\infty$. For later use, we also introduce here the following notation:
\begin{align}\label{eq:FGhat}
\hat{\mathscr{F}}:=\Re(\omega^{\nicefrac{3}{4}} \bF e^{i\omega\phi}),\quad  \hat{\mathscr{G}}:=\Re(\omega^{\nicefrac{3}{4}} \bG e^{i\omega\phi}).
\end{align}

The requirement set by \cref{eq:approximate_solution_PDE} yields a sequence of equations for the coefficients $(\e_i,\h_i)$. In particular, we have the following proposition:
\begin{proposition}
The quantities $\bC$, $\bK$, $\bF$ and $\bG$ have the following expansions in terms of $\omega$:
    \begin{align} \label{eq:GBexpMaxwell}
        C=\sum_{A=0}^{N-1}\omega^{1-A}C_A,\quad\
        K=\sum_{A=0}^{N-1}\omega^{1-A}K_A,\quad \bF=\sum_{A=0}^{N-1}\omega^{1-A}\bF_A,\quad
        \bG=\sum_{A=0}^{N-1}\omega^{1-A}\bG_A, 
    \end{align}
    where, for $A\in \N$,
    \begin{subequations}\label{eq:CA--GA}
    \begin{align}
        C_A&:= \ii \e_{A}^i \nabla_i\phi+\div \e_{A-1}+\e_{A-1}^i\nabla_i\ln\eps, \\
        K_A&:= \ii \h_{A}^i \nabla_i\phi+\div \h_{A-1}+\h_{A-1}^i\nabla_i\ln\mu, \\
        \bF_A&:= \ii \Big( {\epsilon^{ij}}_k\e^k_{A}\nabla_j\phi + \mu\dot{\phi} \h^i_{A} \Big)+{\epsilon^{ij}}_k\nabla_j\e^k_{A-1}+\mu\dot{\h}^i_{A-1} \label{eq:F_A}, \\
        \bG_A&:= \ii \Big( {\epsilon^{ij}}_k\h^k_{A}\nabla_j\phi-\eps\dot{\phi} \e^i_{A} \Big)+{\epsilon^{ij}}_k\nabla_j\h^k_{A-1}-\eps\dot{\e}^i_{A-1},
    \end{align}   
    \end{subequations}
with $\e_{-1}=0=\h_{-1}$.
\end{proposition}
\begin{proof}
We now expand \cref{eq:GB_approx}. This yields: 
\begin{subequations}
\begin{align}
    \nabla_i \mathfrak{E}^i + \mathfrak{E}^i\nabla_i \ln \eps &= \bigg[ \ii \omega \e_0^i\nabla_i\phi +\sum_{A\geq 0} \omega^{-A} \Big(\div \e_{A} + \ii \e_{A+1}^i \nabla_i\phi+\e_{A}^i\nabla_i\ln\eps \Big) \bigg] e^{\ii\omega \phi},\\
    \nabla_i \mathfrak{H}^i+ \mathfrak{H}^i \nabla_i \ln \mu &= \bigg[ \ii\omega \h_0^i\nabla_i\phi +\sum_{A\geq 0} \omega^{-A} \Big( \div \h_{A}+ \ii\h_{A+1}^i \nabla_i\phi+\h_{A}^i\nabla_i\ln\mu \Big) \bigg]e^{\ii\omega \phi},\\
    {\epsilon^{ij}}_k \nabla_j \mathfrak{E}^k + \mu \dot{\mathfrak{H}}^i &= \bigg\{ \sum_{A\geq 0} \omega^{-A} \Big[{\epsilon^{ij}}_k \Big(\nabla_j\e^k_{A} + \ii\e^k_{A+1}\nabla_j\phi \Big) +\mu\dot{\h}^i_{A}+\ii\mu\dot{\phi} \h^i_{A+1}\Big] + \ii\omega\Big({\epsilon^{ij}}_k\e^k_{0}\nabla_j\phi +\mu\dot{\phi} \h_{0} \Big) \bigg\}e^{\ii\omega \phi},\\
    {\epsilon^{ij}}_k \nabla_j \mathfrak{H}^k -\eps \dot{\mathfrak{E}}^i &= \bigg\{\sum_{A\geq 0} \omega^{-A} \Big[{\epsilon^{ij}}_k \Big(\nabla_j\h^k_{A}+\ii \h^k_{A+1}\nabla_j\phi \Big)-\eps\dot{\e}^i_{A}-\ii\eps\dot{\phi} \e^i_{A+1}\Big] + \ii\omega\Big({\epsilon^{ij}}_k\h^k_{0}\nabla_j\phi-\eps\dot{\phi} \e^i_{0} \Big)\bigg\}e^{\ii\omega \phi}.
\end{align}
\end{subequations}
The $\mathcal{O}(\omega)$ contributions are
\begin{subequations}
\begin{align}
    \bC_0&=\e_0^i\nabla_i\phi,\\
    \bK_0&=\h_0^i\nabla_i\phi,\\
    \bF_0^i&= {\epsilon^{ij}}_k\e^k_{0}\nabla_j\phi+\mu \dot{\phi} \h^i_{0}, \\
    \bG_0^i&={\epsilon^{ij}}_k\h^k_{0}\nabla_j\phi-\eps \dot{\phi} \e^i_{0}.
\end{align}
\end{subequations}
The $\mathcal{O}(\omega^{-A})$ for $A\geq 0$ are
\begin{subequations} \label{EqExprCKFG}
\begin{align}
    \bC_{A+1}&=\div \e_{A}+\ii\e_{A+1}^i \nabla_i\phi+\e_{A}^i\nabla_i\ln\eps,\\
    \bK_{A+1}&=\div \h_{A}+\ii\h_{A+1}^i \nabla_i\phi+\h_{A}^i\nabla_i\ln\mu,\\
\bF_{A+1}^i&={\epsilon^{ij}}_k \Big( \nabla_j\e^k_{A}+\ii\e^k_{A+1}\nabla_j\phi \Big)+\mu\dot{\h}^i_{A}+\ii\mu \dot{\phi} \h^i_{A+1}, \\
\bG_{A+1}^i&={\epsilon^{ij}}_k \Big( \nabla_j\h^k_{A}+\ii\h^k_{A+1}\nabla_j\phi \Big) -\eps\dot{\e}^i_{A}-\ii\eps \dot{\phi} \e^i_{A+1}.
\end{align}
\end{subequations}
Using $\e_{-1}=0=\h_{-1}$, these can be written in a unified form of the statement.
\end{proof}

In achieving \cref{eq:approximate_solution_PDE} from \cref{eq:CA--GA} lies the next difference between the two approaches: in the geometric optics approximation we require that, for all $0 \leq A \leq M+1$,
\begin{align}\label{eq:mix-Max}
    \bC_A\equiv 0,\quad \bK_A\equiv 0,\quad \bF_A\equiv 0,\quad\bG_A\equiv 0,
\end{align}
in spacetime to achieve~\eqref{eq:approximate_solution_PDE}. In the Gaussian beam approximation one can show that for~\eqref{eq:approximate_solution_PDE} to hold, it suffices to require that, for all $0 \leq A \leq M+1$, $(\bC_A,\bK_A,\bF_A,\bG_A)$ to vanish on the curve $\gamma$ to some order $S$ (dependent on $A$),\footnote{See already Proposition~\ref{prop:order_2_from spacetime_functions} in the context of Maxwell's equations.} i.e. for all $0 \leq A \leq M+1$,
\begin{align}\label{eq:mix-Max-gamma}
    D^{\alpha}\bC_A|_{\gamma}= 0,\quad D^{\alpha}\bK_A|_{\gamma}= 0,\quad D^{\alpha}\bF_A|_{\gamma}= 0,\quad D^{\alpha}\bG_A|_{\gamma}= 0,\quad \forall |\alpha|\leq S.
\end{align}

Going back to the geometric optics approximation for Maxwell's equations, instead of studying~\cref{eq:mix-Max}, one can study the eikonal equation for $\phi$
\begin{align}\label{eq:eikonalspacetime}
    \nabla\phi \cdot \nabla \phi-\n^2\dot{\phi}^2=0
\end{align}
in combination with transport and constraint equations for $\e_A$:
\begin{subequations}\label{eq:transporteAspacetime}
\begin{align}
    \Big(\partial_t-\frac{\nabla^m\phi}{\n^2\dot{\phi}}\nabla_m\Big) \e^n_{A} &= \frac{1}{2\n^2\dot{\phi}} \bigg[ \e^n_{A} \Big( \Delta\phi - \n^2 \ddot{\phi} \Big) - \ii \Big( \Delta\e^n_{A-1}  - \n^2 \ddot{\e}^n_{A-1} \Big) - \ii \e_{A-1}^m\nabla_n\nabla_m\ln\eps \nonumber\\
    &\qquad+ \Big( \e^{m}_{A}\nabla^{n}\phi - \ii \nabla^n\e^m_{A-1}\Big) \nabla_m \ln \n^2 - \Big( \e^{n}_{A}\nabla^{m}\phi- \ii \nabla^m\e^n_{A-1}\Big) \nabla_m\ln\mu \bigg],\\
    0&=\ii \e_{A}^i \nabla_i\phi+\div \e_{A-1}+\e_{A-1}^i\nabla_i\ln\eps. 
\end{align}   
\end{subequations}
In this case, one \textit{defines} $\h_A$ by the requirement that $\bF_A=0$. We can deduce~\cref{eq:eikonalspacetime,eq:transporteAspacetime} from~\cref{lem:decompGF}. This is the content of~\cref{prop:2sttranstoMax} in the context of the Gaussian beam approximation. However,~\cref{prop:2sttranstoMax} can be adapted straightforwardly to the geometric optics setting.\footnote{Note that~\cref{eq:eikonalspacetime,eq:transporteAspacetime} can also be obtained by plugging the ansatz for $\Eh$ in~\cref{eq:series_approx} into the wave equation for $\bE$, $\nabla( \bE\cdot \nabla\ln \eps)+\nabla \ln\mu \times (\nabla\times \bE)+\Delta \bE-\n^2\partial_t^2 \bE = 0$, which follows from~\cref{EqMax}.} 

In the Gaussian beam approximation we require the same equations to hold on $\gamma$ to some degree and prescribe that $\frac{\nabla^m\phi}{\n^2\dot{\phi}}=-\dot{\gamma}^m$ to obtain ODEs for $\e_A$, where we assume for simplicity that $\gamma$ can be parametrised by $t$ as $(t,\gammu(t))$. Again, one eliminates $\h_A$ (to some degree) on $\gamma$ by the requirement that $\bF_A|_{\gamma}=0$ (to some degree). More precisely,
\begin{definition}\label{def:EikonalandTransporteA}
We say that $\phi:\R^4\rightarrow \C$ satisfies the Eikonal equation on $\gamma$ to degree $j_{\phi}$ if
\begin{align} \label{eq:Eikonal}
    D^{\alpha} \Big( \nabla\phi \cdot \nabla \phi-\n^2\dot{\phi}^2 \Big) \Big|_{\gamma}=0 \qquad \forall |\alpha|\leq j_{\phi}.
\end{align}
We say that $\e_A$ satisfies the $\e_A$-transport equation along $\gamma$ to degree $j_A$ if
\begin{align}  
    \Big( \partial_t+\dot{\gamma}^m\nabla_m \Big)D^{\alpha}\e^n_{A}  \Big|_{\gamma} &= D^{\alpha} \bigg\{ \frac{1}{2\n^2\dot{\phi}} \bigg[ \e^n_{A} \Big( \Delta\phi - \n^2 \ddot{\phi} \Big) - \ii \Big( \Delta\e^n_{A-1}  - \n^2 \ddot{\e}^n_{A-1} \Big) - \ii \e_{A-1}^m\nabla_n\nabla_m\ln\eps \nonumber\\
    &\qquad+ \Big( \e^{m}_{A}\nabla^{n}\phi - \ii \nabla^n\e^m_{A-1}\Big) \nabla_m \ln \n^2 - \Big( \e^{n}_{A}\nabla^{m}\phi- \ii \nabla^m\e^n_{A-1}\Big) \nabla_m\ln\mu \bigg] \bigg\}\bigg|_{\gamma}\nonumber\nonumber\\
    &\qquad+\sum_{0<\beta\leq \alpha}\binom{\alpha}{\beta} D^{\beta} \bigg(\frac{\nabla^m\phi}{\n^2\dot{\phi}} \bigg) \nabla_m D^{\alpha-\beta}\e_A^n \bigg|_{\gamma}\qquad \forall |\alpha|\leq j_A.\label{eq:eAtransport}
    \end{align}
We say that $\e_A$ satisfies the $\e_A$-constraint along $\gamma$ to degree $c_A$ if 
\begin{align} \label{eq:eAconstraint}
    D^{\alpha}\bC_A \big|_{\gamma}=0 \qquad \forall |\alpha|\leq c_A.
\end{align}
\end{definition}

Prescribing that $\frac{\nabla^m\phi}{\n^2\dot{\phi}}=-\dot{\gamma}^m$ to obtain ODEs along $\gamma$ restricts the curve along which one can perform the construction. Requiring that the Eikonal equation holds to degree $0$ means
\begin{align}
    0= \Big( \n^2\dot{\phi}^2-\nabla\phi\cdot\nabla\phi \Big) \Big|_\gamma =\n^2\dot{\phi}^2 \Big(1-\n^2|\dot{\gamma}|^2 \Big) \Big|_\gamma \implies g(\dot{\gamma},\dot{\gamma})=0,
\end{align}
where we use $\dot{\gamma}^t=1$. In other words, $\gamma$ must be a null curve. Requiring that the eikonal equation holds to degree $1$ imposes
\begin{align}
\dot{\gamma}^\nu\partial_{\nu}(\nabla_m\phi)+\dot{\phi}\nabla_m\ln\n \big|_\gamma=0,
\end{align}
which combined with
\begin{align}
    \dot{\gamma}^t=1,\quad\dot{\gamma}^m=-\frac{\nabla^m\phi}{\n^2\dot{\phi}}\bigg|_\gamma
\end{align}
constitutes the equations of geodesic flow on the cotangent bundle. So $\gamma$ must be a null geodesic in this construction.

An advantage of the Gaussian beam approximation is that it does not break down at caustics. The simple ansatz~\eqref{eq:series_approx} remains a valid approximation for all finite time $T$ provided that $\omega$ is chosen sufficiently large. This is in contrast to geometric optics approximation, which breaks down at
caustics. This means that the time $T$, up to which one has good
control over the solution, cannot be taken arbitrarily large by increasing $\omega$. The formation of caustics is not a death sentence for the method, since one can extend the approximate solution through the caustics with Maslov's canonical operator. However, the solution no longer has the simple form~\eqref{eq:series_approx}.

\subsection{Construction of the Gaussian beam approximation}

In this section, we study and construct approximate solutions to Maxwell's equations~\eqref{EqMax} in an inhomogeneous medium. We start with a preparatory lemma that allows us to show that for \cref{eq:approximate_solution_PDE} to hold, it suffices to require each $(\bC_A,\bK_A,\bF_A,\bG_A)$ to vanish on the curve $\gamma$ to some order $S$.

\begin{lemma}\label{lem:degreetoorder}
    Let $\gamma$ be a curve in $\R^{1 + n}$ parametrised by $t\in[0,T]$ such that $\gamma(t) = (t, \gammu(t))$, and $f\in C^{\infty}_c([0,T]\times \R^{n},\C)$ and vanishes along $\gamma$ (up) to (and including) degree $S$:
    \begin{align}
        D^{\alpha}f \big|_{\gamma}=0\quad\forall |\alpha|\leq S.
    \end{align}
    Let $c>0$ be a constant. Then, we have
    \begin{align}
        \sup_{t\in[0,T]}\int_{ \R^n}|f(t, x)|^2e^{-\omega c|x-\gammu(t)|^2} \, d^n x \leq \frac{C[f,T]}{\omega^{S+\frac{n+2}{2}}},
    \end{align}
    where $C$ is a constant depending on $f$ and $T$.
\end{lemma}
\begin{proof}
    Since $f$ is compactly supported and vanishes along $\gamma$, for each $t$,
    \begin{align}
       | f(t,\x)|\leq C[f,T]|\x-\gammu(t)|^{S+1}.
    \end{align}
    This gives
    \begin{align}
        \int_{\R^n}|f(t,\x)|^2e^{-\omega c|\x-\gammu(t)|^2}d^n x \leq C[f,T]\int_{\R^n}|\x-\gammu(t)|^{2(S+1)}e^{-\omega c|x-\gammu(t)|^2}d^n x.
    \end{align}
    We now change variables and define $\hat{\x}=\sqrt{\omega}[\x-\gammu]$. This gives 
    \begin{align}
       \int_{ \R^n}|f(t, \x)|^2e^{-\omega c|\x-\gammu(t)|^2}d^n x \leq C[f,T]\int_{\R^n}\omega^{-S-1}|\hat{\x}|^{2(S+1)}e^{- c|\hat{\x}|^2}\frac{1}{\omega^{\frac{n}{2}}}d^n\hat{x}.
    \end{align}
\end{proof}

\begin{proposition}\label{prop:order_2_from spacetime_functions}
Let $\Eh$ and $\Hh$ as in \cref{eq:series_approx} be given with $N=3$ and let $\gamma$ be a curve in $\R^{1 + n}$ parametrised by $t\in[0,T]$ such that $\gamma(t) = (t, \gammu(t))$. 
Suppose further that
\begin{enumerate}
    \item for $|\alpha|\leq 1$, $D^{\alpha}\phi|_{\gamma}$ are real-valued,
    \item $\e_0|_{\gamma(0)}\neq0$ and for some $\alpha$, with $|\alpha|=1$, $D^{\alpha}\phi|_{\gamma(0)}\neq 0$,
    \item for $|\alpha|=2$, $\Im(D^{\alpha}\phi|_{\gamma})$ is positive-definite,
    \item $\bC_0$, $\bK_0$, $\bF_0$ and $\bG_0$,  vanish on $\gamma$ to degree $5$, \label{cond4}
    \item $\bC_1$, $\bK_1$, $\bF_1$ and $\bG_1$ vanish on $\gamma$ to degree $3$, \label{cond5}
    \item $\bC_2$, $\bK_2$, $\bF_2$ and $\bG_2$ vanish on $\gamma$ to degree $1$. \label{cond6}
\end{enumerate}
 Then, given a $\rho>0$, there exists a smooth function $\chi_\rho : [0,T] \times \R^3 \to [0,1]$ with $\mathrm{supp} \chi_\rho(t, \cdot) \subseteq B_{\rho}(\underline{\gamma}(t)) \subseteq \R^3$ and which is equal to $1$ in a neighbourhood of $\gamma\big([0,T]\big)$, such that $\Eh_{\rho} := \Eh \cdot \chi_\rho$ and $\Hh_{\rho}  \cdot \chi_\rho$ satisfy Maxwell's equations~\eqref{EqMax} to order $2$ by which we mean, there exists a $C(T)>0$ such that
    \begin{align} \label{eq:Max_to_order_two}
       \sup_{t\in[0,T]} \Big\|\omega^{\nicefrac{3}{4}} \bC e^{i\omega\phi}, \omega^{\nicefrac{3}{4}} \bK e^{i\omega\phi}, \omega^{\nicefrac{3}{4}} \bF e^{i\omega\phi}, \omega^{\nicefrac{3}{4}} \bG e^{i\omega\phi} \Big\|_{L^2( \R^3)}\leq C(T)\omega^{-2}.
    \end{align}
\end{proposition}

\begin{proof}
    For the purposes of the proof, let $N\in \N$ be free. We compute from equation~\eqref{eq:GBexpMaxwell} that
    \begin{align}
    \Big\| \omega^{\nicefrac{3}{4}}\bF e^{i\omega\phi} \Big\|_{L^2(\R^3)}^2&=\int_{ \R^3}\omega^{\nicefrac{3}{2}}\bF\cdot\overline{\bF} e^{-2\omega\Im\phi} d^3 x \leq C\omega^{\nicefrac{3}{2}}\sum_{A=0}^{N-1}\int_{\R^3}\omega^{2(1-A)}|\bF_A|^2e^{-2\omega\Im\phi} \, d^3 x,\label{eq:L2_F}
\end{align}
where the last inequality follows from Cauchy--Schwarz.
By the assumption that, for $|\alpha|\leq 1$, $D^{\alpha}\phi|_{\gamma}$ are real-valued and, for $|\alpha|=2$, $\Im(D^{\alpha}\phi|_{\gamma})$ is positive-definite,  there exists $c>0$ and $\tilde{\rho} >0$ such that for each $t \in [0,T]$ and $x \in B_{\tilde{\rho}}(\gammu(t))$
\begin{align} \label{eq:imphi_quadratic_lower_bound}
    \Im(\phi(t,\x))\geq \frac{c}{2} |\x-\gammu(t)|^2 \;.
\end{align}
 The monotonicity of the exponential function implies
\begin{align} \label{eq:monotonicity_est}
    e^{-2\omega\Im\phi (t,x)}\leq e^{-c\omega|\x-\gammu(t)|^2}
\end{align}
for $t \in [0,T]$ and $x \in B_{\tilde{\rho}}(\gammu(t))$. We would now like to use the estimate~\eqref{eq:monotonicity_est} in~\eqref{eq:L2_F} and apply \cref{lem:degreetoorder}. However, we need to ensure that $\bF_A$ are compactly supported. Let $\rho_m=\mathrm{min}(\rho,\tilde{\rho})$ and let $\chi_{\rho}(t,\x)$ be a smooth function such that, for each $t$:
\begin{enumerate}
\item $\mathrm{supp}(\chi_{\rho}(t,\cdot))\subseteq B_{\rho_m}(\gammu(t))$,
\item $\chi_\rho(t,\cdot)\equiv 1$ on $B_{\check{\rho}}(\gammu(t))$ for some  $0<\check{\rho}<\rho_m$.
\end{enumerate}
We now define
\begin{align}
    \Eh_{\rho}=\Eh\cdot\chi_{\rho},\qquad \Hh_{\rho}=\Hh\cdot\chi_{\rho}.
\end{align}
We note that since $\chi_{\rho}\equiv 1$ in a neighbourhood of $\gamma$, $\Eh_{\rho}$ and $\Hh_{\rho}$ maintain properties 1-6 in the statement of the proposition, and the associated $\bF_{A}$ now have support in $B_{\rho}(\gammu(t))$. We can now apply the estimate~\eqref{eq:monotonicity_est} in~\eqref{eq:L2_F} to obtain
 \begin{align}
    \Big\| \omega^{\nicefrac{3}{4}}\bF e^{i\omega\phi} \Big\|_{L^2(\R^3)}^2&\leq  C\omega^{\nicefrac{3}{2}} \Bigg[  \sum_{A=0}^{N-2}\int_{ \R^3}\omega^{2(1-A)}|\bF_A|^2 e^{-c\omega|\x-\gammu|^2} \, d^3 x + \int_{\R^3}\frac{\omega^{2}}{\omega^{2(N-1)}}|\bF_{N-1}|^2 e^{-c\omega|\x-\gammu|^2} \, d^3 x \Bigg].
\end{align}
Compact support allows us to apply \cref{lem:degreetoorder}. For the last integral, where we do not assume any vanishing of $\bF_{N-1}$, we obtain
 \begin{align}
    \omega^{\nicefrac{3}{2}}\int_{ \R^3}\frac{\omega^{2}}{\omega^{2(N-1)}}|\bF_{N-1}|^2 e^{-\omega|\x-\gammu|^2} \, d^3 x \leq C[\bF_{N-1},T]\frac{1}{\omega^{2(N-2)}}.
\end{align}
Therefore, if $N=4$ then
\begin{align}
    \sqrt{\omega^{\nicefrac{3}{2}}\int_{\R^3}\frac{\omega^{2}}{\omega^{6}}|\bF_{3}|^2 e^{-\omega|\x-\gammu|^2} \, d^3 x}\leq C[\bF_{3},T]\frac{1}{\omega^{2}}.
\end{align}
So, it suffices to have $N=3$ terms in the expansion, since any higher order contributions incur an error at $\mathcal{O}(\omega^{-2})$. 

Now, let $S_0$, $S_1$ and $S_2$ be the degree to which $\bF_0$, $\bF_1$ and $\bF_2$ vanish, respectively. Then we estimate, using \cref{lem:degreetoorder},
\begin{subequations}
\begin{align}
    \sqrt{\omega^{\nicefrac{3}{2}}\int_{\R^3}\omega^{2}|\bF_0|^2 e^{-\omega|\x-\gammu|^2} \, d^3 x} &\leq  C[\bF_0,T]\frac{\omega}{\omega^\frac{S_0+1}{2}},\\
    \sqrt{\omega^{\nicefrac{3}{2}}\int_{\R^3}|\bF_1|^2 e^{-\omega|\x-\gammu|^2} \, d^3 x} &\leq C[\bF_1,T]\frac{1}{\omega^\frac{S_1+1}{2}},\\
    \sqrt{\omega^{\nicefrac{3}{2}}\int_{ \R^3}\omega^{-2}|\bF_2|^2 e^{-\omega|\x-\gammu|^2} \, d^3 x} &\leq C[\bF_2,T]\frac{1}{\omega^{\frac{S_2+1}{2}+1}}.
\end{align}    
\end{subequations}
Thus, the estimate~\eqref{eq:Max_to_order_two} requires the degree of vanishing to be $S_0=5$, $S_1=3$, and $S_2=1$.

The cases for $C,K,G$ are the same.
\end{proof}

\begin{lemma}[Propagation of Constraints]\label{lem:ConstraintPropagation}
    Suppose $\phi$ satisfies the Eikonal equation~\eqref{eq:Eikonal} to degree $j_{\phi}\geq 3$. Suppose $\e_0$ and $\e_1$ satisfy the respective $\e_A$-constraints~\eqref{eq:eAconstraint} at $t=0$ and the transport equations for all $t$ to degrees $c_0=j_{\phi}-1$, $c_1=j_{\phi}-3$ and $j_0=j_{\phi}-1$ and $j_1=j_{\phi}-3$. Then, the $\e_0$-constraint is satisfied to degree $j_{\phi}-1$ for all $t$ along $\gamma$ and the $\e_1$-constraint is satisfied to degree $j_{\phi}-3$ for all $t$ along $\gamma$. 
\end{lemma}

\begin{proof}
It is instructive to first do the $j_{\phi}=1$ case. Recall that $\dot{\gamma}^j:=- \frac{\nabla^j\phi}{\n^2\dot{\phi}} \big|_{\gamma }$. We then compute that
\begin{subequations}
\begin{align}
    \bigg(\partial_t-\frac{\nabla^j\phi}{\n^2\dot{\phi}}\nabla_j \bigg) (\sqrt{\eps}\e_0^n) &=-\sqrt{\eps} \e_0^n \frac{\nabla^j\phi}{2\n^2\dot{\phi}}\nabla_j\ln\eps + \sqrt{\eps}  \bigg(\partial_t-\frac{\nabla^j\phi}{\n^2\dot{\phi}}\nabla_j \bigg)\e_0^n,\\
    \bigg(\partial_t-\frac{\nabla^j\phi}{\n^2\dot{\phi}} \nabla_j \bigg) (\nabla_i\phi) &= \nabla_i \dot{\phi} - \frac{1}{2\n^2 \dot{\phi}} \nabla_i[\nabla\phi\cdot\nabla\phi] .
\end{align}
\end{subequations}
Using these identities together with the $\e_0$ transport equation at degree $j_0=0$ and the Eikonal equation~\eqref{eq:Eikonal} to degree $j_{\phi}=1$ gives 
\begin{subequations}
\begin{align}
\Big(\partial_t+\dot{\gamma}^m\nabla_m\Big) (\sqrt{\eps}\e^n_{0})\Big|_{\gamma} &= \frac{1}{2\n^2\dot{\phi}}\sqrt{\eps}\e^n_{0} \Big( \Delta\phi - \n^2 \ddot{\phi} \Big) - \frac{1}{\n^2\dot{\phi}} \Big(\sqrt{\eps}\e^{[n}_{0}\nabla^{m]}\phi\Big) \nabla_m \ln \n^2 \Big|_{\gamma}, \label{eq:e0transport}\\
\Big(\partial_t+\dot{\gamma}^j\nabla_j\Big) (\nabla_i\phi) \Big|_{\gamma } &= \nabla_i\dot{\phi}-\frac{1}{2\n^2 \dot{\phi}}\nabla_i[\nabla\phi\cdot\nabla\phi] \Big|_{\gamma }=-\frac{\dot{\phi}}{2}\nabla_i \ln\n^2 \Big|_{\gamma}.
\end{align}
\end{subequations}
Next, we compute
\begin{align}
    \bigg(\partial_t - \frac{\nabla^j\phi}{\n^2\dot{\phi}}\nabla_j\bigg) (\sqrt{\eps}\e_0^i \nabla_i\phi) &= (\nabla_i\phi) \bigg(\partial_t-\frac{\nabla^j\phi}{\n^2\dot{\phi}}\nabla_j \bigg) (\sqrt{\eps}\e_0^i) + \sqrt{\eps}\e^i_0 \bigg( \partial_t-\frac{\nabla^j\phi}{\n^2\dot{\phi}}\nabla_j \bigg) (\nabla_i\phi).
\end{align}
Evaluating on $\gamma$ and using the Eikonal equation \eqref{eq:Eikonal} gives
\begin{align}
    \Big( \partial_t+\dot{\gamma}^j\nabla_j \Big) (\sqrt{\eps}\e_0^i\nabla_i\phi) \Big|_{\gamma} &= \frac{1}{2\n^2\dot{\phi}} \Big[ \Delta\phi-\n^2 \ddot{\phi} -(\nabla^m\phi) \nabla_m \ln \n^2 \Big]\sqrt{\eps}\e^i_{0}\nabla_i\phi= \nabla_j \bigg( \frac{\nabla^j\phi}{2\n^2\dot{\phi}} \bigg) \sqrt{\eps} \e_0^i\nabla_i\phi \bigg|_{\gamma}.
\end{align}
By the ODE uniqueness, $\bC_0|_{\gamma}=0$ is the unique solution.

We now proceed to general $j_{\phi}\geq 2$. We note that 
\begin{align}
   -\frac{\nabla^j\phi}{\n^2\dot{\phi}} \nabla_jD^{\alpha}(\sqrt{\varepsilon}\e_0^i\nabla_i \phi)= D^{\alpha} \bigg[-\frac{\nabla^j\phi}{\n^2\dot{\phi}} \nabla_j(\sqrt{\varepsilon}\e_0^i\nabla_i \phi) \bigg]   + \sum_{0<\beta\leq \alpha}\binom{\alpha}{\beta}D^{\beta} \bigg(\frac{\nabla^j\phi}{\n^2\dot{\phi}} \bigg) \nabla_j D^{\alpha-\beta} (\sqrt{\varepsilon} \e_0^i \nabla_i\phi ),
\end{align}
which we can use to compute that, if $c_0=j_{\phi}-1\geq1$ and $j_0=j_{\phi}-1\geq 1$, then for $|\alpha|\leq j_{\phi}-1$ we have
\begin{align}
    \Big( \partial_t + \dot{\gamma}^j \nabla_j \Big) D^{\alpha}(\sqrt{\eps} \e_0^i \nabla_i\phi ) \Big|_{\gamma} = D^{\alpha}\bigg[ \nabla_j \bigg(\frac{\nabla^j\phi}{2\n^2\dot{\phi}} \bigg) \sqrt{\eps}\e_0^i \nabla_i\phi \bigg] + \sum_{0<\beta\leq \alpha}\binom{\alpha}{\beta}D^{\beta}\bigg(\frac{\nabla^j\phi}{\n^2\dot{\phi}} \bigg) \nabla_j D^{\alpha-\beta} (\sqrt{\varepsilon}\e_0^i \nabla_i\phi) \bigg|_{\gamma}.
\end{align}
The result now proceeds by induction, using $D^{\beta}\bC_0|_{\gamma}=0$ for $|\beta|\leq |\alpha|-1$, where the base case is proved above.
    
We now proceed to show the propagation of the~$\e_1$-constraint. Let us start with the $j_{\phi}=3$ case. Note that since the $\e_0$-constraint holds initially to degree $2$, the $\e_0$-constraint holds for all $t$ to degree $2$ from above. In propagating the $e_1$-constraint we must use $j_{\phi}=3$ and $j_0=2$. We want to compute
\begin{align}
    \Big( \partial_t+ \dot{\gamma}^m \nabla_m \Big) \bC_1 = \Big( \partial_t + \dot{\gamma}^m \nabla_m \Big) \big(\div\e_{0} + \ii \e_{1}^i\nabla_i\phi + \e_{0}^i\nabla_i\ln\eps \big). 
\end{align}
We recall the $\e_{0}$ and $\e_1$-transport equations~\eqref{eq:eAtransport}:
\begin{subequations}
\begin{align}
    \Big( \partial_t+\dot{\gamma}^m\nabla_m \Big)D^{\alpha}\e^n_{0} \Big|_{\gamma} &= D^{\alpha} \bigg\{ \frac{1}{2\n^2\dot{\phi}} \bigg[ \e^n_{0} \Big( \Delta\phi - \n^2 \ddot{\phi} - \nabla^{m}\phi  \nabla_m\ln\mu \Big) +  \e^{m}_{0} \nabla^{n}\phi \nabla_m \ln \n^2 \bigg] \bigg\}\bigg|_{\gamma}\nonumber\nonumber\\
    &\qquad+\sum_{0<\beta\leq \alpha}\binom{\alpha}{\beta} D^{\beta} \bigg(\frac{\nabla^m\phi}{\n^2\dot{\phi}} \bigg) \nabla_m D^{\alpha-\beta}\e_0^n \bigg|_{\gamma}, \\
    \Big( \partial_t+\dot{\gamma}^m\nabla_m \Big) \e^n_{1} \Big|_{\gamma} &= \frac{1}{2\n^2\dot{\phi}} \bigg[ \e^n_{1} \Big( \Delta\phi - \n^2 \ddot{\phi} \Big) - \ii \Big( \Delta\e^n_{0}  - \n^2 \ddot{\e}^n_{0} \Big) - \ii \e_{0}^m\nabla_n\nabla_m\ln\eps \nonumber\\
    &\qquad+ \Big( \e^{m}_{1}\nabla^{n}\phi - \ii \nabla^n\e^m_{0}\Big) \nabla_m \ln \n^2 - \Big( \e^{n}_{1}\nabla^{m}\phi- \ii \nabla^m\e^n_{0}\Big) \nabla_m\ln\mu \bigg] \bigg|_{\gamma}.
\end{align}
\end{subequations}
Since $j_0=2>1$, the $\e_0$-transport equation holds to degree $1$. Using this and the $\e_1$-transport equation, an arduous computation yields
\begin{align}
    \Big(\partial_t + \dot{\gamma}^m \nabla_m \Big) \bC_1  \Big|_{\gamma} &= \frac{1}{2\n^2 \dot{\phi}} \Big\{ \Big(\Delta\phi - \n^2 \ddot{\phi} -\nabla^{m}\phi \nabla_m\ln\mu \Big) \bC_1 + \Delta(\e_0^n\nabla_n\phi) - \n^2 \partial_t^2 (\e_0^n\nabla_n\phi) \nonumber\\
    &\qquad-(\nabla_m\ln\mu) \nabla^m( \e^n_{0}\nabla_n\phi) - \ii \e^n_1 \nabla_n[ \nabla\phi\cdot\nabla\phi] +\ii \nabla\phi\cdot\nabla\phi \e^{m}_{1} \nabla_m\ln \n^2 \nonumber\\
    &\qquad-\frac{\nabla_n\dot{\phi}}{\dot{\phi}} \Big[\e^n_{0}\Delta\phi - \n^2 \e^n_{0} \ddot{\phi} + \e^{m}_{0}\nabla^{n}\phi \nabla_m \ln \n^2 - \e^{n}_{0}\nabla^{m}\phi \nabla_m\ln\mu \nonumber\\
    &\qquad- 2 \ii \n^2 \dot{\phi}^2 \e_1^n -2 \n^2 \dot{\phi} (\partial_t+\dot{\gamma}^m\nabla_m) \e_0^n \Big]\Big\}\Big|_{\gamma},
\end{align}
where we used
\begin{subequations}
\begin{align}
    \ddot{\e}_0^n \nabla_n\phi + \e_0^n\nabla_n\ddot{\phi} &= \partial_t^2 (\e_0^n\nabla_n\phi) - 2\dot{\e}_0^n\nabla_n\dot{\phi} \nonumber\\
    &= \partial_t^2 (\e_0^n\nabla_n\phi) - 2 ( \partial_t + \dot{\gamma}^m \nabla_m) \e_0^n \nabla_n \dot{\phi} + 2\dot{\gamma}^m\nabla_m\e_0^n \nabla_n\dot{\phi}, \\
    (\Delta\e_0^n) \nabla_n\phi+\e_0^n\Delta\nabla_n\phi & =\Delta(\e_0^n\nabla_n\phi) - 2(\nabla^m\e_0^n) \nabla_m\nabla_n\phi.
\end{align}   
\end{subequations}
Since $\bC_0|_{\gamma}=0$ to degree 2, \cref{prop:timeder} gives $\partial_t^2 (\e_0^n\nabla_n\phi)|_{\gamma}=0$.\footnote{This is where we need to require that the $\e_0$-constraint and transport equations hold to degree $2$, which then requires the Eikonal to degree $3$.} Using these facts along with the Eikonal equation at degree $1$ and the $\e_0$-transport equation yields
\begin{align}
    \Big(\partial_t+\dot{\gamma}^k \nabla_k \Big) \bC_1\Big|_{\gamma} &= \frac{1}{2\n^2 \dot{\phi}} \Big( \Delta\phi-\mu\eps\ddot{\phi}- \nabla^{m}\phi \nabla_m \ln\mu \Big)\bC_1\Big|_{\gamma}.
\end{align}
Weighting with $\sqrt{\eps}$ gives
\begin{align}
    \Big(\partial_t + \dot{\gamma}^m\nabla_m \Big) (\sqrt{\eps} \bC_1) \Big|_{\gamma} &= -\nabla_m\bigg(\frac{\nabla^m\phi}{2\n^2\dot{\phi}}\bigg) \sqrt{\eps}\bC_1\bigg|_{\gamma}.
\end{align}
By ODE uniqueness, $\bC_1|_{\gamma}=0$ for all $t$. The general result with derivatives now proceeds by induction.
\end{proof}

We now want to show that if $(\phi,\e_0,\e_1,\e_2)$ satisfies \cref{def:EikonalandTransporteA} to some degree, then we may define $h_0, h_1, h_2$ such that  assumptions 4-6 of \cref{prop:order_2_from spacetime_functions} are satisfied.

\begin{proposition}\label{prop:2sttranstoMax}
Suppose that $\phi$ satisfies the Eikonal equation~\eqref{eq:Eikonal} to degree $j_{\phi}=6$, the $(\e_0,\e_1)$-constraint equations~\eqref{eq:eAconstraint} are satisfied to degrees $(c_0,c_1)=(5,3)$ at $t=0$ and the $\e_2$-constraint to degree $c_2=1$ along $\gamma$. Suppose that the $\e_0$ and $\e_1$-transport equations~\eqref{eq:eAtransport} are satisfied to degrees $(j_0,j_1)=(5,3)$.
Define
\begin{subequations} \label{eq:h012def}
\begin{align}
    D^{\alpha}\h^i_{0} \big|_\gamma &:= D^{\alpha}\bigg(-\frac{1}{\mu \dot{\phi}}{\epsilon^{ij}}_k\e^k_{0}\nabla_j\phi\bigg)\bigg|_{\gamma} &\forall |\alpha|\leq 5, \label{eq:h0def} \\
    D^{\alpha}\h^i_{1} \big|_\gamma &:= D^{\alpha}\bigg(-\frac{1}{\mu \dot{\phi}}{\epsilon^{ij}}_k\e^k_{1}\nabla_j\phi+\frac{\ii}{\mu \dot{\phi}}{\epsilon^{ij}}_k\nabla_j\e^k_{0}+\frac{\ii}{\dot{\phi}} \dot{\h}_0^i\bigg)\bigg|_{\gamma} \qquad&\forall |\alpha|\leq 3, \label{eq:h1def}\\
    D^{\alpha}\h^i_{2} \big|_\gamma &:= D^{\alpha}\bigg(-\frac{1}{\mu \dot{\phi}} {\epsilon^{ij}}_k\e^k_{2}\nabla_j\phi+\frac{\ii}{\mu \dot{\phi}}{\epsilon^{ij}}_k\nabla_j\e^k_{1} + \frac{\ii}{\dot{\phi}} \dot{\h}_1^i\bigg)\bigg|_{\gamma} \qquad&\forall |\alpha|\leq 1, \label{eq:h2def}
\end{align}   
\end{subequations}
where $\dot{\h}_A$ and its spatial derivatives are computed from the tangential derivative to $\gamma$ and spatial derivatives as 
\begin{align}
D^\alpha\dot{\h}_A^i=\Big(\dot{\gamma}^{\nu}\partial_{\nu}D^\alpha\h^i_{A}-\dot{\gamma}^{k}\nabla_{k}D^\alpha\h^i_{A}\Big).
\end{align}
Note that \cref{eq:h012def} corresponds to $F_0, F_1, F_2$ vanishing along $\gamma$ to degree $5,3,1$, respectively. Then
\begin{enumerate}
    \item $\bG_0|_{\gamma}=0=\bK_0|_{\gamma}$ to degree $5$, \label{item_G0}
    \item $\bG_1|_{\gamma}=0=\bK_1|_{\gamma}$ to degree $3$,
     \item $\bG_2|_{\gamma}=0=\bK_2|_{\gamma}$ to degree $1$.
\end{enumerate}
\end{proposition}

\begin{proof}
We now use Lemma \ref{lem:decompGF} and start with \eqref{EqUseForK}, which gives
\begin{equation}
     \mu \dot{\phi} K_A = F_A^i \nabla_i \phi + \ii \div F_{A-1} + \ii \mu\rd_t K_{A-1} \;.
\end{equation}
Setting $A=0$, it follows directly from $F_0$ vanishing along $\gamma$ to degree $5$ that $K_0$ vanishes along $\gamma$ to degree $5$. For $A=1$ we proceed similarly, now using that $\div F_0$ vanishes to degree $4$ along $\gamma$ and that $\rd_t K_0$ vanishes to degree $4$ along $\gamma$ by Proposition \ref{prop:timeder}. Finally, the case $A=2$ proceeds in exactly the same way.

To show the vanishing of $G_A$, we use \eqref{EqUseForG}, which  gives
\begin{equation}
\begin{split}
    \mu\dot{\phi} \bG_{A}^n &= (\star\bF_{A})^{mn}\nabla_m\phi+\bC_{A}\nabla^n\phi- \ii\nabla^n\bC_{A-1}-\mu\div\Big(\frac{\ii}{\mu}\star\bF_{A-1}^{n}\Big)+\ii\partial_t \bF_{A-1}^n \nonumber \\
   &\qquad +2\n^2 \dot{\phi}\Big((\e_{A-1}^n-\mathrm{transport})[0]\Big)-\ii\e_{A}^n\Big(\mathrm{Eikonal}[0]\Big) \;. \label{EqUseUseG}
   \end{split}
\end{equation}
Furthermore, we recall that by \cref{lem:ConstraintPropagation}, the $\e_0$ and $\e_1$ constraints are satisfied to degree $c_0=5$ and $c_1=3$ for all $t$ along $\gamma$. We now start with $A=0$ in \eqref{EqUseUseG}. Since $F_0, C_0$ and $\mathrm{Eikonal}[0]$ vanish to degree $5$ along $\gamma$, it follows that $G_0$ vanishes to degree $5$ along $\gamma$.\footnote{Na\"ively, it seems as though to produce $\bG_0|_{\gamma}=0$ to degree $5$ one requires that the Eikonal equation~\eqref{eq:Eikonal} be satisfied merely to degree $5$. However, to propagate the constraint $\bC_0$ to degree $5$ along $\gamma$, one requires an additional degree for the Eikonal equation.} The cases $A=1,2$ again follow similarly using Proposition \ref{prop:timeder} -- and for $A=2$ we also use our assumption of the proposition that $C_2$ vanishes to degree $1$ along $\gamma$.
\end{proof}

\subsubsection{Construction of the phase function}

In this subsection, we provide the existence result for the eikonal equation~\eqref{eq:Eikonal}.
\begin{proposition}\label{prop:construction_phase_function}
Let $\x_0\in \R^3$ and $j_{\phi}\geq 2$. Consider initial data of $D^{\alpha}\phi|_{(0,\x_0)}=D^{\alpha}\phu|_{\x_0}$ for $0\leq |\alpha|\leq j_{\phi}$, such that 
\begin{enumerate}
\item for $|\alpha|\leq 1$, $D^{\alpha}\phu|_{\x_0}\in \R$ and there is some $\alpha $, with $|\alpha|=1$, such that $D^{\alpha}\phu|_{\x_0}\neq 0$,
    \item for $|\alpha|=2$, the bilinear form $\Im\big(D^{\alpha}\phu|_{\x_0}\big)$ is positive definite.
\end{enumerate}
Let $\gamma$ be the future-directed null geodesic (with respect to $g$) starting at $(0,\x_0)$ with initial tangent 
    \begin{align}
    \dot{\gamma}(0)=\bigg(1,\frac{1}{\n\big|\nabla\phu\big|} \nabla\phu \bigg|_{\x_0}\bigg).
    \end{align} 
Then, there exists a (unique)\footnote{By this we mean that the formal Taylor expansion to degree $j_\phi$ of $\phi$ along $\gamma$ is uniquely determined.} solution of the Eikonal equation~\eqref{eq:Eikonal} to degree $j_{\phi}$ along all of $\gamma$ which attains the prescribed initial data above and for $p=0,1$ and $|\alpha|\leq j_{\phi}-p$ satisfies
    \begin{align}
      \Big[ (\partial_t)^p D^{\alpha}\dot{\phi}\Big](0,\x_0)=\Big[ (\partial_t)^p D^{\alpha}\Big(-\frac{1}{\n}\sqrt{\nabla^i \phu\nabla_i \phu}\Big)\Big](0,\x_0).
    \end{align}
Moreover, the  bilinear form $\Im\big(D^{\alpha}\phi|_{\gamma(t)}\big)$ with $|\alpha| = 2$ is positive-definite for all $t$.
\end{proposition}

\begin{remark}
    The existence of an approximate solution to the eikonal equation to degree $j_\phi$ can also be directly inferred from the spacetime construction in \cite{Sbierski2013} with the Lorentzian metric $g$ on $\R^{1+3}$ or from \cite{ralston1982}. For the convenience of the reader and to keep the paper self-contained, we however give a proof below. Moreover, the method of proof chosen here is naturally adapted to the canonical $1+3$-splitting of $\R^{1+3}$ and the Riemannian geometry of $(\R^3, \underline{g})$ and slightly differs from those in \cite{ralston1982}, \cite{Sbierski2013}.
\end{remark}

\begin{proof}
 It is here that using the conformally rescaled optical metric $g$ on $\R^{3+1}$ given in \cref{eq:optical_metric} becomes convenient. As shown in \cref{SecOpticalGeom}, $t \mapsto \gamma(t) = \big(t, \underline{\gamma}(t)\big)$ is a null geodesic in $(\R^{1+3}, g)$ if and only if $t \mapsto \underline{\gamma}(t)$ is a  geodesic in $(\R^3, \underline{g})$ parametrised by $\underline{g}$-arclength. Recall that the geodesic equation in $(\R^3, \underline{g})$ is given by \cref{EqHamiltonianEq}. 
By standard ODE existence and uniqueness, we obtain a solution of the above geodesic equation \eqref{EqHamiltonianEq} with initial point $\gammu^i(0) = \x_0^i$ and initial tangent
\begin{align}
    \frac{d\gammu^i(0)}{dt} = \dot{\gammu}^i(0) = \frac{1}{\n\big|\nabla\phu\big|} \nabla^i \phu \bigg|_{\x_0}.
\end{align}
From the spacetime perspective, we obtain a future-directed null geodesic $\gamma(t)=(t,\gammu^i(t))$. We now construct our solution to the eikonal equation~\eqref{eq:Eikonal} by defining
\begin{align} \label{eq:phase_construction}
    \dot{\phi} \big|_{\upgamma} = - \frac{|\nabla \phu|}{\n} \bigg|_{\x_0} , \qquad \nabla^k \phi \big|_{\upgamma} = - \n^2 \dot{\phi} \dot{\gammu}^k \big|_\gamma = \frac{|\nabla \phu|}{\n}\Big|_{\x_0} \n^2\dot{\gammu}^k|_{\upgamma}.
\end{align}
We compute
\begin{align}
    \n^2\dot{\phi}^2-\nabla\phi\cdot\nabla\phi|_{\gamma}&=\n^2\Big|\frac{1}{\n}\nabla\phu\Big|_{\x_0}\Big|^2-\Big|\frac{1}{\n}\nabla\phu\Big|_{\x_0}\Big|^2 \n^4\delta_{km}\dot{\gammu}^k\dot{\gammu}^m\Big|_{\upgamma}=\n^2\Big|\frac{1}{\n}\nabla\phu\Big|_{\x_0}\Big|^2-\Big|\frac{1}{\n}\nabla\phu\Big|_{\x_0}\Big|^2 \n^4\frac{1}{\n^2}\Big|_{\upgamma}=0,
\end{align}
since $|\dot{\gammu}(t)|^2=\frac{1}{\n^2}$ from the nullity of $\gamma$. This completes the construction for degree $0$. 
We note four observations before moving onto the degree $1$ construction:
\begin{itemize}
\item First, the Eikonal equation also implies 
\begin{equation}
    \frac{|\nabla \phi|^2}{\n^2} \bigg|_\gamma = \dot{\phi}^2 \big|_\gamma = \frac{|\nabla \phu|^2}{\n^2} \bigg|_{\x_0}.
\end{equation}
\item Second, this equation and the Eikonal equation imply that
\begin{align}\label{eq:phi_const_along_gamma}
    \Big( \partial_t + \dot{\gamma}^j \nabla_j \Big) \phi \Big|_{\gamma} = 0 \qquad\implies\qquad \phi \big|_{\gamma} = \phu \big|_{\x_0}.
\end{align}
    \item Third, for all $t$, we have
\begin{align}\label{eq:nabphireal}
    \dot{\gammu}^k = \dot{\gamma}^k = -\frac{1}{\n^2\dot{\phi} }\nabla^k \phi\Big|_{\upgamma}=\frac{1}{\n|\nabla\phi| }\nabla^k \phi\Big|_{\upgamma}.
\end{align}

\item Fourth, the following equations are equivalent by the definition of $\dot{\phi}|_{\gamma}$:
\begin{equation} \label{eq:dt_eikonal}
    \frac{d}{d t} \Big( \dot{\phi}\big|_\gamma \Big) =  \Big( \partial_t + \dot{\gamma}^j \nabla_j \Big) \dot{\phi}\Big|_\gamma = 0 \qquad \Longleftrightarrow \qquad \partial_t \big( \nabla\phi\cdot\nabla\phi - \n^2\dot{\phi}^2 \big) \big|_\gamma = 0.
\end{equation}
\end{itemize}

We now move to degree $1$, where we will show that $\nabla_a (\nabla\phi\cdot\nabla\phi - \n^2\dot{\phi}^2) |_{\gamma} = 0$ is satisfied by the construction given in \cref{eq:phase_construction}. We start by writing
\begin{align}
    \nabla_a \big( \nabla\phi\cdot\nabla\phi - \n^2\dot{\phi}^2 \big) \big|_{\gamma} = 2 \big( \nabla^i \phi\nabla_a \nabla_i \phi - \n^2 \dot{\phi} \nabla_a \dot{\phi} - \n^2 \dot{\phi}^2 \nabla_a \ln \n \big) \big|_\gamma,
\end{align}
which is equivalent to 
\begin{equation} \label{eq:deg1_intermetiate}
    \bigg( \partial_t - \frac{\nabla^i \phi}{\n^2 \dot{\phi}} \nabla_i \bigg) \nabla_a \phi \bigg|_\gamma = \Big( \partial_t + \dot{\gamma}^i \nabla_i \Big) \nabla_a \phi \Big|_\gamma = - \dot{\phi} \nabla_a \ln \n \Big|_\gamma.
\end{equation}
On the other hand, from \cref{eq:phase_construction} and the geodesic equation \cref{EqHamiltonianEq}, we obtain
\begin{equation}
    \Big( \partial_t + \dot{\gamma}^i \nabla_i \Big) \nabla^a \phi \Big|_\gamma = \frac{d}{d t} \Big( \nabla^a \phi \big|_\gamma \Big) = \frac{d}{d t} \Big( - \n^2 \dot{\phi} \dot{\gamma}^a \big|_\gamma \Big) = -\dot{\phi} \frac{d}{dt} p_a \big|_{\gamma} = - \dot{\phi} \nabla_a \ln \n \big|_{\gamma}\;.
\end{equation}
Thus, the eikonal equation is satisfied to degree $1$ by the construction in \cref{eq:phase_construction}. 

We now consider the degree $2$ and compute
\begin{equation}
    \begin{split}
        \frac{1}{2} \nabla_b \nabla_a \big( \nabla\phi\cdot\nabla\phi - \n^2\dot{\phi}^2 \big) \big|_{\gamma} &= \nabla_b \nabla^i \phi \nabla_a \nabla_i \phi + \nabla^i \phi \nabla_a \nabla_b \nabla_i \phi - 2 \n^2 \nabla_b \ln \n \cdot \dot{\phi} \nabla_a \dot{\phi} - \n^2 \nabla_b \dot{\phi} \nabla_a \dot{\phi}  \\ &\quad- \n^2 \dot{\phi} \nabla_a \nabla_b \dot{\phi} - 2 \n^2 \nabla_a \ln \n \nabla_b \ln \n \cdot \dot{\phi}^2 - 2 \n^2 \nabla_b \dot{\phi} \cdot \dot{\phi} \nabla_a \ln \n - \n^2 \dot{\phi}^2 \nabla_a \nabla_b \ln \n  \\
        & \overset{!}{=} 0. \label{EqMiddle}
    \end{split}
\end{equation}
When we evaluate \eqref{EqMiddle} on $\gamma$ and use \cref{eq:deg1_intermetiate} to write
\begin{equation}
    \nabla_a \dot{\phi} \Big|_\gamma = \frac{\nabla^i \phi}{\n^2 \dot{\phi}} \nabla_i \nabla_a \phi - \dot{\phi} \nabla_a \ln \n \Big|_\gamma
\end{equation}
we obtain after division by $-\n^2 \dot{\phi}|_\gamma$:
\begin{align}
   &\dot{\gamma}^\nu\partial_{\nu}(\nabla_a\nabla_b\phi)+\frac{1}{\n^2\dot{\phi}}\nabla_b(\ln\n)\delta^{kj}\nabla_k\phi\nabla_j\nabla_a\phi+\frac{1}{\n^2\dot{\phi}}\nabla_a(\ln\n)\delta^{kj}\nabla_k\phi\nabla_j\nabla_b\phi \nonumber \\
   &\qquad-\nabla_a(\ln\n)\nabla_b(\ln\n)\dot{\phi}+\nabla_a\nabla_b[\ln\n]\dot{\phi} +\frac{1}{\n^2\dot{\phi}}\bigg(\frac{\nabla_p\phi\delta^{kp}\nabla_q\phi\delta^{jq}}{\n^2\dot{\phi}^2}-\delta^{kj}\bigg)\nabla_b\nabla_j\phi\nabla_k\nabla_a\phi=0 \;. \label{EqNTER}
\end{align}
Defining
\begin{subequations}
\begin{align}
    M_{a b} &= \nabla_a \nabla_b \phi \big|_\gamma, \\
    L_{a b} &= \frac{1}{\n^4 \dot{\phi}^3} \Big[ \big( \nabla_a \phi \big) \big( \nabla_b \phi \big) - \n^2 \dot{\phi}^2 \delta_{a b} \Big] \Big|_\gamma, \\
    N_{a b} &= \frac{1}{\n^2\dot{\phi}} \big( \nabla_a \ln\n \big) \big( \nabla_b \phi \big) \Big|_\gamma,\\
    R_{a b} &= \dot{\phi} \Big[ \nabla_a \nabla_b \ln\n - \big(\nabla_a \ln\n \big) \big( \nabla_b \ln\n \big) \Big] \Big|_\gamma\;,
\end{align}
\end{subequations}
\eqref{EqNTER} takes the form of a Riccati equation along $\gamma$ for the second spatial derivatives of $\phi$:
\begin{align} \label{eq:Riccati}
    \frac{d}{dt}M+M\cdot L\cdot M+N\cdot M+M\cdot N^T+R=0\;.
\end{align}

Let $J$ and $V$ be $3\times 3$ matrices that satisfy the linear ODE system
\begin{align}
    \frac{d}{dt}J=N^TJ+LV,\qquad \frac{d}{dt}V=-NV-RJ.
\end{align}
Since this is linear, we have the existence of a global solution. Moreover, if $J$ is invertible, then $M=VJ^{-1}$ solves the Riccati equation.  

We now show that $J$ is invertible by constructing a conserved quantity from the symplectic form on the cotangent bundle 
\begin{align}
    \omega=dx^k\wedge dp_{k},
\end{align}
and then arguing by contradiction. To this end, let $v\in \C^3$ and define 
\begin{align}
    X_v=(Jv)^k\partial_{x^{k}}+(Vv)^{k}\partial_{p_{k}}.
\end{align}
We compute 
\begin{align}
    \omega(X,\overline{X})(t)=[J(t)v]\cdot[\overline{V(t)v}]- [V(t)v]\cdot[\overline{J(t)v}],
\end{align}
and 
\begin{align}
    \frac{d}{dt}\omega(X_v,\overline{X}_v)(t)&= [\dot{J}(t)v]\cdot[\overline{V(t)v}]+ [J(t)v]\cdot[\overline{\dot{V}(t)v}]- [\dot{V}(t)v]\cdot[\overline{J(t)v}]- [V(t)v]\cdot[\overline{\dot{J}(t)v}] \nonumber\\
    &= [N^TJv+LVv]\cdot[\overline{V(t)v}]- [V(t)v]\cdot[\overline{N^TJv+LVv}] \nonumber\\
    &\qquad+ [J(t)v]\cdot[\overline{-NVv-RJv}]-[ -NVv-RJv]\cdot[\overline{J(t)v}] \nonumber\\
    &=0,
\end{align}
since $L$ and $R$ are symmetric and all of $L,N,R$ are real.

Suppose $J$ is not invertible, then there is a $t_0>0$ and $v\in \C^3$ where $J(t_0)v=0$. If we take the initial data $J(0)=\mathbbm{1}_3$ and $V(0)=M(0)$, then we can use the conservation of $\omega(X_v,\overline{X}_v)(t)$ to show
\begin{align}
    0=\omega(X_v,\overline{X}_v)(t_0)=\omega(X_v,\overline{X}_v)(0)= v\cdot[\overline{M(0)}v]- [M(0)v]\cdot v=-\ii \{\Im[M(0)]v\}\cdot v,
\end{align}
which is a contradiction to the positivity of $\Im[M(0)]$. 

To show that $\Im[M(t)]>0$ for all time, we note that $M(t):= V(t)J^{-1}(t)$, so $V(t)=M(t)J(t)$. We then compute
\begin{align}
   \omega(X_v,\overline{X}_v)(t)&= [J(t)v]\cdot[\overline{V(t)v}]-[ V(t)v]\cdot[\overline{J(t)v}]= [J(t)v]\cdot[\overline{M(t)J(t)v}]-[ M(t)J(t)v]\cdot[\overline{J(t)v}] \nonumber\\
    &=-2\ii\{\Im[M(t)]J(t)v\}\cdot[\overline{J(t)v}].
\end{align}
Conservation of $\omega(X_v,\overline{X}_v)(t)$ then gives
\begin{align}
    \{\Im[M(t)]J(t)v\}\cdot[\overline{J(t)v}] = \{\Im[M(0)]J(0)v\}\cdot[\overline{J(0)v}] = \{\Im[M(0)]v\}\cdot\overline{v}>0.
\end{align}
Since $J$ is an invertible linear map, it is an isomorphism. Therefore, for any $w\in \C^3$, $\exists v\in \C^3$ such that $w=Jv$. Hence, we have
\begin{align}
    \{\Im[M(t)]w\}\cdot\overline{w}>0,\qquad \forall w\neq 0,\forall t\geq 0.  
\end{align}

If $j_{\phi}=2$, then the above completes the construction. However, if $j_{\phi}>2$, then at degree $2 < q \leq j_{\phi}$, we find that 
\begin{align}
    \nabla_{k_1}...\nabla_{k_q} \big(\nabla\phi\cdot\nabla\phi - \n^2\dot{\phi}^2 \big) \big|_{\gamma} = 0
\end{align}
gives a linear ODE for $\nabla_{k_1}...\nabla_{k_q} \phi |_\gamma$, which we can solve for all time with standard ODE existence results. 

We can now extend to a spacetime function via Borel's lemma~\ref{lem:Borel}:
\begin{align}\label{eq:phi_borel}
    \phi(t,\x)=a(t)+b_k(t)[\x^k-\gammu^k(t)]+\sum_{2\leq|\alpha|\leq j_{\phi}}\frac{1}{\alpha!}D^{\alpha}\phi|_{\gamma}[\x-\gammu(t)]^{\alpha},
\end{align}
with
\begin{align}
    a(t)=\phi|_{\gamma}=\phu|_{x_0},\qquad b_k(t)=\nabla_k\phi|_{\gamma}=\frac{|\nabla \phu|}{\n}\Big|_{\x_0}  \n^2\dot{\gammu}^k|_{\upgamma}.
\end{align}

Finally, we now address the initial values of $\dot{\phi}$ and $\ddot{\phi}$. First, by definition, we have
\begin{align}
    \dot{\phi} \big|_{(0,\x_0)} = - \frac{|\nabla \phu|}{\n} \bigg|_{\x_0}.
\end{align}
We then use that, by construction and Proposition \ref{prop:timeder}, we have
\begin{equation}
   \rd_t^p D^\alpha \dot{\phi}^2|_{(0,\x_0)}  = \rd_t^p D^\alpha \big(\frac{1}{\n^2} \nabla \phi \cdot \nabla \phi\big) |_{(0,\x_0)}
\end{equation}
for $p = 0,1$ and $|\alpha| \leq j_\phi - p$. Since the complex square root is a smooth function in a neighbourhood of $\dot{\phi}^2|_{\x_0}$, we obtain by the chain rule and induction  that 
\begin{equation}
    \rd_t^p D^\alpha \dot{\phi}|_{(0,\x_0)} = \rd_t^p D^\alpha \big[ -\frac{1}{\n} \sqrt{\nabla \phi \cdot \nabla \phi}\big]|_{(0,\x_0)} \;.
\end{equation}

\end{proof}

\subsubsection{Construction theorem}

In this subsection we provide the existence result for the approximate solutions to the Maxwell equations in an inhomogeneous medium. This makes use of~\cref{prop:construction_phase_function} above. 

\begin{theorem}\label{thm:construction_theorem}
    Let $\x_0\in \R^3$ and $\rho>0$ be given. Suppose that we are given initial data of $D^{\alpha}\phi|_{(0,\x_0)}=D^{\alpha}\phu|_{\x_0}$ for $0\leq |\alpha|\leq 7$, such that 
\begin{enumerate}
\item for $|\alpha|\leq 1$, $D^{\alpha}\phu|_{\x_0}\in \R$ and there is some $\alpha $, with $|\alpha|=1$, such that $D^{\alpha}\phu|_{\x_0}\neq 0$, 
    \item for $|\alpha|=2$, the bilinear form $\Im\big(D^{\alpha}\phu|_{\x_0}\big)$ is positive definite.
\end{enumerate}
This represents initial data for the Eikonal equation~\eqref{eq:Eikonal} to degree $7$ along the future-directed null geodesic starting at $(0,\x_0)$ with initial tangent
    \begin{align}
    \dot{\gamma}(0)=\bigg(1,\frac{1}{\n\big|\nabla\phu\big|} \nabla^i\phu \bigg|_{\x_0}\bigg).
    \end{align} 
    Let $\phi$ be the solution as given in \cref{prop:construction_phase_function}.  Furthermore, suppose that we are also given initial data
    \begin{enumerate}
    \item $D^{\alpha}\eu_0|_{\x_0}$ such that $\eu_0|_{\x_0}\neq 0$ and $D^{\alpha}( \eu_0^k \nabla_k\phi )|_{\x_0}=0$ for $|\alpha|\leq 5$,
    \item $D^{\alpha} \eu_1|_{\x_0}$ such that $D^{\alpha}( \eu_1^k \nabla_k\phi -\ii\div\eu_0-\ii\eu_0^k\nabla_k\ln\eps)|_{\x_0}=0$ for $|\alpha|\leq 3$.
    \end{enumerate}
    This represents (constrained) initial data for the $\e_A$-transport equations~\eqref{eq:eAtransport} to degrees $5-2A$ along $\gamma$, for $A=0,1$. 
    
    Then, there exist smooth $\hat{\E}(t,\x)$ and $\hat{\H}(t,\x)$ of the form
    \begin{align} \label{EqehApprox}
        &\hat{\bE} = \omega^{\nicefrac{3}{4}} \Re \big[(\be_0 + \omega^{-1} \be_1  + \omega^{-2} \be_{2} )e^{\ii \omega \phi} \big], \qquad\hat{\bH} = \omega^{\nicefrac{3}{4}} \Re \big[(\h_0 + \omega^{-1} \h_1  + \omega^{-2} \h_{2} )e^{\ii \omega \phi} \big],
    \end{align}
    such that:
    \begin{enumerate}
        \item at each $t$, $\hat{\bE}(t,\x)$ and $\hat{\bH}(t,\x)$ are supported in $B_\rho(\gammu(t))$,
        \item $\Eh$ and $\Hh$ satisfy Maxwell's equations to order 2:
        \begin{align} \label{EqEstimatesApproxSol}
    \sup_{t\in[0,T]}\|\div(\eps\Eh)\|_{L^2( \R^3)}&\leq C(T)\omega^{-2},\qquad
   \sup_{t\in[0,T]}\|\underset{=\hat{\mathscr{F}}}{\underbrace{\nabla\times\Eh+\mu\partial_t\Hh}}\|_{L^2( \R^3)}\leq C(T)\omega^{-2}, \nonumber\\
     \sup_{t\in[0,T]}\|\div(\mu\Hh)\|_{L^2(\R^3)}&\leq C(T)\omega^{-2},\qquad
     \sup_{t\in[0,T]}\|\underset{=\hat{\mathscr{G}}}{\underbrace{\nabla\times\Hh-\eps\partial_t\Eh}}\|_{L^2(\R^3)}\leq C(T)\omega^{-2},
    \end{align}
    \item for $A=0,1$, we have
    \begin{align}
        D^{\alpha}\e_A|_{(0,\x_0)}&=D^{\alpha}\eu_A|_{\x_0},
    \end{align}
    for all $|\alpha|\leq 5-2A$,
    \item for $A=0,1,2$ and $|\alpha|\leq 5-2A$, $D^{\alpha}_{\x}\h_A|_{\gamma}$ satisfies \cref{eq:h012def}. In particular, for $A=0,1$, $\h_A$ satisfies the $\h_A$-transport equations~\eqref{eq:hAtransport} to degree $5-2A$ along $\gamma$ and has induced initial data
    \begin{align}
         D^{\alpha}{\h}^i_{A}|_{(0,\x_0)}=D^{\alpha}{\hu}^i_{A}|_{\x_0},
    \end{align}
    for all $|\alpha|\leq 5-2A$, where $D^{\alpha}{\hu}^i_{A}|_{\x_0}$ is defined in \cref{eq:hdefID},
    \item for all $t \geq 0$ we have $\Im(\nabla\phu)(t, x)\neq 0$ for $x \in \overline{\mathrm{supp}\big(\hat{E}(t, \cdot) \big) \cup \mathrm{supp}\big(\hat{H}(t, \cdot) \big)} \setminus\{\gammu(t)\}$. \label{LastPoint}
    \end{enumerate}
\end{theorem}

\begin{proof}
We would like to appeal to \cref{prop:order_2_from spacetime_functions}. Recall from \cref{prop:2sttranstoMax} that the conditions \ref{cond4}--\ref{cond6} of Proposition~\ref{prop:order_2_from spacetime_functions} can be satisfied if we can construct a solution to the Eikonal equation along $\gamma$ to degree $6$ and solutions to the $\e_0$ and $\e_1$-transport equations to degree $(5,3)$ along $\gamma$ (provided the constraints are satisfied initially to degrees (5,3) respectively) and such that $\e_2$ satisfies the $\e_2$-constraint along $\gamma$. This last point can be done trivially after one constructs $\e_1$.

We can use \cref{prop:construction_phase_function} to construct a solution to the Eikonal equation along $\gamma$ to degree $7$. The transport ODEs governing $\e_0$ and $\e_1$ are linear and, therefore, global existence follows from the standard ODE theory. Pick $\e_2$ such that
\begin{align}
    D^{\alpha} \big( {\e}_2^k \nabla_k\phi \big) \big|_\gamma = D^{\alpha} \big( \ii \div{\e}_1+ \ii{\e}_1^k\nabla_k\ln\eps \big) \big|_{\gamma}\qquad \forall |\alpha|\leq 1.
\end{align}
Now we can define $(\h_0,\h_1,\h_2)$ via \cref{eq:h012def}. This then completes our construction along $\gamma$ by \cref{prop:2sttranstoMax}. Using~\cref{lem:Borel}, we can  build smooth spacetime functions $(\phi,\e_0,\e_1,\e_2)$  whose derivatives along $\gamma$ agree with those constructed. 

We are now in the setting of \cref{prop:order_2_from spacetime_functions}, which completes the construction. The fact that the $h_A$-transport equations are satisfied follows directly from \cref{lem:decompGF}. The last point \ref{LastPoint} in the above theorem can be ensured by virtue of $\nabla \nabla \Im\phi (t, \gammu(t))$ being positive definite and choosing the bump function $\chi_\rho$ in \cref{prop:order_2_from spacetime_functions} to have even smaller support around $\gamma$.
\end{proof}

\begin{remark}
Note that in \cref{prop:2sttranstoMax} we require the Eikonal equation to be satisfied only to degree $6$ -- which, by  \cref{prop:construction_phase_function} determines $6$ derivatives of $\phi$ along $\gamma$ uniquely. However, when constructing $5$ derivatives  of $\e_0$ and $3$ derivatives of $\e_1$ from their transport equations, $7$ derivatives of $\phi$ along $\gamma$ enter. The reason that in the above theorem we have provided initial data for $\phi$ up to and including $7$ derivatives and required the Eikonal equation to be satisfied to degree $7$ is that in this way five derivatives of $e_0$ and three derivatives of $e_1$ along $\gamma$ are uniquely fixed by our choice of initial data. However, this is not needed in the remainder of the paper and the above theorem remains true as stated if one only prescribes six derivatives of $\phi$ and constructs a solution to the Eikonal equation to degree six.
\end{remark}

\subsection{Conservation laws}

In this section, we present the leading-order conservation laws for the approximate solutions defined above. These follow from the transport equations satisfied by $\be_0$ and $\bh_0$, and represent energy conservation at leading order, as well as a conservation of the state of polarisation.

\begin{proposition}[Conservation laws] \label{Prop:conservation} Consider the approximate solution defined in \cref{thm:construction_theorem}. Then, the following conservation laws hold:
\begin{subequations}
\begin{align}
    \left( \partial_t + \dot{\gamma}^j \nabla_j \right) \frac{\varepsilon \be_0 \cdot \overline{\be}_0}{\sqrt{\det\left( \frac{1}{2 \pi \ii} \mathcal{A} \right)}} \Bigg|_{\gamma} &= 0, \label{eq:cons_energy0} \\
    \left( \partial_t + \dot{\gamma}^j \nabla_j \right) \frac{\mu \bh_0 \cdot \overline{\bh}_0}{\sqrt{\det\left( \frac{1}{2 \pi \ii} \mathcal{A} \right)}} \Bigg|_{\gamma} &= 0, \\
    \left( \partial_t + \dot{\gamma}^j \nabla_j \right) \frac{\n \be_0 \cdot \overline{\bh}_0}{\sqrt{\det\left( \frac{1}{2 \pi \ii} \mathcal{A} \right)}} \Bigg|_{\gamma} &= 0,
\end{align}    
\end{subequations}
where $\mathcal{A}_{a b}(t,x) = 2 \ii \nabla_a \nabla_b \Im \phi(t,x)$ and $A_{a b} (t)= \mathcal{A}|_{\gamma(t)}$. 
\end{proposition}

\begin{proof}
We focus on the first conservation law in \cref{eq:cons_energy0}. We have
\begin{align} \label{eq:d_energy}
    \Big( \partial_t + \dot{\gamma}^j \nabla_j \Big) \frac{\varepsilon \be_0 \cdot \overline{\be}_0}{\sqrt{\det\left( \frac{1}{2 \pi \ii} \mathcal{A} \right)}} \Bigg|_{\gamma} &= \frac{1}{\sqrt{\det\left( \frac{1}{2 \pi \ii} A \right)}} \bigg[ \Big( \partial_t + \dot{\gamma}^j \nabla_j \Big) \big( \varepsilon \be_0 \cdot \bar{\be}_0 \big) - \frac{\varepsilon \be_0 \cdot \bar{\be}_0}{2} (A^{-1})^{i j} \Big( \partial_t + \dot{\gamma}^j \nabla_j \Big) \mathcal{A}_{i j} \bigg] \bigg|_{\gamma}.
\end{align}
Using the transport equation \eqref{eq:e0transport}, the first term on the right-hand side of the above equation is
\begin{align}
    \Big( \partial_t + \dot{\gamma}^j \nabla_j \Big) \big( \varepsilon \be_0 \cdot \bar{\be}_0 \big) \Big|_{\gamma} &= \frac{1}{\n^2 \dot{\phi}} \Big[ \Re\big( \Delta \phi - \n^2 \ddot{\phi} \big) - (\nabla^i \phi) \nabla_i \ln \n^2 \Big] \varepsilon \be_0 \cdot \overline{\be}_0 \Big|_{\gamma} \nonumber \\
    &= \frac{1}{\n^2 \dot{\phi}} \bigg\{ \Re \bigg[ \Delta \phi - \frac{1}{\n^2 \dot{\phi}^2} (\nabla^i \phi) (\nabla^j \phi) \nabla_i \nabla_j \phi \bigg] - (\nabla^i \phi) \nabla_i \ln \n \bigg\} \varepsilon \be_0 \cdot \overline{\be}_0 \bigg|_{\gamma}.
\end{align}
The second equality in the above equation follows by replacing $\ddot{\phi}|_\gamma = \frac{\nabla^i \phi}{\n^2 \dot{\phi}} \nabla_i \dot{\phi} |_\gamma$, which comes from \cref{eq:dt_eikonal}, and by using \cref{eq:deg1_intermetiate} to replace $\nabla_i \dot{\phi} |_\gamma$. 

The second term on the right-hand side of \cref{eq:d_energy} can be calculated by taking the imaginary part of the Riccati equation \eqref{eq:Riccati}. Note that we have
\begin{equation}
    \Big( \partial_t + \dot{\gamma}^j \nabla_j \Big) \mathcal{A}_{i j} \Big|_{\gamma(t)} = \frac{d}{dt} A_{i j}(t). 
\end{equation}
We immediately see that the two terms on the right-hand side of \cref{eq:d_energy} cancel, and we obtain the conservation law in \cref{eq:cons_energy0}. The proofs of the other conservation laws follow identically if we also use the corresponding transport law for $\bh_0$ given in \cref{eq:hAtransport}.
\end{proof}

\subsection{The stationary phase approximation for the approximate solutions}

In this section, we show how the stationary phase approximation can be used to expand the integrals that define the total energy, energy centroid, total linear momentum, total angular momentum, and quadrupole moment corresponding to the approximate Gaussian beam solutions constructed in \cref{thm:construction_theorem}.

Consider the approximate solutions $\Eh$ and $\Hh$ constructed in \cref{EqehApprox}. Then, the corresponding energy density and Poynting vector are
\begin{subequations} \label{eq:approx_density}
\begin{align}
    \hat{\mathcal{E}} &= \frac{\omega^{\nicefrac{3}{2}}}{4} \left( \varepsilon \be \cdot \Bar{\be} + \mu \bh \cdot \Bar{\bh} \right) e^{-2 \omega \Im \phi} + \frac{\omega^{\nicefrac{3}{2}}}{4} \Re \left[ \left( \varepsilon \be \cdot \be + \mu \bh \cdot \bh \right) e^{2 \ii  \omega \Re \phi} \right] e^{-2 \omega \Im \phi}, \\
    \hat{\bS} &= \frac{\omega^{\nicefrac{3}{2}}}{2} \n^2 \Re \left( \be \times \Bar{\bh} \right) e^{-2\omega  \Im \phi} + \frac{\omega^{\nicefrac{3}{2}}}{2} \n^2 \Re \left[ \left( \be \times \bh \right) e^{2 \ii \omega \Re \phi} \right] e^{-2\omega  \Im \phi},
\end{align}    
\end{subequations}
where $\be = \sum_{A=0}^2 \omega^{-A} \be_A$ and $\bh = \sum_{A=0}^2 \omega^{-A} \bh_A$.
The quantities
\begin{subequations}
    \begin{align}
    \hat{\bbE}(t)&:=\int_{\R^3}\hat{\mathcal{E}}(t,x)d^3x,\qquad
    &\hat{\bbX}^i(t)&:=\frac{1}{\hat{\bbE}}\int_{\R^3}x^i\hat{\mathcal{E}}(t,x)d^3x,\\
    \hat{\bbP}_i(t)&:=\int_{\R^3}\hat{\mathcal{S}}_i(t,x)d^3x,\qquad
    &\hat{\bbJ}_i(t)&:=\int_{\R^3}\eps_{ijk}\hat{r}^j(t,x)\hat{\mathcal{S}}^k(t,x)d^3x,\\
    &
    &\hat{\bbQ}^{ij}(t)&:=\int_{\R^3}\hat{r}^i(t,x)\hat{r}^j(t,x)\hat{\mathcal{E}}(t,x)d^3x,
\end{align}
\end{subequations}
where $\hat{r}^i(t,x):=x^i-\hat{\bbX}^i(t)$, are the energy, energy centroid, total linear momentum, total angular momentum,  and quadrupole moment of the approximate solutions. These can be approximated using the stationary phase method \cite[Sec. 7.7]{hormander1983analysis}, which is reviewed in \cref{app:Stationary}. In particular, by applying \cite[Th. 7.7.1]{hormander1983analysis}, it follows that the integrals of the above terms proportional to $e^{\pm 2 \ii \omega \Re \phi}$ decay to an arbitrarily high order in $\omega$. The integrals of the remaining terms can be approximated using \cref{th:stationary} \cite[Th. 7.7.5]{hormander1983analysis} and are of the form
\begin{equation}
    \int_{\R^3} u(x) e^{\ii \omega f(x)} \, d^3 x = \frac{e^{\ii \omega f(x_s)}}{\sqrt{\det\left( \frac{\omega}{2 \pi \ii} A \right)}} \sum_{j<k} \omega^{-j} L_j u (x_s) + \mathcal{O}(\omega^{-k}),
\end{equation}
where $A_{a b} = \nabla_a \nabla_b f(x_s)$, $f(x) = 2 \ii \Im \phi$, $x_s = \underline{\gamma}(t)$, and all the assumptions of \cref{th:stationary} are satisfied.

\begin{proposition} \label{prop:approx_average_quantities}
Consider the approximate solution given in \cref{EqehApprox}, together with the corresponding energy density and Poynting vector given in \cref{eq:approx_density}. Let $T>0$. Then, for $t \in [0,T]$ and up to error terms of order $\omega^{-2}$, the corresponding total energy, energy centroid, total linear momentum, total angular momentum, and quadrupole moment of the approximate solution are 
\begin{subequations}
\begin{align}
    \hat{\mathbb{E}}(t) &= \frac{1}{\sqrt{\det\left( \frac{1}{2 \pi \ii} A \right)}} \Big( u + \omega^{-1} L_1 u \Big) \Big|_{\gamma(t)} + \mathcal{O}(\omega^{-2}),\label{eq:E_hat}  \\ 
    \hat{\mathbb{X}}^i(t) &= \gamma^i(t) + \frac{\ii \omega^{-1}}{\hat{\mathbb{E}} \sqrt{\det\left( \frac{1}{2 \pi \ii} A \right)}} (A^{-1})^{i a} \Big[ \nabla_a u - \ii u (A^{-1})^{b c} \nabla_a \nabla_b \nabla_c \Im \phi \Big] \Big|_{\gamma(t)} + \mathcal{O}(\omega^{-2}), \label{eq:X_hat}\\
    \hat{\mathbb{P}}_i(t) &= \frac{1}{\sqrt{\det\left( \frac{1}{2 \pi \ii} A \right)}} \Big( v_i + \omega^{-1} L_1 v_i \Big) \Big|_{\gamma(t)} + \mathcal{O}(\omega^{-2}),\label{eq:P_hat} \\
    \hat{\mathbb{J}}_i (t) &= \epsilon_{i j k} \hat{r}^j \hat{\mathbb{P}}^k + \frac{\ii \omega^{-1}}{\sqrt{\det\left( \frac{1}{2 \pi \ii} A \right)}} \epsilon_{i j k} (A^{-1})^{j a} \Big[ \nabla_a v^k - \ii v^k (A^{-1})^{b c}  \nabla_a \nabla_b \nabla_c \Im \phi \Big] \Big|_{\gamma(t)} + \mathcal{O}(\omega^{-2}), \label{eq:hat_J}\\
    \hat{\mathbb{Q}}^{i j}(t) &= \ii \omega^{-1} \hat{\mathbb{E}} (A^{-1})^{i j} + \mathcal{O}(\omega^{-2}),\label{eq:Q_hat}
\end{align}
\end{subequations}
where $L_1$ is the differential operator defined in \cref{eq:L_def} with $f(x) = 2\ii \Im \phi$, $A_{i j} = A_{i j}(t) = 2\ii \nabla_i \nabla_j \Im \phi|_{\gamma(t)}$, $\hat{r}^j(t) = \hat{r}^j(t, \gammu(t)) = \gamma^j(t) - \hat{\mathbb{X}}^j(t)$ and  
\begin{subequations}
\begin{align}
    u &= \frac{1}{4} \big( \varepsilon \be \cdot \Bar{\be} + \mu \bh \cdot \Bar{\bh} \big) = \frac{1}{4} \big( \varepsilon \be_0 \cdot \Bar{\be}_0 + \mu \bh_0 \cdot \Bar{\bh}_0 \big) + \frac{\omega^{-1}}{2} \Re \big( \varepsilon \be_0 \cdot \Bar{\be}_1 + \mu \bh_0 \cdot \Bar{\bh}_1 \big) + \mathcal{O}_\infty(\omega^{-2}), \label{eq:u_def} \\
    v &= \frac{\n^2}{2} \Re \big( \be \times \Bar{\bh} \big) = \frac{\n^2}{2} \Re \big( \be_0 \times \Bar{\bh}_0 \big) + \frac{\omega^{-1} \n^2}{2} \Re \big( \be_0 \times \Bar{\bh}_1 + \be_1 \times \Bar{\bh}_0 \big) + \mathcal{O}_\infty(\omega^{-2}) \label{eq:v_def}.
\end{align}
\end{subequations}
The constants in the error terms depend on $T$ and on the approximate solution in \cref{EqehApprox}. The notation $\mathcal{O}_\infty$ indicates here that the error bounds also hold after taking finitely many derivatives of the expression, where the exact constant then also depends on the number of derivatives taken.
\end{proposition}

\begin{proof}
 The integrals of the terms in \cref{eq:approx_density} proportional to $e^{\pm 2 \ii \omega \Re \phi}$ decay to arbitrary high order in $\omega$ by \cite[Th. 7.7.1]{hormander1983analysis}. For the integrals of the remaining terms in \cref{eq:approx_density}, we apply \cref{th:stationary}, which gives the above expressions.    
\end{proof}

In particular, we note that at leading order the linear momentum is
\begin{equation} \label{eq:P_leading_order}
    \hat{\bbP}_i = \hat{\bbE} \n \frac{\nabla_i \phi}{|\nabla \phi|} \bigg|_{\gamma(t)} + \mathcal{O}(\omega^{-1}) = - \frac{\hat{\bbE}}{\dot{\phi}} \nabla_i \phi \bigg|_{\gamma(t)} + \mathcal{O}(\omega^{-1}). 
\end{equation}
This follows by using \cref{eq:h0def} to express $\bh_0|_{\gamma}$ in terms of $\be_0|_{\gamma}$, and the eikonal equation \eqref{eq:Eikonal} with $j_\phi = 0$ which gives $\dot{\phi}|_{\gamma} = - \frac{|\nabla\phi|}{\n}|_{\gamma}$. Furthermore, note that $\dot{\phi}|_{\gamma}$ is constant by construction, as given in \cref{eq:phase_construction}.

Next, we analyse the relation between the state of polarisation and the angular momentum carried by the wave packet. To see this, we decompose the total angular momentum into components parallel and orthogonal to $\nabla_i \phi |_{\gamma}$.

\begin{proposition} \label{prop:J_decomposition}
In the setting of Proposition \ref{prop:approx_average_quantities} the total angular momentum $\hat{\bbJ}$ given in \cref{eq:hat_J} can be decomposed as
\begin{equation} \label{eq:J_split}
    \hat{\mathbb{J}}_i = (\hat{\mathbb{J}}_{\parallel}) \frac{\nabla_i \phi}{|\nabla \phi|} \bigg|_{\gamma} + \epsilon_{i a b} (\hat{\mathbb{J}}_{\perp})^a \frac{\nabla^b \phi}{|\nabla \phi|} \bigg|_{\gamma}\;,
\end{equation}
where
\begin{subequations}
\begin{align}
    (\hat{\mathbb{J}}_{\parallel}) &= \omega^{-1} \frac{\hat{\mathbb{E}}}{\dot{\phi}} \bigg[ s + \ii (A^{-1})^{j a} B\indices{_a^i} \epsilon_{i j k} \frac{\nabla^k \phi}{|\nabla \phi|} \bigg] \bigg|_{\gamma} + \mathcal{O}(\omega^{-2}), \\
    (\hat{\mathbb{J}}_{\perp})^a &= \omega^{-1} \frac{\ii \hat{\bbE}}{\dot{\phi}} \bigg[ \frac{\nabla^c \phi}{|\nabla \phi|} (A^{-1})\indices{_c^b} B\indices{_b^a} - |\nabla \phi| (A^{-1})^{a b}\nabla_b \ln \n \bigg] \bigg|_{\gamma} +  \mathcal{O}(\omega^{-2}), 
\end{align}
$B_{a b} = \nabla_a \nabla_b \Re \phi$ and
\begin{equation}
    s = \frac{\ii \n \be_0 \cdot \overline{\bh}_0}{\eps \be_0 \cdot \overline{\be}_0} \bigg|_{\gamma} = \mathrm{const.} \in [-1, 1].
\end{equation}
\end{subequations}

\end{proposition}

\begin{proof}
The component of the angular momentum in the direction of $\nabla_i \phi |_{\gamma}$ can be obtained from \cref{eq:hat_J} as
\begin{align} \label{eq:J.P}
    (\hat{\mathbb{J}}_{\parallel}) = \hat{\mathbb{J}}_{i} \frac{\nabla^{i} \phi}{|\nabla \phi|} \bigg|_{\gamma}
    &= \frac{\ii \omega^{-1}}{|\nabla \phi| \sqrt{\det\left( \frac{1}{2 \pi \ii} A \right)}} \epsilon_{i j k} (\nabla^i \phi) (A^{-1})^{j a} \nabla_a v^k \bigg|_{\gamma} + \mathcal{O}(\omega^{-2}) \nonumber \\
    &= \frac{\ii \omega^{-1} \n^2}{2 |\nabla \phi| \sqrt{\det\left( \frac{1}{2 \pi \ii} A \right)}} (\nabla^i \phi) (A^{-1})^{j a} \Re \Big( \overline{\bh}_{0 j} \nabla_a \be_{0 i} - \be_{0 j} \nabla_a \overline{\bh}_{0 i} \Big) \bigg|_{\gamma} + \mathcal{O}(\omega^{-2}) \nonumber \\
    &= \omega^{-1} \frac{\hat{\mathbb{E}}}{\dot{\phi}} \Bigg[  \frac{\ii \n \be_0 \cdot \overline{\bh}_0}{ \eps \be_0 \cdot \overline{\be}_0 } + 2 \ii (A^{-1})^{j a} B\indices{_a^i} \frac{\Re \left( \n \be_{0 [i} \overline{\bh}_{0 j]} \right) }{ \eps \be_0 \cdot \overline{\be}_0 } \Bigg] \Bigg|_{\gamma} +  \mathcal{O}(\omega^{-2}) \nonumber \\
    &= \omega^{-1} \frac{\hat{\mathbb{E}}}{\dot{\phi}} \bigg[ s +  \ii (A^{-1})^{j a} B\indices{_a^i} \epsilon
    _{i j k} \frac{\nabla^{k} \phi}{|\nabla \phi|} \bigg] \bigg|_{\gamma} + \mathcal{O}(\omega^{-2}).
\end{align}
The first line follows from \cref{eq:hat_J}, together with the fact that $v_i|_\gamma \propto \nabla_i \phi |_\gamma + \mathcal{O}(\omega^{-1})$. The second line is obtained using the definition of $v$, together with the orthogonality relations $\be_0^i \nabla_i \phi |_\gamma = 0 = \bh_0^i \nabla_i \phi |_\gamma$. To obtain the third line, we used $\nabla_a (\be_0^i \nabla_i \phi)|_\gamma = 0 = \nabla_a (\bh_0^i \nabla_i \phi)|_\gamma$ by construction of the approximate solution, and we split $\nabla_i \nabla_j \phi = B_{i j} + \frac{1}{2} A_{i j}$. Finally, the fourth line follows by fixing an orthonormal frame $\Big( \frac{\nabla^i \phi}{|\nabla \phi|} \Big|_\gamma, X, Y \Big)$, where $X$ and $Y$ are real vectors. Then, to satisfy the constraint $e_0^i \nabla_i \phi |_\gamma = 0$, we can generally parametrise $\be_0$ as
\begin{equation} \label{eq:e0_basis}
    \be_0 \big|_{\gamma} = \mathfrak{a}(t) \big[ z_1(t) m(t) + z_2(t) \overline{m}(t) \big], 
\end{equation}
where $m = \frac{1}{\sqrt{2}}(X - \ii Y)$, $\mathfrak{a}(t)$ is a strictly positive real scalar function and $z_{1,2}(t)$ are complex scalar functions that satisfy $|z_1|^2 + |z_2|^2 = 1$. Then, using \cref{eq:h0def} to express $\bh_0|_\gamma$ in terms of $\be_0|_\gamma$, we obtain
\begin{subequations}
\begin{align}
    \frac{\ii \n \be_0 \cdot \overline{\bh}_0}{\eps \be_0 \cdot \overline{\be}_0} \Bigg|_{\gamma} &= |z_1|^2 - |z_2|^2 = s \in [-1, 1], \\
    \frac{\Re \left( \n \be_{0 [i} \overline{\bh}_{0 j]} \right) }{ \eps \be_0 \cdot \overline{\be}_0 } \Bigg|_{\gamma} &= \epsilon_{i j k} \frac{\nabla^k \phi}{|\nabla \phi|} \bigg|_{\gamma}.
\end{align}
\end{subequations}
The conservation of $s$ follows by applying \cref{Prop:conservation}:
\begin{equation}
    \dot{s} = \left( \partial_t + \dot{\gamma}^j \nabla_j \right) \frac{\ii \n \be_0 \cdot \overline{\bh}_0}{\eps \be_0 \cdot \overline{\be}_0} \bigg|_{\gamma} = \left( \partial_t + \dot{\gamma}^j \nabla_j \right) \frac{ \frac{\ii \n}{\sqrt{\det\left( \frac{1}{2 \pi \ii} A \right)}} \be_0 \cdot \overline{\bh}_0}{ \frac{\eps}{\sqrt{\det\left( \frac{1}{2 \pi \ii} A \right)}} \be_0 \cdot \overline{\be}_0 } \Bigg|_{\gamma} = 0.
\end{equation}

The component of angular momentum in directions orthogonal to $\nabla_i \phi |_\gamma$, called transverse angular momentum \cite{BLIOKH20151}, is determined by the vector
\begin{align}
    (\hat{\mathbb{J}}_{\perp})_i &= \epsilon_{i a b} \frac{\nabla^a \phi}{|\nabla \phi|} \hat{\mathbb{J}}^b \bigg|_{\gamma} \nonumber\\
    &= \frac{2 \ii \omega^{-1}}{|\nabla \phi| \sqrt{\det\left( \frac{1}{2 \pi \ii} A \right)} } \delta_i^{[c} (\nabla^{d]} \phi) (A^{-1})\indices{_c^a} \bigg( \nabla_a v_d - \frac{1}{u} v_d \nabla_a u \bigg) \bigg|_{\gamma} + \mathcal{O}(\omega^{-2}), \nonumber\\
    &= - \frac{2 \ii \omega^{-1} \hat{\bbE}}{|\nabla \phi| \dot{\phi}} \delta_i^{[c} (\nabla^{d]} \phi) (A^{-1})\indices{_c^a} \bigg[ B_{a d} - \frac{1}{|\nabla \phi|^2} (\nabla_d \phi) (\nabla^j \phi) B_{a j} + (\nabla_d \phi) \nabla_a \ln \n \bigg]\bigg|_{\gamma} + \mathcal{O}(\omega^{-2}) \nonumber \\
    &= \frac{\ii \omega^{-1} \hat{\bbE}}{|\nabla \phi| \dot{\phi}} \bigg[ (\nabla^b \phi) (A^{-1})\indices{_b^a} B_{a i} - |\nabla \phi|^2 (A^{-1})\indices{_i^a} \nabla_a \ln \n \bigg] \bigg|_{\gamma} \nonumber\\
    &\qquad \qquad - \frac{\ii \omega^{-1} \hat{\bbE}}{|\nabla \phi| \dot{\phi}} (\nabla_i \phi) (\nabla^c \phi) (A^{-1})\indices{_c^a} \bigg[ B\indices{_a^j} \frac{\nabla_j \phi}{|\nabla \phi|^2} - \nabla_a \ln \n \bigg] \bigg|_{\gamma} + \mathcal{O}(\omega^{-2}).
\end{align}
In the above equation, the first line follows from \cref{eq:X_hat,eq:hat_J}, the second line follows from expanding the derivatives of $v$ and $u$, using the constraints $\be_0^i \nabla_i \phi |_\gamma = 0 = \bh_0^i \nabla_i \phi |_\gamma$, and $\nabla_a (\be_0^i \nabla_i \phi)|_\gamma = 0 = \nabla_a (\bh_0^i \nabla_i \phi)|_\gamma$, as well as \cref{eq:h0def}. The final line follows by simply rearranging some of the previous terms. Note that the term in the last line is proportional to $\nabla_i \phi|_\gamma$, so we can drop it as it will not contribute to \cref{eq:J_split} due to the cross product. 
\end{proof}

The total angular momentum $\hat{\bbJ}_i$ consists of two longitudinal terms and a transverse term. The longitudinal term proportional to $s$ is called spin angular momentum and is determined by the state of polarisation of $\be_0|_\gamma$. In particular, we have $s = \pm 1$ for circular polarisation ($z_1 = 0$ or $z_2 = 0$ in \cref{eq:e0_basis}), $s = 0$ for linear polarisation, and $s \in (-1,1) \setminus \{0\}$ for elliptical polarisation \cite{born1980,kravtsov1990}. The other two terms are called the intrinsic longitudinal and transverse orbital angular momentum \cite{AM_Light,BLIOKH20151,PhysRevA.70.013809,Plachenov:23} and are determined by $\nabla_a \nabla_b \phi |_\gamma$.

\section{Construction of a one-parameter family of initial data} \label{SecConstructionID}

In this section, we construct compactly supported Gaussian beam initial data for Maxwell's equations by correcting the initial data for the approximate solution so that the constraint equations are satisfied exactly. Thus, the class of initial data introduced in \cref{DefGBID} is non-empty. We also show that two sets of Gaussian beam initial data with sufficiently matching phase and amplitude jets are equivalent up to $\mathcal{O}_{L^2(\R^3)}(\omega^{-2})$.

\begin{theorem}\label{ThmConstructionMaxwellID}
    Consider the setting of Theorem \ref{thm:construction_theorem}. Then, there exists a one-parameter family of smooth initial data $\Eu(x; \omega)$, $\Hu(x; \omega) $ for Maxwell's equations \eqref{EqMax} of the form
\begin{subequations} \label{eq:ansatz}
\begin{align} 
    \Eu(x; \omega) &= \omega^{\nicefrac{3}{4}} \Re \Big\{ \big[ \be_0 (0,x) + \omega^{-1} \be_1 (0,x) \big] e^{\ii \omega \phi (0,x)} \Big\} + \mathcal{O}_{L^2(\R^3)}(\omega^{-2}),  \\
    \Hu (x; \omega) &= \omega^{\nicefrac{3}{4}} \Re \Big\{ \big[ \bh_0 (0,x) + \omega^{-1} \bh_1 (0,x) \big]  e^{\ii \omega \phi (0,x)}  \Big\} + \mathcal{O}_{L^2(\R^3)}(\omega^{-2}),
\end{align}        
\end{subequations}
    where  $\be_A(0,x)$, $\bh_A(0,x)$ for $A  \in \{0,1\}$, and $\phi(0,x)$ are as in \cref{EqehApprox}, and such that $\mathrm{supp} \big(\Eu(\cdot; \omega)\big) \cup \mathrm{supp}\big( \Hu(\cdot; \omega)\big) \subseteq B_\rho(x_0)$ for all $\omega >1$, with $0< \rho <1$ as in \cref{thm:construction_theorem}. In particular, this constitutes $\mathcal{K}$-supported Gaussian beam initial data of order $2$ with $\mathcal{K} = B_\rho(x_0)$.
\end{theorem}

In particular, this theorem shows that the class of Gaussian beam initial data given in \cref{DefGBID} is non-empty.
 
\begin{proof}
The main point to prove is that one can perturb the initial data $\hat{\bE}(0,x)$ and $\hat{\bH}(0,x)$ induced by \cref{EqehApprox} by compactly supported functions $\bE_{cor}(x)$, $\bH_{cor}(x)$ in $\mathcal{O}_{L^2(\R^3)}(\omega^{-2})$ such that Maxwell's constraint equations \eqref{EqMaxCE} and \eqref{EqMaxCH} are satisfied by $\Eu(x) := \hat{\bE}(0,x) + \bE_{cor}(x)$ and $\Hu(x) := \hat{\bH}(0,x) + \bH_{cor}(x)$.

    In order for $\Eu$ and $\Hu$ to solve the constraint equations \eqref{EqMaxCE} and \eqref{EqMaxCH}, the correction terms must solve
    \begin{subequations} \label{EqSolvingConstraint}
        \begin{align} 
            \nabla \cdot ( \varepsilon \bE_{cor}) &= - \nabla \cdot ( \varepsilon \hat{\bE}|_{t = 0}), \\
            \nabla \cdot ( \mu \bH_{cor}) &= - \nabla \cdot ( \mu \hat{\bH}|_{t = 0}) .
        \end{align}
    \end{subequations}
     Recall that $\hat{\bE}(0,x)$ and $\hat{\bH}(0,x)$ are supported in $B_\rho(x_0)$. Equation \eqref{EqSolvingConstraint} can now be solved using Bogovskii's operator \cite{Bogovskii}. Specifically, we use Lemma III.3.1 in~\cite{Galdi} (note that the compatibility conditions $\int_{B_\rho(x_0)}\nabla \cdot ( \varepsilon \hat{\bE}|_{t = 0}) \, d^3x = 0 = \int_{B_\rho(x_0)}\nabla \cdot ( \mu \hat{\bH}|_{t = 0}) \, d^3x $ are trivially satisfied) to obtain $\varepsilon \bE_{cor}$, $\mu \bH_{cor} \in C^\infty_0(B_\rho(x_0))$ with 
     \begin{subequations} \label{EqEstCor}
     \begin{align} 
     \|\varepsilon \bE_{cor}\|_{H^1(B_\rho(x_0))} &\leq C_\rho \| \nabla \cdot ( \varepsilon \hat{\bE}|_{t = 0})\|_{L^2(B_\rho(x_0))} = \mathcal{O}(\omega^{-2}), \label{EqEstECor} \\
     \|\mu \bH_{cor}\|_{H^1(B_\rho(x_0))} &\leq C_\rho \| \nabla \cdot ( \mu \hat{\bH}|_{t = 0})\|_{L^2(B_\rho(x_0))} = \mathcal{O}(\omega^{-2}), \label{EqEstHCor}
     \end{align}
     \end{subequations}
     where we have used \cref{EqEstimatesApproxSol} and the constant $C_\rho >0$ depends only on $0<\rho<1$. Since $\varepsilon$ and $\mu$ are smooth positive functions, this allows us to divide by $\varepsilon$ and $\mu$ to  define smooth $\bE_{cor}$ and $\bH_{cor}$, and thus also $\Eu$ and $\Hu$. The bound $\mathcal{O}_{L^2(\R^3)}$ in \cref{eq:ansatz} on the correction terms $\bE_{cor}$, $\bH_{cor}$ now follows from \cref{EqEstCor}. Moreover, by construction, we have $\mathrm{supp} \big(\Eu(\cdot, \omega)\big) \cup \mathrm{supp}\big( \Hu(\cdot, \omega)\big) \subseteq B_\rho(x_0)$ for all $\omega >1$ and Maxwell's constraint equations \eqref{EqMaxCE} and \eqref{EqMaxCH} are satisfied.
 This in particular shows point \ref{ID_def_point3} in \cref{DefGBID}. Points \ref{ID_def_point1} and \ref{ID_def_point2} in \cref{DefGBID} follow directly from \cref{thm:construction_theorem}. 
\end{proof}

The following proposition states under what conditions the constructed initial data in \cref{ThmConstructionMaxwellID} is equivalent\footnote{By this, we mean that it `differs by a function $\mathcal{O}_{L^2(\R^3)}(\omega^{-2})$'.} to general Gaussian beam initial data of order $2$ as in \cref{DefGBID}.

\begin{proposition} \label{PropEquivGBData}
Let $\x_0 \in \R^3$ and  $\overset{\scriptscriptstyle (a)}{\mathcal{K}} \subseteq \R^3$ be precompact open neighbourhoods of $\x_0$ for $a \in \{0, 1\}$. Consider two one-parameter families of vector fields $\Re \overset{\scriptscriptstyle (a)}{\underline{\pmb{\mathfrak{E}}}}(x; \omega)$ with
\begin{align}
    \overset{\scriptscriptstyle (a)}{\underline{\pmb{\mathfrak{E}}}}(x; \omega) = \omega^{\nicefrac{3}{4}} \overset{\scriptscriptstyle (a)}{\eu}(x) e^{\ii \omega \overset{\scriptscriptstyle (a)}{\phu}(x)} =\omega^{\nicefrac{3}{4}} \Big[ \overset{\scriptscriptstyle (a)}{\eu}_0(x) + \omega^{-1} \overset{\scriptscriptstyle (a)}{\eu}_1(x)\Big] e^{\ii \omega \overset{\scriptscriptstyle (a)}{\phu}(x)} ,
\end{align}
for $a \in \{0, 1\}$, where
\begin{enumerate}
    \item  $\overset{\scriptscriptstyle (a)}{{\phu}} \in C^\infty (\R^3, \C)$, with $\Im \overset{\scriptscriptstyle (a)}{{\phu}} \geq 0$ and $\Im \overset{\scriptscriptstyle (a)}{{\phu}} \big|_{\x_0} = 0$, $\nabla_i \Im\overset{\scriptscriptstyle (a)}{{\phu}} \big|_{\x_0} = 0$ for $i = 1,2,3$, $ \Im \nabla_i \nabla_j \overset{\scriptscriptstyle (a)}{{\phu}} \big|_{\x_0}$ is a positive definite matrix, and $\nabla \Im  \overset{\scriptscriptstyle (a)}{\phu} \neq 0$ in $\cl(\overset{\scriptscriptstyle (a)}{\mathcal{K}}) \setminus \{\x_0\}$ for $a \in \{0, 1\}$. \label{matching_assumptions_1}
    \item $\overset{\scriptscriptstyle (a)}{\eu}_A \in C^\infty_0(\overset{\scriptscriptstyle (a)}{\mathcal{K}}, \C^3)$ for $a \in \{0, 1\}$ and $A \in \{0, 1\}$.
    \item $D^\alpha \overset{\scriptscriptstyle (0)}{\phu} \big|_{\x_0} = D^\alpha \overset{\scriptscriptstyle (1)}{\phu} \big|_{\x_0} $ for $0 \leq |\alpha| \leq 5$. \label{matching3}
    \item $D^\alpha \overset{\scriptscriptstyle (0)}{\eu}_0 \big|_{\x_0} = D^\alpha \overset{\scriptscriptstyle (1)}{\eu}_0 \big|_{\x_0} $ for $0 \leq |\alpha| \leq 3$.
    \item $D^\alpha \overset{\scriptscriptstyle (0)}{\eu}_1 \big|_{\x_0} = D^\alpha \overset{\scriptscriptstyle (1)}{\eu}_1 \big|_{\x_0} $ for $0 \leq |\alpha| \leq 1$. \label{matching5}
\end{enumerate}
Then, we have $|| \Re \overset{\scriptscriptstyle (0)}{\underline{\pmb{\mathfrak{E}}}}( \cdot; \omega) - \Re \overset{\scriptscriptstyle (1)}{\underline{\pmb{\mathfrak{E}}}}( \cdot; \omega)||_{L^2(\R^3)} = \mathcal{O}(\omega^{-2})$.
\end{proposition}

\begin{proof}
We note that $|| \Re \overset{\scriptscriptstyle (0)}{\underline{\pmb{\mathfrak{E}}}}( \cdot; \omega) - \Re \overset{\scriptscriptstyle (1)}{\underline{\pmb{\mathfrak{E}}}}( \cdot; \omega)||_{L^2(\R^3)} \leq || \overset{\scriptscriptstyle (0)}{\underline{\pmb{\mathfrak{E}}}}( \cdot; \omega) - \overset{\scriptscriptstyle (1)}{\underline{\pmb{\mathfrak{E}}}}( \cdot; \omega)||_{L^2(\R^3)}$. Therefore, it is enough to show that 
\begin{equation}
    || \overset{\scriptscriptstyle (0)}{\underline{\pmb{\mathfrak{E}}}}( \cdot; \omega) - \overset{\scriptscriptstyle (1)}{\underline{\pmb{\mathfrak{E}}}}( \cdot; \omega)||_{L^2(\R^3)}^2 = \mathcal{O}(\omega^{-4}).
\end{equation}
Similarly to \cref{eq:imphi_quadratic_lower_bound,eq:monotonicity_est}, based on the assumptions on the phase functions given in point \ref{matching_assumptions_1}, there exist constants $c > 0$ and $\rho > 0$ such that
\begin{align} \label{eq:exp_est}
    \Im \overset{\scriptscriptstyle (a)}{{\phu}}(x) \geq \frac{c}{2}|x - x_0|^2, \qquad
    e^{-2 \omega \Im \overset{\scriptscriptstyle (a)}{{\phu}}(x)} \leq e^{- \omega c |x - x_0|^2} \qquad \forall x \in B_\rho(x_0).
\end{align}
Furthermore, we can also define the strictly positive constants\footnote{The fact that they are strictly positive follows from $\Im \overset{\scriptscriptstyle (a)}{{\phu}} \geq 0$ together with $\nabla \Im  \overset{\scriptscriptstyle (a)}{\phu} \neq 0$ in $\cl(\overset{\scriptscriptstyle (a)}{\mathcal{K}}) \setminus \{\x_0\}$.}
\begin{equation}
    \overset{\scriptscriptstyle (a)}{{m}} = \inf_{x \in \overset{\scriptscriptstyle (a)}{\mathcal{K}}\setminus B_\rho(x_0)} \Im \overset{\scriptscriptstyle (a)}{{\phu}},
\end{equation}
so that we have
\begin{equation}
    |e^{\ii \omega \overset{\scriptscriptstyle (a)}{{\phu}}}| = e^{-\omega \Im \overset{\scriptscriptstyle (a)}{{\phu}}} \leq e^{-\omega \overset{\scriptscriptstyle (a)}{{m}}} \qquad \forall x \in \overset{\scriptscriptstyle (a)}{\mathcal{K}}\setminus B_\rho(x_0).
\end{equation}
Next, we introduce the following notation:
\begin{align} 
    \delta \phu = \overset{\scriptscriptstyle (0)}{{\phu}} - \overset{\scriptscriptstyle (1)}{{\phu}}, \qquad \delta \eu_0 = \overset{\scriptscriptstyle (0)}{\eu}_0 - \overset{\scriptscriptstyle (1)}{\eu}_0 , \qquad \delta\eu_1 = \overset{\scriptscriptstyle (0)}{\eu}_1 - \overset{\scriptscriptstyle (1)}{\eu}_1, \qquad \delta \eu =  \overset{\scriptscriptstyle (0)}{\eu} - \overset{\scriptscriptstyle (1)}{\eu} = \delta \eu_0 + \omega^{-1} \delta \eu_1.
\end{align}
Then, using assumptions \ref{matching3} to \ref{matching5} and Taylor's theorem, we have
\begin{equation} \label{eq:Taylor_est}
    |\delta \phu(x)| \leq C_{\phu} r^6, \qquad |\delta \eu_0(x)| \leq C_0 r^4, \qquad |\delta \eu_1(x)| \leq C_1 r^2,
\end{equation}
where $r = |x - x_0|$, $C_{\phu}$, $C_0$, and $C_1$ are constants, and $x \in B_\rho(x_0)$. We can also use these relations to write
\begin{equation}
    |\delta \eu(x)| \leq C_2 (r^4 + \omega^{-1} r^2),
\end{equation}
where $C_2 = \max(C_0, C_1)$ and $x \in B_\rho(x_0)$.

Based on this, we can perform the following pointwise estimates for $x \in B_\rho(x_0)$. We have
\begin{align}
    \overset{\scriptscriptstyle (0)}{\underline{\pmb{\mathfrak{E}}}} - \overset{\scriptscriptstyle (1)}{\underline{\pmb{\mathfrak{E}}}} = \omega^{\nicefrac{3}{4}} \big[ \delta \eu e^{\ii \omega \overset{\scriptscriptstyle (0)}{{\phu}}} + \overset{\scriptscriptstyle (1)}{\eu} \big( e^{\ii \omega \overset{\scriptscriptstyle (0)}{{\phu}}} - e^{\ii \omega \overset{\scriptscriptstyle (1)}{{\phu}}} \big) \big].
\end{align}
For the first term in the above equation, we can use \cref{eq:exp_est,eq:Taylor_est} to write
\begin{equation}
    \omega^{\nicefrac{3}{4}} |\delta \eu| |e^{\ii \omega \overset{\scriptscriptstyle (0)}{{\phu}}}| \leq C_2 \omega^{\nicefrac{3}{4}} (r^4 + \omega^{-1} r^2) e^{-\frac{c \omega r^2}{2}}.
\end{equation}
In the second term, we have $\overset{\scriptscriptstyle (1)}{\eu}$, which is uniformly bounded, and we can write $|\overset{\scriptscriptstyle (1)}{\eu}| \leq C_{\eu}$ for all $x \in \R^3$. For the difference of exponentials, we can write
\begin{equation}
    e^{\ii \omega \overset{\scriptscriptstyle (0)}{{\phu}}} - e^{\ii \omega \overset{\scriptscriptstyle (1)}{{\phu}}} = \int_0^1 \frac{d}{d s} e^{\ii \omega [s \overset{\scriptscriptstyle (0)}{{\phu}} + (1-s) \overset{\scriptscriptstyle (1)}{{\phu}}]} d s = \ii \omega \delta \phu \int_0^1  e^{\ii \omega [s \overset{\scriptscriptstyle (0)}{{\phu}} + (1-s) \overset{\scriptscriptstyle (1)}{{\phu}}]} d s.
\end{equation}
Taking the absolute value and using \cref{eq:exp_est} to get $s \Im \overset{\scriptscriptstyle (0)}{{\phu}} + (1-s) \Im \overset{\scriptscriptstyle (1)}{{\phu}} \geq \frac{c r^2}{2} $, we obtain
\begin{equation}
    |e^{\ii \omega \overset{\scriptscriptstyle (0)}{{\phu}}} - e^{\ii \omega \overset{\scriptscriptstyle (1)}{{\phu}}}| \leq \omega |\delta \phu| \int_0^1  e^{- \omega [s \Im \overset{\scriptscriptstyle (0)}{{\phu}} + (1-s) \Im \overset{\scriptscriptstyle (1)}{{\phu}}]} d s \leq \omega |\delta \phu| e^{- \frac{\omega c r^2}{2}} \leq \omega C_{\phu} r^6 e^{- \frac{\omega c r^2}{2}}\qquad \forall x \in B_{\rho}(x_0).
\end{equation}
Bringing all terms together, we get for some constant $C>0$ 
\begin{equation}
    |\overset{\scriptscriptstyle (0)}{\underline{\pmb{\mathfrak{E}}}} - \overset{\scriptscriptstyle (1)}{\underline{\pmb{\mathfrak{E}}}}| \leq C (\omega^{\nicefrac{3}{4}} r^4 + \omega^{-\nicefrac{1}{4}} r^2 + \omega^{\nicefrac{7}{4}}r^6) e^{-\frac{c \omega r^2}{2}} \qquad \forall x \in B_{\rho}(x_0).
\end{equation}
or equivalently (for some different constant $C$)
\begin{equation} \label{eq:pw_est}
    |\overset{\scriptscriptstyle (0)}{\underline{\pmb{\mathfrak{E}}}} - \overset{\scriptscriptstyle (1)}{\underline{\pmb{\mathfrak{E}}}}|^2 \leq C (\omega^{\nicefrac{3}{2}} r^8 + \omega^{-\nicefrac{1}{2}} r^4 + \omega^{\nicefrac{7}{2}} r^{12}) e^{-c \omega r^2} \qquad \forall x \in B_{\rho}(x_0).
\end{equation}
We can now estimate the integral by splitting it as follows:
\begin{equation}
    || \overset{\scriptscriptstyle (0)}{\underline{\pmb{\mathfrak{E}}}}( \cdot; \omega) - \overset{\scriptscriptstyle (1)}{\underline{\pmb{\mathfrak{E}}}}( \cdot; \omega)||_{L^2(\R^3)}^2 = \int_{B_\rho(x_0)} |\overset{\scriptscriptstyle (0)}{\underline{\pmb{\mathfrak{E}}}} - \overset{\scriptscriptstyle (1)}{\underline{\pmb{\mathfrak{E}}}}|^2 d^3 x + \int_{\R^3 \setminus B_\rho(x_0)} |\overset{\scriptscriptstyle (0)}{\underline{\pmb{\mathfrak{E}}}} - \overset{\scriptscriptstyle (1)}{\underline{\pmb{\mathfrak{E}}}}|^2 d^3 x.
\end{equation}
The first integral can be estimated using \cref{eq:pw_est} and the following bound
\begin{equation}
    \int_{B_\rho(x_0)} r^p e^{-c \omega r^2} d^3 x\leq \int_{\R^3} r^p e^{-c \omega r^2} d^3 x= \underbrace{4 \pi \int_0^\infty (r')^{p+2} e^{-c(r')^2} \, dr'}_{<\infty\text{ if } p>-3} \cdot \omega^{-\frac{p+3}{2}}\,,
\end{equation}
where we have used the substitution $r' = \sqrt{\omega} r$. We obtain
\begin{equation}
    \int_{B_\rho(x_0)} |\overset{\scriptscriptstyle (0)}{\underline{\pmb{\mathfrak{E}}}} - \overset{\scriptscriptstyle (1)}{\underline{\pmb{\mathfrak{E}}}}|^2 d^3 x \leq C  \int_{B_\rho(x_0)} \left(\omega^{\nicefrac{3}{2}} r^8 + \omega^{-\nicefrac{1}{2}} r^4 + \omega^{\nicefrac{7}{2}} r^{12} \right) e^{-c \omega r^2} d^3 x = \mathcal{O}(\omega^{-4}).
\end{equation}
For the second integral, we can write
\begin{equation}
   \int_{\R^3 \setminus B_\rho(x_0)} |\overset{\scriptscriptstyle (0)}{\underline{\pmb{\mathfrak{E}}}} - \overset{\scriptscriptstyle (1)}{\underline{\pmb{\mathfrak{E}}}}|^2 d^3 x \leq 2 \sum_{a=0}^1 \int_{\R^3 \setminus B_\rho(x_0)} |\overset{\scriptscriptstyle (a)}{\underline{\pmb{\mathfrak{E}}}}|^2 d^3 x.
\end{equation}
But we have
\begin{equation}
    |\overset{\scriptscriptstyle (a)}{\underline{\pmb{\mathfrak{E}}}}|^2 = \omega^{\nicefrac{3}{2}} |\overset{\scriptscriptstyle (a)}{\eu}|^2 e^{-2\omega \Im \overset{\scriptscriptstyle (a)}{{\phu}}} \leq \omega^{\nicefrac{3}{2}} C_{\eu} e^{-2 \omega \overset{\scriptscriptstyle (a)}{{m}}} \qquad \forall x \in \R^3 \setminus B_\rho(x_0).
\end{equation}
Thus, we obtain (for some constant C)
\begin{equation}
   \int_{\R^3 \setminus B_\rho(x_0)} |\overset{\scriptscriptstyle (0)}{\underline{\pmb{\mathfrak{E}}}} - \overset{\scriptscriptstyle (1)}{\underline{\pmb{\mathfrak{E}}}}|^2 d^3 x \leq C \omega^{\nicefrac{3}{2}} e^{-2 \omega m},
\end{equation}
where $m = \min(\overset{\scriptscriptstyle (0)}{{m}}, \overset{\scriptscriptstyle (1)}{{m}}) > 0$. Thus, this term is exponentially small in $\omega$. Thus, we obtain the final result
\begin{equation}
    || \overset{\scriptscriptstyle (0)}{\underline{\pmb{\mathfrak{E}}}}( \cdot; \omega) - \overset{\scriptscriptstyle (1)}{\underline{\pmb{\mathfrak{E}}}}( \cdot; \omega)||_{L^2(\R^3)}^2 = \mathcal{O}(\omega^{-4}) + \mathcal{O}(\omega^{\nicefrac{3}{2}} e^{-2 \omega m}) = \mathcal{O}(\omega^{-4}).
\end{equation}
\end{proof}

\section{The energy estimate} \label{SecEnergyEst}

In this section, we establish the basic energy estimate for Maxwell's equations in an inhomogeneous medium. It will be relevant to obtain a bound on the quality of the Gaussian beam approximation.

\begin{proposition} \label{PropDifferenceFieldsEst}
Let $\tilde{\bE}$, $\tilde{\bH} \in C^\infty([0, \infty) \times \R^3, \R^3)$ be such that for each $t \geq 0$ we have $\tilde{\bE}(t, \cdot)$ and $\tilde{\bH}(t, \cdot)$ compactly supported in $\R^3$. Moreover, we define 
\begin{subequations}
    \begin{align}
        \nabla \times \tilde{\bE} + \mu \dot{\tilde{\bH}} &=: -\tilde{\mathscr{F}}, \\
        \nabla \times \tilde{\bH} - \varepsilon \dot{\tilde{\bE}} &=: -\tilde{\mathscr{G}}, \
    \end{align}
\end{subequations}
and we use the notation $\tilde{\mathcal{E}} := \frac{1}{2}( \varepsilon |\tilde{\bE}|^2 + \mu |\tilde{\bH}|^2)$ and $\tilde{\mathbb{E}}(t):= \int_{\R^3} \tilde{\mathcal{E}}(t,x) \, d^3x$.
Then the following energy estimate holds:
    \begin{equation} \label{EqPropEnergyEstimate}
    \tilde{\mathbb{E}}^{\nicefrac{1}{2}}(t) \leq \tilde{\mathbb{E}}^{\nicefrac{1}{2}}(0) + \frac{1}{\sqrt{2c_m}} \int_0^t \Bigg[ \int_{\R^3} \Big( |\tilde{\mathscr{F}}|^2 + |\tilde{\mathscr{G}}|^2 \Big) \, d^3x \Bigg]^{\nicefrac{1}{2}} \, dt ,
\end{equation}
where the constant $c_m$ is defined below \cref{eq:constitutive}. Moreover, if $\tilde{\mathscr{F}}$ and $\tilde{\mathscr{G}}$ vanish, then $\tilde{\mathbb{E}}$ is independent of time.
\end{proposition}

\begin{proof}
We compute
\begin{align}
    \frac{d}{dt} \tilde{\mathbb{E}}(t) &=  \int_{\R^3} \Big( \varepsilon \tilde{\bE} \cdot \dot{\tilde{\bE}} + \mu \tilde{\bH} \cdot \dot{\tilde{\bH}} \Big) \, d^3x \nonumber\\
    &= \int_{\R^3} \Big[ \tilde{\bE} \cdot \big(\nabla \times \tilde{\bH} \big) - \tilde{\bH} \cdot \big(\nabla \times \tilde{\bE} \big) \Big] \, d^3x + \int_{\R^3} \Big( \tilde{\bE} \cdot \tilde{\mathscr{G}} - \tilde{\bH} \cdot \tilde{\mathscr{F}} \Big) \, d^3x \nonumber\\
    &= \int_{\R^3} \nabla \cdot \big( \tilde{\bH} \times \tilde{\bE} \big) \, d^3x + \int_{\R^3}\big( \tilde{\bE} \cdot \tilde{\mathscr{G}} - \tilde{\bH} \cdot \tilde{\mathscr{F}} \big) \, d^3x .
\end{align}
The first term on the right-hand side vanishes due to the assumption of compact spatial support for each fixed time. If $\tilde{\mathscr{F}}$ and $\tilde{\mathscr{G}}$ vanish, then this shows that $\tilde{\mathbb{E}}$ is independent of time. In full generality, to estimate the second term, we compute
\begin{align}
    \bigg|\int_{\R^3} \big( \tilde{\bE} \cdot \tilde{\mathscr{G}} - \tilde{\bH} \cdot \tilde{\mathscr{F}} \big) \, d^3x \bigg| &\leq \frac{1}{\sqrt{c_m}} \int_{\R^3} \big( \sqrt{\varepsilon}|\tilde{\bE} \cdot \tilde{\mathscr{G}}| + \sqrt{\mu}|\tilde{\bH} \cdot \tilde{\mathscr{F}}| \big) \, d^3x \nonumber\\
    &\leq \frac{1}{\sqrt{c_m}} \bigg[ \int_{\R^3} \big( \varepsilon \tilde{\bE} \cdot \tilde{\bE} + \mu \tilde{\bH} \cdot \tilde{\bH} \big) \, d^3x \bigg]^{\nicefrac{1}{2}} \bigg[ \int_{\R^3} \big( |\tilde{\mathscr{F}}|^2 + |\tilde{\mathscr{G}}|^2 \big) \, d^3x \bigg]^{\nicefrac{1}{2}},
\end{align}
where the constant $c_m$ is defined below \cref{eq:constitutive}, and the second inequality follows from Cauchy--Schwarz. Thus, we obtain
\begin{equation}
    \frac{d}{dt} \tilde{\mathbb{E}} \leq \sqrt{\frac{2}{c_m}} \tilde{\mathbb{E}}^{\nicefrac{1}{2}}  \bigg[ \int_{\R^3} \big(|\tilde{\mathscr{F}}|^2 + |\tilde{\mathscr{G}}|^2 \big) \, d^3x \bigg]^{\nicefrac{1}{2}} .
\end{equation}
This gives
\begin{equation}
    \frac{d}{dt}\tilde{\mathbb{E}}^{\nicefrac{1}{2}} \leq \frac{1}{\sqrt{2c_m}} \bigg[ \int_{\R^3} \big( |\tilde{\mathscr{F}}|^2 + |\tilde{\mathscr{G}}|^2) \, d^3x\bigg]^{\nicefrac{1}{2}},
\end{equation}
from which \cref{EqPropEnergyEstimate} follows by integration.
\end{proof}

\section{Approximation of exact solutions and proof of main results}\label{sec:approx_exact_and_main_results}

In this section, we give the proofs of \cref{MainThm} and \cref{PropJQLateTime}. Thus, we start by assuming $\mathcal{K}$-supported Gaussian beam initial data of order $2$, as in \cref{DefGBID}. Let $(\bE, \bH)$ denote the corresponding solution. Let $\rho_0>0$ be large enough that $\mathcal{K} \subseteq B_{\rho_0}(\x_0)$. By finite speed of propagation, we then have 
\begin{equation} \label{EqComSupActualSol}
\mathrm{supp}(\bE(t, \cdot)) \cup \mathrm{supp}(\bH(t, \cdot)) \subseteq B_{\rho_0 + t}(\x_0) \qquad \forall t \geq 0.
\end{equation}

We now construct the approximate Gaussian beam solution. Consider the induced jets $D^\alpha \phu|_{\x_0}$ for $0 \leq |\alpha| \leq 7$, $D^\alpha \eu_0|_{\x_0}$ for $0 \leq |\alpha| \leq 5$, and $D^\alpha \eu_1|_{\x_0}$ for $0 \leq |\alpha| \leq 3$. By the properties of these jets according to \cref{DefGBID}, the assumptions of \cref{thm:construction_theorem} are satisfied, and we obtain the approximate solutions
\begin{align} 
        &\hat{\bE} = \omega^{\nicefrac{3}{4}} \Re \Big[ \big(\be_0 + \omega^{-1} \be_1  + \omega^{-2} \be_{2} \big) e^{\ii \omega \phi} \Big], \qquad\hat{\bH} = \omega^{\nicefrac{3}{4}} \Re \Big[ \big(\bh_0 + \omega^{-1} \bh_1  + \omega^{-2} \bh_{2} \big) e^{\ii \omega \phi} \Big]
    \end{align}
of \cref{EqehApprox}. Without loss of generality, we can assume that $\rho>0$ in \cref{thm:construction_theorem} was chosen small enough that \cref{EqComSupActualSol} also holds for $\hat{\bE}$ and $\hat{\bH}$.
We set
\begin{equation}
    \tilde{\bE} := \bE - \hat{\bE}, \qquad \tilde{\bH} := \bH - \hat{\bH},
\end{equation}
which implies $\tilde{\mathscr{F}} = \hat{\mathscr{F}}$ and $\tilde{\mathscr{G}} = \hat{\mathscr{G}}$, where $\hat{\mathscr{F}}$ and $\hat{\mathscr{G}}$ are as in \cref{eq:FGhat}.

The following lemma shows that the Gaussian beam initial data we started with and the induced initial data of the approximate solution are sufficiently close.
\begin{lemma}
    We have
    \begin{equation} \label{EqClosenessOfID}
    \tilde{\mathbb{E}}(0) := \frac{1}{2} \int_{t=0} \big( \varepsilon | \tilde{\bE}|^2 + \mu| \tilde{\bH}|^2 \big) \, d^3x = \mathcal{O}(\omega^{-4}) .
\end{equation}
\end{lemma}

\begin{proof}
Since $\varepsilon$ and $\mu$ are uniformly bounded, this follows from showing $||\tilde{\bE}(0, \cdot)||_{L^2(\R^3)} = \mathcal{O}(\omega^{-2})$ and $||\tilde{\bH}(0, \cdot)||_{L^2(\R^3)} = \mathcal{O}(\omega^{-2})$. We first look at
\begin{align}
    \tilde{\bE}(0,x) &= \omega^{\nicefrac{3}{4}} \Re \Big\{ \big[ \eu_0(x) + \omega^{-1} \eu_1(x) \big] e^{\ii \omega \phu(x)} \Big\}   -  \omega^{\nicefrac{3}{4}} \Re \Big\{ \big[\be_0(0,x) + \omega^{-1} \be_1 (0,x) \big] e^{\ii \omega \phi (0,x)} \Big\} \nonumber\\
    &\qquad - \underbrace{\omega^{\nicefrac{3}{4}} \Re \Big[ \omega^{-2} \be_{2}(0,x) e^{\ii \omega \phi(0,x)} \Big]}_{= \mathcal{O}_{L^2(\R^3)}(\omega^{-2})}
    + \mathcal{O}_{L^2(\R^3)}(\omega^{-2})\,
\end{align}
where we use Lemma \ref{lem:degreetoorder} for the underbraced term.
Since at $\x_0$ the jets of the structure functions agree to a sufficiently high order, $||\tilde{\bE}(0, \cdot)||_{L^2(\R^3)} = \mathcal{O}(\omega^{-2})$ follows from \cref{PropEquivGBData}. We proceed in a similar manner for $||\tilde{\bH}(0, \cdot)||_{L^2(\R^3)} = \mathcal{O}(\omega^{-2})$, noting that the relations in \cref{eq:hdefID} ensure the agreement of the jets at $x_0$ to a sufficiently high order.
\end{proof}

Given the closeness of the initial data, we can now infer the closeness of the actual solution to the approximate solution up to a finite time.
\begin{lemma}
    Let $T>0$ be given. Then there exists a constant $C >0$ (dependent on $T$) such that 
\begin{equation} \label{EqClosenessSolution}
    \tilde{\mathbb{E}}^{\nicefrac{1}{2}}(t) \leq C \omega^{-2} \qquad \textnormal{ for } 0 \leq t \leq T  .
\end{equation}   
\end{lemma}

\begin{proof}
By \cref{EqComSupActualSol} and \cref{thm:construction_theorem}, $\tilde{\bE}(t, \cdot)$ and $\tilde{\bH}(t, \cdot)$ are compactly supported for each $t \geq 0$. Thus, we can invoke the energy estimate \eqref{EqPropEnergyEstimate} from \cref{PropDifferenceFieldsEst} and use \cref{EqClosenessOfID,EqEstimatesApproxSol}.
\end{proof}

\begin{proposition} \label{PropGoOverToHat}
Let $f \in C^0([0, \infty) \times \R^3)$ and $\mathcal{D} \in \{ \cE, \bS^i, \varepsilon \bE \cdot \bE, \mu \bH \cdot \bH \}$. Let $\hat{\mathcal{D}}$ be the respective quantity with respect to the approximate Gaussian beam solution. Then, there exists $C>0$ such that for all $0 \leq t \leq T$ we have
\begin{equation}
    \int_{\R^3} f(t,x) \mathcal{D}(t,x) \, d^3x = \int_{\R^3} f(t,x) \hat{\mathcal{D}}(t,x) \, d^3x + ||f(t, \cdot)||_{L^\infty(B_{\rho_0 +t}(\x_0))} C \omega^{-2} .
\end{equation}
\end{proposition}

\begin{proof}
We give the proof for $\mathcal{D} = \varepsilon \bE \cdot \bE$. The other cases follow analogously.
\begin{align}
    \bigg| \int_{\R^3} f \varepsilon \bE \cdot \bE \, d^3x &- \int_{\R^3} f \varepsilon \hat{\bE} \cdot \hat{\bE} \, d^3x \bigg| \nonumber\\
     &\leq ||f(t, \cdot)||_{L^\infty(B_{\rho_0 +t}(\x_0))} C_m || 2 \tilde{\bE} \cdot \bE - \tilde{\bE} \cdot \tilde{\bE} ||_{L^1(\R^3)} \nonumber\\
    &\leq ||f(t, \cdot)||_{L^\infty(B_{\rho_0 + t}(\x_0))} C_m ||\tilde{\bE}||_{L^2(\R^3)} \Big[ 2 ||\bE||_{L^2(\R^3)} + ||\tilde{\bE}||_{L^2(\R^3)} \Big] ,
\end{align}
where we have used \cref{EqComSupActualSol}, H\"older's inequality, and the Cauchy--Schwarz inequality. We now use the boundedness of the energy $\mathbb{E}$ and \cref{EqClosenessSolution}.
\end{proof}

The next corollary is a direct consequence of \cref{PropGoOverToHat} and the compact support \eqref{EqComSupActualSol} of the exact and approximate solutions. It shows that, up to time $T$, the following integrated quantities of the exact and approximate solutions are close. Note that once \cref{eq:XtohatX} is established, it is used for \cref{eq:JtohatJ,eq:QtohatQ}.
\begin{corollary} \label{cor:approx2exact}
We have for all $0 \leq t \leq T$ 
\begin{subequations}
 \begin{align}
    \mathbb{E}(t) &= \hat{\mathbb{E}}(t) + \mathcal{O}(\omega^{-2}), \\
    \mathbb{X}^i(t) &= \hat{\mathbb{X}}^i(t) + \mathcal{O}(\omega^{-2}),\label{eq:XtohatX} \\
    \mathbb{P}_i(t) &= \hat{\mathbb{P}}_i(t) + \mathcal{O}(\omega^{-2}), \\
    \mathbb{J}_i(t) &= \hat{\mathbb{J}}_i(t) + \mathcal{O}(\omega^{-2}),\label{eq:JtohatJ} \\
    \mathbb{\bbQ}^{i j}(t) &= \hat{\mathbb{\bbQ}}^{i j}(t) + \mathcal{O}(\omega^{-2}).\label{eq:QtohatQ}
\end{align}     
\end{subequations}  
\end{corollary}

We can now prove \cref{PropJQLateTime}.

\begin{proof}[Proof of \cref{PropJQLateTime}]
The proof of \cref{eq:J_Q} follows from \cref{prop:approx_average_quantities,prop:J_decomposition}, together with \cref{cor:approx2exact}. We also use $\dot{\phi}|_{\gamma(t)} = \mathrm{const.} = \dot{\phu}|_{\x_0}$, as given in \cref{eq:phase_construction}, as well as the leading order form of $\hat{\mathbb{P}}$ given in \cref{eq:P_leading_order}. The fact that $s$ is a constant in the interval $[-1, 1]$ follows from \cref{prop:J_decomposition}.
\end{proof}

We now prove a first estimate on the energy centroid.

\begin{proposition}\label{PropFirstEstimateCentroid}
   There exists $C>0$ such that $|\mathbb{X}^i(t) - \gamma^i(t)| \leq C \omega^{-1}$ for all $0 \leq t \leq T$.
\end{proposition}

\begin{proof} 
Using \cref{cor:approx2exact}, we compute
\begin{equation} 
\mathbb{E} [ \mathbb{X}^i(t) - \gamma^i(t) ] =\hat{\mathbb{E}} [ \hat{\mathbb{X}}^i(t) - \gamma^i(t) ] +\mathcal{O}(\omega^{-2}).
\end{equation}
Then, by the stationary phase expansion in \cref{th:stationary} with $x_s=\gamma(t)$, it follows that 
\begin{align}
    \mathbb{E} [ \mathbb{X}^i(t) - \gamma^i(t) ] = \int_{\R^3} [x^i - \gamma^i(t)] \hat{\mathcal{E}}(t,x) \, d^3x +\mathcal{O}(\omega^{-2})=\mathcal{O}(\omega^{-1}) .
\end{align}
\end{proof}

\begin{proposition} \label{PropErrorVanishingHigherOrder}
Let $\hat{\mathcal{D}} \in \{\hat{\cE}, \hat{\bS}^i, \varepsilon \hat{\bE}^2, \mu \hat{\bH}^2\}$ and let $p \in C^\infty([0, \infty) \times \R^3)$ be a function such that $D^\alpha p(t, \bbX(t)) = 0$ for all multi-indices $0 \leq |\alpha| \leq 2$. Then, there exists a constant $C>0$ such that for all $0 \leq t \leq T$ we have 
\begin{equation}
        \bigg|\int_{\R^3} p(t,x) \hat{\mathcal{D}} \, d^3x \bigg| \leq C \cdot \omega^{-2} .
\end{equation}
\end{proposition}

\begin{proof}
We do this for $\hat{\mathcal{E}}$ but the other cases are analogous.

We define $r^{i}(t,x):=x^i-\bbX(t)$ and, for each $t\in [0,T]$, we Taylor expand $p$ around $\bbX(t)$ to write 
\begin{align}
    \int_{\R^3} p(t,x) \hat{\mathcal{E}}\, d^3x&=\int_{\R^3}\sum_{|\alpha| = 3} \Big[\frac{D^\alpha p(t,\bbX(t)) }{\alpha !} + h_\alpha(t,x)\Big]r^\alpha\hat{\mathcal{E}}\, d^3x,
    \end{align}
    with $\lim_{x \to \bbX(t)} h_\alpha (t, x) = 0 $. Using \cref{eq:approx_density} and noting, from \cite[Th. 7.7.1]{hormander1983analysis}, that the integrals of that appear in the above terms proportional to $e^{\pm 2 \ii \omega \Re \phi}$ decay to an arbitrarily high order in $\omega$ gives
    \begin{align}
    \int_{\R^3} p(t,x) \hat{\mathcal{E}}\, d^3x&=\int_{\R^3}\sum_{|\alpha| = 3} \Big[\frac{D^\alpha p(t,\bbX(t)) }{\alpha !} + h_\alpha(t,x)\Big]{r}^\alpha u \omega^{\nicefrac{3}{2}} e^{-2 \omega \Im \phi}\, d^3x,
\end{align}
where $u$ is given in \cref{eq:u_def}. 
We now use the stationary phase approximation in \cref{th:stationary} with $x_s=\gammu(t)$, $f=2\ii\Im\phi$ and
\begin{align}
    q=\sum_{|\alpha| = 3} \Big[\frac{D^\alpha p(t,\bbX(t)) }{\alpha !} + h_\alpha(t,x)\Big]{r}^\alpha u.
\end{align}
Using~\cref{PropFirstEstimateCentroid} to give $r(t,\gammu(t))=\mathcal{O}(\omega^{-1})$ and $\Im\phi|_{\gamma}=\Im\phi|_{x_0}=0$ by \cref{eq:phi_const_along_gamma}, we obtain
\begin{align}
   \int_{\R^3} p(t,x) \hat{\mathcal{E}}\, d^3x&=\mathcal{O}(\omega^{-2}),
\end{align}
since $L_0u|_{\gammu(t)}=\mathcal{O}(\omega^{-3})$ and $\omega^{-1}L_1u|_{\gammu(t)}=\mathcal{O}(\omega^{-2})$.
\end{proof}

We are now in a position to prove the ODE system \eqref{EqMainThm} in \cref{MainThm}.

\subsection{Proof of \cref{MainThm}}

\begin{proof}[Proof of \cref{MainThm}] 

We begin by defining
$ r^i(t, x) := x^i - \mathbb{X}^i(t)$.
Note that, by \cref{PropFirstEstimateCentroid}, we have $r^i(t, x) =   x^i - \gamma^i(t) + \mathcal{O}(\omega^{-1})$. We now prove the evolution equations \eqref{EqMainThm} one by one.

\textbf{Step 1: proof of \cref{EqMainThmXDot}.} We start from \cref{EqXDot} and Taylor-expand $\n^{-2}$ for each $0 \leq t \leq T$ around $\mathbb{X}(t)$:
\begin{equation}
\n^{-2}(x) = \sum_{|\alpha| \leq 2} \frac{D^\alpha \n^{-2}(\bbX(t))}{\alpha!}r^\alpha + \underbrace{\sum_{|\alpha| = 3} \bigg[ \frac{D^\alpha \n^{-2}(\bbX(t))}{\alpha!} + h_\alpha(x) \bigg] r^\alpha}_{=:R(t,x)},
\end{equation}
where $h_\alpha(x) \to 0$ for $x \to \bbX(t)$. Thus, we obtain
\begin{align}
  \mathbb{E} \cdot \dot{\mathbb{X}}^i (t) = \int_{\R^3} \n^{-2} \bS^i \, d^3x &= \n^{-2} \Big|_{\mathbb{X}(t)} \mathbb{P}^i + \nabla_j \n^{-2} \Big|_{\mathbb{X}(t)} \int_{\R^3} r^j \bS^i \, d^3x \nonumber\\
  &\qquad + \frac{1}{2} \nabla_j \nabla_k \n^{-2} \Big|_{\mathbb{X}(t)} \int_{\R^3} r^j r^k \bS^i \, d^3x + \int_{\R^3} R(t,x) \bS^i \, d^3x ,
\end{align}
where we used the definition \eqref{DefP} of $\mathbb{P}$ in the first term on the right-hand side. The second term can be rewritten as
\begin{align} \label{EqExpressJSym}
    \int_{\R^3} r^j \bS^i d^3x = \int_{\R^3} \left( r^{[j} \bS^{i]} + r^{(j} \bS^{i)} \right) \, d^3x = - \frac{1}{2} \epsilon^{i j k} \mathbb{J}_k + \int_{\R^3} r^{(j} \bS^{i)} \, d^3x,
\end{align}
where we used the definition \eqref{DefJ} of $\mathbb{J}$. The symmetric term in the above equation can be related to the time derivative of the quadrupole moment. Using \cref{eq:dotQ} and the Taylor expansion of $\n^{-2}$, we obtain
\begin{align}
    \frac{1}{2}\dot{\bbQ}^{i j} = \int_{\mathbb{R}^3} \n^{-2} r^{(i} \bS^{j)} \, d^3x 
    &= \n^{-2} \Big|_{\bbX(t)} \int_{\R^3} r^{(i} \bS^{j)} \, d^3x +  \nabla_k \n^{-2} \Big|_{\mathbb{X}(t)} \int_{\R^3}  r^k r^{(i} \bS^{j)}  \, d^3x \nonumber\\
    &\qquad + \int_{\R^3} \left[ \frac{1}{2} \nabla_a \nabla_b \n^{-2} \Big|_{\bbX(t)} r^a r^b + R(t,x) \right] r^{(i} \bS^{j)} \, d^3x.
\end{align}
We use \cref{PropGoOverToHat} for all the remaining integrals, and putting everything together yields
\begin{align}
    \mathbb{E} \cdot \dot{\mathbb{X}}^i (t) &= \n^{-2} \Big|_{\mathbb{X}(t)} \mathbb{P}^i - \frac{1}{2}  \nabla_j \n^{-2} \Big|_{\mathbb{X}(t)} \epsilon^{i j k}  \mathbb{J}_k + \frac{1}{2} \n^2 \nabla_j \n^{-2} \Big|_{\mathbb{X}(t)} \dot{\bbQ}^{i j} \nonumber\\
     &- \n^2 (\nabla_j \n^{-2}) (\nabla_k \n^{-2}) \Big|_{\mathbb{X}(t)} \int_{\R^3}  r^k r^{(i} \hat{\bS}^{j)}  \, d^3x  + \frac{1}{2} \nabla_j \nabla_k \n^{-2} \Big|_{\mathbb{X}(t)} \int_{\R^3} r^j r^k \hat{\bS}^i \, d^3x \nonumber \\
     &- \n^2 \nabla_j \n^{-2} \Big|_{\mathbb{X}(t)} \int_{\R^3} \left[ \frac{1}{2} \nabla_a \nabla_b \n^{-2} \Big|_{\bbX(t)} r^a r^b + R(t,x) \right] r^{(i} \hat{\bS}^{j)} \, d^3x + \int_{\R^3} R(t,x) \hat{\bS}^i \, d^3x + \mathcal{O}(\omega^{-2}).
\end{align}
In the above equation, the terms on the last line are $\mathcal{O}(\omega^{-2})$ by \cref{PropErrorVanishingHigherOrder}. The integrals on the second line can be evaluated using \cref{th:stationary}, which gives
\begin{align}
    \int_{\R^3} r^k r^i \hat{\bS}^j \, d^3 x &= \frac{1}{\sqrt{\det\left( \frac{1}{2 \pi \ii} A \right)}} \Big[ r^k r^i v^j + \omega^{-1} L_1 \left( r^k r^i v^j \right) \Big] \Big|_{\gamma(t)} + \mathcal{O}(\omega^{-2}) \nonumber \\
    &= \frac{\ii \omega^{-1}}{\sqrt{\det\left( \frac{1}{2 \pi \ii} A \right)}} (A^{-1})^{k i} v^j \Big|_{\gamma(t)} + \mathcal{O}(\omega^{-2}) \nonumber \\
    &= \ii \omega^{-1} (A^{-1})^{k i} \hat{\mathbb{P}}^j + \mathcal{O}(\omega^{-2}) = \frac{1}{\hat{\mathbb{E}}} \hat{\bbQ}^{k i} \hat{\mathbb{P}}^j + \mathcal{O}(\omega^{-2}) = \frac{1}{\mathbb{E}} \bbQ^{k i} \mathbb{P}^j + \mathcal{O}(\omega^{-2}),
\end{align}
where $v$ is defined as in \cref{eq:v_def}. To obtain the second line, we used \cref{PropFirstEstimateCentroid}, which gives $r^i|_{\gamma(t)} = \mathcal{O}(\omega^{-1})$. In the last line of equalities, we first used \cref{eq:P_hat}, then \cref{eq:Q_hat}, and finally \cref{cor:approx2exact}.

The evolution equation for the energy centroid can now be written as
\begin{align}
    \dot{\mathbb{X}}^i &= \frac{1}{\mathbb{E} \n^2} \mathbb{P}^i - \frac{1}{\mathbb{E} \n^2} \epsilon^{i j k} \mathbb{J}_j \nabla_k \ln \n - \frac{1}{\mathbb{E}} \dot{\mathbb{Q}}^{i j} \nabla_j \ln \n \nonumber \\
    &\qquad- \frac{1}{\mathbb{E}^2 \n^2} \Big[ \mathbb{P}^i \mathbb{Q}^{j k} \nabla_j \nabla_k \ln \n + 2 \mathbb{P}^j \mathbb{Q}^{i k} (\nabla_j \ln \n) (\nabla_k \ln \n)  \Big] + \mathcal{O}(\omega^{-2}),
\end{align}
where $\n$ and its derivatives are evaluated at $\bbX(t)$.

\textbf{Step 2: proof of \cref{EqMainThmPDot}.} We start from \cref{EqPDot} and Taylor-expand $\nabla_i \ln \varepsilon$ and $\nabla_i \ln \mu$ for each $0 \leq t \leq T$ around $\mathbb{X}(t)$:
\begin{subequations}
\begin{align}
    (\nabla_i \ln \varepsilon)(x) &= \sum_{|\alpha| \leq 2} \frac{D^\alpha \nabla_i \ln \varepsilon (\bbX(t))}{\alpha!}r^\alpha + \underbrace{\sum_{|\alpha| = 3} \bigg[ \frac{D^\alpha \nabla_i \ln \varepsilon (\bbX(t))}{\alpha!} + h^\varepsilon_\alpha(x) \bigg] r^\alpha}_{=:R^\varepsilon_i(t,x)} , \\
    (\nabla_i \ln \mu)(x) &= \sum_{|\alpha| \leq 2} \frac{D^\alpha \nabla_i \ln \mu(\bbX(t))}{\alpha!}r^\alpha + \underbrace{\sum_{|\alpha| = 3} \bigg[ \frac{D^\alpha \nabla_i \ln \mu(\bbX(t))}{\alpha!} + h^\mu_{\alpha}(x) \bigg] r^\alpha}_{=:R^\mu_i(t,x)} ,
\end{align}
\end{subequations}
where $h^{\varepsilon}_\alpha(x) \to 0$ and $h^{\mu}_\alpha(x) \to 0$ for $x \to \bbX(t)$. Thus, we obtain
\begin{align}
    \dot{\mathbb{P}}_i &= \frac{1}{2} \int_{\R^3} \big( \varepsilon \bE \cdot \bE \nabla_i \ln \varepsilon + \mu \bH \cdot \bH \nabla_i \ln \mu \big)  \, d^3x \nonumber\\
    &= \nabla_i \ln \varepsilon \Big|_{\mathbb{X}^i(t)} \int_{\R^3} \frac{\varepsilon}{2} \bE \cdot \bE \, d^3x + \nabla_i \ln \mu \Big|_{\mathbb{X}^i(t)} \int_{\R^3} \frac{\mu}{2} \bH \cdot \bH \, d^3x \nonumber\\
    &\qquad +\nabla_j \nabla_i \ln \varepsilon \Big|_{\mathbb{X}^i(t)} \int_{\R^3} r^j \frac{\varepsilon}{2} \bE \cdot \bE \, d^3x + \nabla_j \nabla_i \ln \mu \Big|_{\mathbb{X}^i(t)} \int_{\R^3} r^j\frac{\mu}{2} \bH \cdot \bH \, d^3x \nonumber\\
    &\qquad + \nabla_k \nabla_j \nabla_i \ln \varepsilon \Big|_{\mathbb{X}^i(t)} \int_{\R^3} r^j r^k \frac{\varepsilon}{2} \bE \cdot \bE \, d^3x + \nabla_k \nabla_j \nabla_i \ln \mu \Big|_{\mathbb{X}^i(t)} \cdot \int_{\R^3} r^j r^k \frac{\mu}{2} \bH \cdot \bH \, d^3x \nonumber\\
    &\qquad + \int_{\R^3} \bigg[ R^\varepsilon_i(t,x) \frac{\varepsilon}{2} \bE \cdot \bE + R^\mu_i(t,x) \frac{\mu}{2} \bH \cdot \bH \bigg] \, d^3 x .
\end{align}
Since the dipole moment of the energy density with respect to the energy centroid vanishes by definition, we have
\begin{equation} 
    \mathbb{D}^j = 0 \quad \Leftrightarrow \quad \int_{\R^3} r^j \frac{\varepsilon}{2} \bE \cdot \bE \, d^3x = - \int_{\R^3} r^j \frac{\mu}{2} \bH \cdot \bH \, d^3x.
\end{equation}
Using this relation and \cref{PropGoOverToHat}, we obtain
\begin{align}
    \dot{\mathbb{P}}_i &= \nabla_i \ln \varepsilon \Big|_{\mathbb{X}^i(t)} \int_{\R^3} \frac{\varepsilon}{2} \hat{\bE} \cdot \hat{\bE} \, d^3x + \nabla_i \ln \mu \Big|_{\mathbb{X}^i(t)} \int_{\R^3} \frac{\mu}{2} \hat{\bH} \cdot \hat{\bH} \, d^3x  + \nabla_j \nabla_i \ln \frac{\varepsilon}{\mu}  \Big|_{\mathbb{X}^i(t)} \int_{\R^3} r^j \frac{\varepsilon}{2} \hat{\bE} \cdot \hat{\bE} \, d^3x  \nonumber\\
    &\qquad + \nabla_k \nabla_j \nabla_i \ln \varepsilon \Big|_{\mathbb{X}^i(t)} \int_{\R^3} r^j r^k \frac{\varepsilon}{2} \hat{\bE} \cdot \hat{\bE} \, d^3x + \nabla_k \nabla_j \nabla_i \ln \mu \Big|_{\mathbb{X}^i(t)} \int_{\R^3} r^j r^k \frac{\mu}{2} \hat{\bH} \cdot \hat{\bH} \, d^3x \nonumber\\
    &\qquad + \int_{\R^3} \bigg[ R^\varepsilon_i(t,x) \frac{\varepsilon}{2} \hat{\bE} \cdot \hat{\bE} + R^\mu_i(t,x) \frac{\mu}{2} \hat{\bH} \cdot \hat{\bH} \bigg] \, d^3 x  + \mathcal{O}(\omega^{-2}).
\end{align}
The last term on the right-hand side is $\mathcal{O}(\omega^{-2})$ by \cref{PropErrorVanishingHigherOrder}, and all remaining integrals can be evaluated using \cref{th:stationary} with $x_s=\gammu(t)$. We have
\begin{subequations}
\begin{align}
    \int_{\R^3} \frac{\varepsilon}{2} \hat{\bE} \cdot \hat{\bE} \, d^3x &= \frac{1}{4\sqrt{\det\left( \frac{1}{2 \pi \ii} A \right)}} \Big[ \varepsilon \be \cdot \overline{\be} + \omega^{-1} L_1 \left( \varepsilon \be \cdot \overline{\be} \right) \Big] \Big|_{\gamma(t)} + \mathcal{O}(\omega^{-2}) \nonumber \\
    &= \frac{1}{2} \hat{\mathbb{E}} + \mathcal{O}(\omega^{-2}) = \frac{1}{2} \mathbb{E} + \mathcal{O}(\omega^{-2}), \\
    \int_{\R^3} \frac{\mu}{2} \hat{\bH} \cdot \hat{\bH} \, d^3x &= \frac{1}{4\sqrt{\det\left( \frac{1}{2 \pi \ii} A \right)}} \Big[ \mu \bh \cdot \overline{\bh} + \omega^{-1} L_1 \left( \mu \bh \cdot \overline{\bh} \right) \Big] \Big|_{\gamma(t)} + \mathcal{O}(\omega^{-2}) \nonumber \\
    &= \frac{1}{2} \hat{\mathbb{E}} + \mathcal{O}(\omega^{-2}) = \frac{1}{2} \mathbb{E} + \mathcal{O}(\omega^{-2}).
\end{align}
\end{subequations}
In the equations above, the first lines follow from the stationary phase approximation given in \cref{th:stationary}. The second lines follow from combining \cref{Lemma_E/2} with \cref{eq:E_hat}. Then, for the final equality, we invoke \cref{PropGoOverToHat}. We continue with the evaluation of the dipole term
\begin{align}
    \int_{\R^3} r^j \frac{\varepsilon}{2} \hat{\bE} \cdot \hat{\bE} \, d^3x &= \frac{1}{4\sqrt{\det\left( \frac{1}{2 \pi \ii} A \right)}} \Big[ r^j \varepsilon \be \cdot \overline{\be} + \omega^{-1} L_1 \left( r^j \varepsilon \be \cdot \overline{\be} \right) \Big] \Big|_{\gamma(t)} + \mathcal{O}(\omega^{-2}) \nonumber \\
    &= \frac{1}{4\sqrt{\det\left( \frac{1}{2 \pi \ii} A \right)}} \Big\{ \big[\gamma^j(t) - \hat{\mathbb{X}}^j(t) \big] \varepsilon \be \cdot \overline{\be} \nonumber\\ 
    &\qquad+ \ii \omega^{-1} (A^{-1})^{j a} \big[ \nabla_a (\varepsilon \be \cdot \overline{\be}) - \ii \varepsilon \be \cdot \overline{\be} (A^{-1})^{b c} \nabla_a \nabla_b \nabla_c \Im \phi \big] \Big\} \Big|_{\gamma(t)} + \mathcal{O}(\omega^{-2}) \nonumber\\
    &= \mathcal{O}(\omega^{-2}).
\end{align}
We used \cref{PropGoOverToHat} in the second equality to replace $\bbX^j = \hat{\bbX}^j + \mathcal{O}(\omega^{-2})$, and the last equality follows after replacing $\hat{\mathbb{X}}^j(t)$ with the expression given in \cref{eq:X_hat}. 
Finally, the quadrupole terms are
\begin{align}
    \int_{\R^3} r^j r^k \frac{\varepsilon}{2} \hat{\bE} \cdot \hat{\bE} \, d^3x &= \frac{1}{4\sqrt{\det\left( \frac{1}{2 \pi \ii} A \right)}} \Big[ r^j r^k \varepsilon \be \cdot \overline{\be} + \omega^{-1} L_1 \left( r^j r^k \varepsilon \be \cdot \overline{\be} \right) \Big] \Big|_{\gamma(t)} + \mathcal{O}(\omega^{-2}) \nonumber \\
    &= \frac{\ii \omega^{-1}}{4\sqrt{\det\left( \frac{1}{2 \pi \ii} A \right)}} \varepsilon \be_0 \cdot \overline{\be}_0 (A^{-1})^{j k} \Big|_{\gamma(t)} + \mathcal{O}(\omega^{-2}) \nonumber \\
    &= \frac{1}{2} \hat{\mathbb{Q}}^{j k}(t) + \mathcal{O}(\omega^{-2}) = \frac{1}{2} \bbQ^{j k}(t) + \mathcal{O}(\omega^{-2}), 
    \end{align} 
    and
    \begin{align}  
    \int_{\R^3} r^j r^k \frac{\mu}{2} \hat{\bH} \cdot \hat{\bH} \, d^3x &= \frac{1}{4\sqrt{\det\left( \frac{1}{2 \pi \ii} A \right)}} \Big[ r^j r^k \mu \bh \cdot \overline{\bh} + \omega^{-1} L_1 \left( r^j r^k \mu \bh \cdot \overline{\bh} \right) \Big] \Big|_{\gamma(t)} + \mathcal{O}(\omega^{-2}) \nonumber \\
    &= \frac{\ii \omega^{-1}}{4\sqrt{\det\left( \frac{1}{2 \pi \ii} A \right)}} \mu \bh_0 \cdot \overline{\bh}_0 (A^{-1})^{j k} \Big|_{\gamma(t)} + \mathcal{O}(\omega^{-2}) \nonumber \\
    &= \frac{\ii \omega^{-1}}{4\sqrt{\det\left( \frac{1}{2 \pi \ii} A \right)}} \varepsilon \be_0 \cdot \overline{\be}_0 (A^{-1})^{j k} \Big|_{\gamma(t)} + \mathcal{O}(\omega^{-2}) \nonumber \\
    &= \frac{1}{2} \hat{\mathbb{Q}}^{j k}(t) + \mathcal{O}(\omega^{-2}) = \frac{1}{2} \bbQ^{j k}(t) + \mathcal{O}(\omega^{-2}).
\end{align}    
In the above, we have used a combination of \cref{Lemma_E/2} with \cref{eq:E_hat}, then \cref{eq:Q_hat}, and for the final equalities, we have invoked \cref{PropGoOverToHat}.

Bringing everything together, the evolution equation for the total linear momentum becomes
\begin{equation}
\dot{\mathbb{P}}_i = \mathbb{E} \nabla_i \ln \n + \bbQ^{j k} \nabla_i \nabla_j \nabla_k \ln \n + \mathcal{O}(\omega^{-2}),    
\end{equation}
where the derivatives of $\n$ are evaluated at $\bbX(t)$.

\textbf{Step 3: proof of \cref{EqMainThmJDot}.} We start from \cref{EqPDot} and use the same Taylor expansion for $\nabla_i \ln \varepsilon$ and $\nabla_i \ln \mu$ as above to obtain
\begin{align}
    \dot{\mathbb{J}}_i &= \epsilon_{i j k} \mathbb{P}^j \dot{\mathbb{X}}^k + \frac{1}{2} \int_{\R^3} \epsilon_{ijk} r^j \big( \varepsilon \bE \cdot \bE \nabla^k \ln \varepsilon + \mu \bH \cdot \bH \nabla^k \ln \mu \big) \, d^3x \nonumber\\
    &= \epsilon_{i j k} \mathbb{P}^j \dot{\mathbb{X}}^k + \epsilon_{i j k} \bigg[ \nabla^k \ln \varepsilon \Big|_{\mathbb{X}^i(t)} \int_{\R^3} r^j \frac{\varepsilon}{2} \bE \cdot \bE \, d^3x + \nabla^k \ln \mu \Big|_{\mathbb{X}^i(t)} \int_{\R^3} r^j \frac{\mu}{2} \bH \cdot \bH \, d^3x \bigg] \nonumber\\
    &\qquad + \epsilon_{i j k} \bigg[ \nabla_l \nabla^k \ln \varepsilon \Big|_{\mathbb{X}^i(t)} \int_{\R^3} r^j r^l \frac{\varepsilon}{2} \bE \cdot \bE \, d^3x + \nabla_l \nabla^k \ln \mu \Big|_{\mathbb{X}^i(t)} \int_{\R^3} r^j r^l \frac{\mu}{2} \bH \cdot \bH \, d^3x \bigg] \nonumber\\
    &\qquad + \epsilon\indices{_{i j}^k} \int_{\R^3} r^j \bigg[ R^{\varepsilon}_k \frac{\varepsilon}{2} \bE \cdot \bE + R^{\mu}_k \frac{\mu}{2} \bH \cdot \bH \bigg] \, d^3 x .
\end{align}
Using \cref{PropGoOverToHat} and the vanishing of the dipole moment, we obtain
\begin{align}
    \dot{\mathbb{J}}_i &= \epsilon_{i j k} \mathbb{P}^j \dot{\mathbb{X}}^k + \epsilon_{i j k} \nabla^k \ln \frac{\varepsilon}{\mu} \Big|_{\mathbb{X}^i(t)} \int_{\R^3} r^j \frac{\varepsilon}{2} \hat{\bE} \cdot \hat{\bE} \, d^3x  \nonumber\\
    &\qquad + \epsilon_{i j k} \bigg[ \nabla_l \nabla^k \ln \varepsilon \Big|_{\mathbb{X}^i(t)} \int_{\R^3} r^j r^l \frac{\varepsilon}{2} \hat{\bE} \cdot \hat{\bE} \, d^3x + \nabla_l \nabla^k \ln \mu \Big|_{\mathbb{X}^i(t)} \int_{\R^3} r^j r^l \frac{\mu}{2} \hat{\bH} \cdot \hat{\bH} \, d^3x \bigg] \nonumber\\
    &\qquad + \epsilon\indices{_{i j}^k} \int_{\R^3} r^j \bigg[ R^{\varepsilon}_k \frac{\varepsilon}{2} \hat{\bE} \cdot \hat{\bE} + R^{\mu}_k \frac{\mu}{2} \hat{\bH} \cdot \hat{\bH} \bigg] \, d^3 x + \mathcal{O}(\omega^{-2}).
\end{align}
The last term on the right-hand side is $\mathcal{O}(\omega^{-2})$ by \cref{PropErrorVanishingHigherOrder}, and all remaining integrals have been evaluated above in Step 2. The evolution equation for the total angular momentum is 
\begin{equation}
    \dot{\mathbb{J}}_i = \epsilon_{i j k} \Big( \mathbb{P}^{j} \dot{\mathbb{X}}^k + \bbQ^{j l} \nabla_l \nabla^k \ln \n \Big) + \mathcal{O}(\omega^{-2}),
\end{equation}
where the derivatives of $\n$ are evaluated at $\bbX(t)$. 

\textbf{Step 4: proof of \cref{EqMainThmQDot}.} We start from \cref{eq:dotQ} and use \cref{PropGoOverToHat} and~\cref{cor:approx2exact} (to replace $r^i = x^i - \bbX^i = x^i - \hat{\bbX}^i + \mathcal{O}(\omega^{-2}) = \hat{r}^j + \mathcal{O}(\omega^{-2})$) to obtain
\begin{equation}
    \dot{\bbQ}^{i j} = 2 \int_{\R^3} \n^{-2} r^{(i} \bS^{j)}  \, d^3x = 2 \int_{\R^3} \n^{-2} \hat{r}^{(i} \hat{\bS}^{j)}  \, d^3x + \mathcal{O}(\omega^{-2})
\end{equation}
Next, we compute
\begin{align} 
    \frac{d}{dt} \hat{\bbQ}^{ij}(t) &=- \underbrace{2\dot{\hat{\bbX}}^{(i}(t) \int_{\R^3}  \hat{r}^{j)} \hat{\cE}(t,x) \, d^3x }_{=: I}  + \underbrace{\int_{\R^3} \hat{r}^i \hat{r}^j \big( \varepsilon \dot{\hat{\bE}} \cdot \hat{\bE} + \mu \dot{\hat{\bH}} \cdot \hat{\bH} \big) \, d^3x}_{=:II}
\end{align}
Let us first evaluate the first term on the right-hand side. We compute
\begin{align}
    \int_{\R^3}  \hat{r}^{j} \hat{\cE}(t,x) \, d^3x =\int_{\R^3}  x^j\hat{\cE}(t,x) \, d^3x-\hat{\bbX}^j\int_{\R^3}  \hat{\cE}(t,x) \, d^3x=\hat{\bbE}\hat{\bbX}^j-\hat{\bbX}^j\hat{\bbE}=0. 
\end{align}
This shows $I = 0$. We now continue with $II$:
\begin{align}
    II &= \int_{\R^3} \hat{r}^i \hat{r}^j \Big[ (\nabla \times \hat{\bH}) \cdot \hat{\bE} - (\nabla \times \hat{\bE}) \cdot \hat{\bH} \Big] \, d^3x + \int_{\R^3} \hat{r}^i \hat{r}^j \big( \hat{\mathscr{G}} \cdot \hat{\bE} - \hat{\mathscr{F}} \cdot \hat{\bH} \big) \, d^3x \nonumber\\
    &= \int_{\R^3} \hat{r}^i \hat{r}^j \nabla \cdot ( \hat{\bH} \times \hat{\bE}) \, d^3x + \mathcal{O}(\omega^{-2})
    \end{align}
    where we use~\cref{EqEstimatesApproxSol} and the compact support of $\hat{E}$ and $\hat{H}$. So, 
    \begin{align}
    II&= \int_{\R^3} \Big[ \hat{r}^i (\hat{\bE} \times \hat{\bH})^j + \hat{r}^j ( \hat{\bE} \times \hat{\bH})^i \Big] \, d^3x + \mathcal{O}(\omega^{-2}) \nonumber\\
    &= 2 \int_{\R^3} \n^{-2} \hat{r}\indices{^{(i}} \hat{\bS}{}\indices{^{j)}} \, d^3x + \mathcal{O}(\omega^{-2}) .
\end{align}
Thus, we have
\begin{equation}
    \dot{\bbQ}^{i j} = \int_{\R^3} \hat{r}^i \hat{r}^j \partial_t \hat{\cE}(t,x) \, d^3x + \mathcal{O}(\omega^{-2}).
\end{equation}
Next, we can write
\begin{equation}
    \partial_t \hat{\mathcal{E}} = \omega^{\nicefrac{3}{2}} \big( \partial_t u - 2 \omega u \Im \dot{\phi} \big) e^{-2 \omega \Im \phi} + \omega^{\nicefrac{3}{2}} \partial_t \Big[ \Re \big( u e^{2 \ii  \omega \Re \phi} e^{-2 \omega \Im \phi} \big) \Big], 
\end{equation}
with $u$ defined in \cref{eq:u_def}. Using \cite[Th. 7.7.1]{hormander1983analysis}, the integrals of the above terms proportional to $e^{\pm 2 \ii \omega \Re \phi}$ decay to an arbitrarily high order in $\omega$. For the remaining terms, we apply the stationary phase approximation given in \cref{th:stationary} with $x_s=\gammu(t)$, $f=2\ii\Im\phi$, $A_{ab}=2\ii\nabla_a\nabla_a\Im\phi |_{\gamma(t)}$, and
    \begin{align}
    q&=\omega^{\nicefrac{3}{2}}\partial_tu\qquad\text{or}\qquad q=-2\omega^{\nicefrac{3}{2}}\omega \hat{r}^i\hat{r}^j\Im\dot{\phi}u  
\end{align}
from which we obtain
\begin{align}
    \dot{\bbQ}^{i j} = \frac{1}{\sqrt{\det \left( \frac{1}{2 \pi \ii} A \right)}} \Big[ \hat{r}^i \hat{r}^j \partial_t u &+ \omega^{-1} L_1 \big( \hat{r}^i \hat{r}^j \partial_t u \big) - 2 \omega \hat{r}^i \hat{r}^j u \Im \dot{\phi}\nonumber\\
    &- 2 L_1 \big( \hat{r}^i \hat{r}^j u \Im \dot{\phi} \big) - 2 \omega^{-1} L_2 \big( \hat{r}^i \hat{r}^j u \Im \dot{\phi} \big) \Big] \Big|_{\gamma(t)} + \mathcal{O}(\omega^{-2}) .
\end{align}
We analyse all of the above terms individually. Recall that $\Im \dot{\phi}|_{\gamma(t)} = 0$ and $\hat{r}^i|_{\gamma(t)} = \mathcal{O}(\omega^{-1})$. Furthermore, taking the imaginary part of the Eikonal equation~\eqref{eq:Eikonal} to degree $j_{\phi}=1$ gives $\nabla_a \Im \dot{\phi} \big|_\gamma = - \frac{\ii}{2 \n^2 \dot{\phi}} A_{a b} \nabla^b \phi \big|_\gamma = \frac{\ii}{2} A_{a b} \dot{\gamma}^b$. The first three terms are
\begin{subequations}
\begin{align}
    \hat{r}^i \hat{r}^j \partial_t u \Big|_{\gamma(t)} &= \mathcal{O}(\omega^{-2}), \\
    \omega^{-1} L_1 \big( \hat{r}^i \hat{r}^j \partial_t u \big) \Big|_{\gamma(t)} &= \frac{\ii \omega^{-1}}{2} (A^{-1})^{a b} \nabla_a \nabla_b  \left( \hat{r}^i \hat{r}^j \partial_t u \right) \Big|_{\gamma(t)} + \mathcal{O}(\omega^{-2}) \nonumber\\
    &= \ii \omega^{-1} (A^{-1})^{i j} \partial_t u \Big|_{\gamma(t)}  + \mathcal{O}(\omega^{-2}), \\
    - 2 \omega \hat{r}^i \hat{r}^j u \Im \dot{\phi} \Big|_{\gamma(t)} &= 0.
\end{align}
\end{subequations}
The fourth term is
\begin{align}
    - 2 L_1 \big( \hat{r}^i \hat{r}^j u \Im \dot{\phi} \big) \Big|_{\gamma(t)} &= - \ii (A^{-1})^{a b} \nabla_a \nabla_b \big( \hat{r}^i \hat{r}^j u \Im \dot{\phi} \big) \Big|_{\gamma(t)} + \mathcal{O}(\omega^{-2}) \nonumber \\
    &= 2 u \dot{\gamma}^{(i} \hat{r}^{j)} \Big|_{\gamma(t)} + \mathcal{O}(\omega^{-2}) \nonumber \\
    &= -2 \ii \omega^{-1} \dot{\gamma}^{(i} (A^{-1})^{j) a} \Big[ \nabla_a u - \ii u (A^{-1})^{b c} \nabla_a \nabla_b \nabla_c \Im \phi \Big] \Big|_{\gamma(t)} + \mathcal{O}(\omega^{-2}) ,
\end{align}
where the last equality was obtained using \cref{eq:X_hat} to evaluate $\hat{r}^j |_\gamma = \gamma^j - \hat{\bbX}^j$. The fifth term is
\begin{align} \label{eq:L2_terms}
    - 2 \omega^{-1} L_2 &\big( \hat{r}^i \hat{r}^j u \Im \dot{\phi} \big) \Big|_{\gamma(t)} = \frac{\omega^{-1}}{4} (A^{-1})^{a b}(A^{-1})^{cd}\nabla_{a}\nabla_b\nabla_c\nabla_d\big( \hat{r}^i \hat{r}^j u \Im \dot{\phi} \big) \Big|_{\gamma(t)} \nonumber\\
    &\qquad-\frac{\omega^{-1}}{24} (A^{-1})^{a b}(A^{-1})^{cd}(A^{-1})^{ef}\nabla_a\nabla_b\nabla_c\nabla_d\nabla_e\nabla_f\Big[ g\hat{r}^i \hat{r}^j u (\Im \dot{\phi}) \Big]\Big|_{\gamma(t)}+ \mathcal{O}(\omega^{-2}),
    \end{align}
    where 
    \begin{align}
        g(t,x)=2 \ii \Im \phi(t,x) - 2 \ii \Im \phi[t,\gammu(t)]-\frac{1}{2}A_{ab}(t)[x-\gammu(t)]^a [x-\gammu(t)]^b.
    \end{align}
    To obtain the above equation, we used \cref{eq:L2} and note that only the first two terms in \cref{eq:L2} have relevant contributions from \cref{rem:g} in combination with $\Im\phi|_{\gamma}=0=\Im\dot{\phi}|_{\gamma}$ and $\hat{r}=\mathcal{O}(\omega^{-1})$. 
    Note that a derivative must hit each $\hat{r}$, two derivatives must hit $\Im\phi$ since $\nabla\Im\phi|_{\gamma}=0$ by~\cref{eq:nabphireal}, and three derivatives must hit $g$ to give non-trivial contributions. Applying these same facts gives
    \begin{subequations}
    \begin{align}
    &\nabla_{a}\nabla_b\nabla_c\nabla_d\big( \hat{r}^i \hat{r}^j u \Im \dot{\phi} \big) \Big|_{\gamma(t)} 
    =-6\ii u\delta^i_{(b}\delta^j_{c}\partial_t\nabla_a\nabla_{d)}(2\ii\Im\phi)\Big|_{\gamma(t)} +12\ii\delta^i_{(b}\delta^j_{c}\nabla_{a}u A_{d)e}\dot{\gamma}^e\Big|_{\gamma(t)} + \mathcal{O}(\omega^{-1}),\\
    &\nabla_a\nabla_b\nabla_c\nabla_d\nabla_e\nabla_f\Big[ 2 \ii (\Im \phi) \hat{r}^i \hat{r}^j u (\Im \dot{\phi}) \Big]\Big|_{\gamma(t)} 
    =\frac{6!\ii}{2\cdot 3!} \nabla_{(a}\nabla_b\nabla_c(2 \ii\Im \phi) \delta_d^i \delta_e^j  A_{f)k}\dot{\gamma}^ku\Big|_{\gamma(t)} + \mathcal{O}(\omega^{-1}).
    \end{align}
    \end{subequations} 
    Thus, we can rewrite \cref{eq:L2_terms} as 
    \begin{align}
    - 2 \omega^{-1} L_2 &\big( \hat{r}^i \hat{r}^j u \Im \dot{\phi} \big) \Big|_{\gamma(t)} =\frac{\ii\omega^{-1}}{4} (A^{-1})^{(a d}(A^{-1})^{ef)}\delta^i_{a}\delta^j_{d}\Big[-6u\partial_t\nabla_{e}\nabla_f(2\ii\Im\phi)+12\nabla_{e}u A_{fk}\dot{\gamma}^k\Big]
         \Big|_{\gamma(t)}  \nonumber\\
    &\qquad - \frac{5\ii\omega^{-1}}{2}(A^{-1})^{(a b}(A^{-1})^{cd}(A^{-1})^{ef)} \nabla_{a}\nabla_b\nabla_c(2 \ii\Im \phi) \delta_d^i \delta_e^j  A_{fk}\dot{\gamma}^ku\Big|_{\gamma(t)}  + \mathcal{O}(\omega^{-2}),
    \end{align}
   where we used the fact that the symmetrisation extends to the contracted indices. Next, the identities
    \begin{subequations}
    \begin{align}
        (A^{-1})^{(a b}(A^{-1})^{cd)}&=\frac{1}{3} \Big[(A^{-1})^{a b}(A^{-1})^{cd}+(A^{-1})^{a c}(A^{-1})^{bd}+(A^{-1})^{a d}(A^{-1})^{bc} \Big],\\
       (A^{-1})^{(a b}(A^{-1})^{cd}(A^{-1})^{ef)} &=\frac{1}{5}
    \Big[3(A^{-1})^{a b}(A^{-1})^{(c d}(A^{-1})^{ef)}+2(A^{-1})^{cf}(A^{-1})^{ae}(A^{-1})^{bd}\Big],
    \end{align}
    \end{subequations}
    allow us to compute
     \begin{align}
    &- 2 \omega^{-1} L_2 \big( \hat{r}^i \hat{r}^j u \Im \dot{\phi} \big) \Big|_{\gamma(t)} = -\frac{\ii\omega^{-1}}{2} \Big[2(A^{-1})^{i a}(A^{-1})^{jb}+(A^{-1})^{ij }(A^{-1})^{ab}\Big] u \partial_t\nabla_{a}\nabla_b(2\ii\Im\phi)
         \Big|_{\gamma(t)} \nonumber\\
    &\qquad- \frac{\ii \omega^{-1}}{2} \Big\{(A^{-1})^{a b} \Big[2(A^{-1})^{c (i}\dot{\gamma}^{j)}+\dot{\gamma}^c(A^{-1})^{ij}\Big]+2\dot{\gamma}^c(A^{-1})^{bj}(A^{-1})^{ai}\Big\} u\nabla_{a}\nabla_b\nabla_{c}(2 \ii\Im \phi) \Big|_{\gamma(t)} \nonumber\\
    &\qquad+\ii\omega^{-1} \Big[(A^{-1})^{i a}\dot{\gamma}^j+(A^{-1})^{ja}\dot{\gamma}^i+(A^{-1})^{ij }\dot{\gamma}^a \Big]\nabla_{a}u 
         \Big|_{\gamma(t)}+ \mathcal{O}(\omega^{-2}).
    \end{align}
Putting everything together, we obtain
\begin{align}
    \dot{\bbQ}^{i j} = \frac{\ii \omega^{-1}}{\sqrt{\det \left( \frac{1}{2 \pi \ii} A \right)}} &\bigg\{ (A^{-1})^{i j} \left[ \left(\partial_t + \dot{\gamma}^a \nabla_a \right) u - \frac{u}{2} (A^{-1})^{b c} \left(\partial_t + \dot{\gamma}^a \nabla_a \right) \nabla_b \nabla_c (2\ii \Im \phi) \right] \nonumber\\
    &\qquad -u (A^{-1})^{i b} (A^{-1})^{j c} \left(\partial_t +  \dot{\gamma}^a \nabla_a \right) \nabla_b \nabla_c (2\ii \Im \phi) \bigg\}  \bigg|_{\gamma(t)} + \mathcal{O}(\omega^{-2}).
\end{align}
The term in square brackets is $\mathcal{O}(\omega^{-1})$ by \cref{Prop:conservation}, and we are left with
\begin{align}
    \dot{\bbQ}^{i j} &= - \frac{\ii \omega^{-1} u}{\sqrt{\det \left( \frac{1}{2 \pi \ii} A \right)}}  (A^{-1})^{i b} (A^{-1})^{j c} \left(\partial_t +  \dot{\gamma}^a \nabla_a \right) \nabla_b \nabla_c (2\ii \Im \phi) \Big|_{\gamma(t)} + \mathcal{O}(\omega^{-2}) \nonumber \\
    &=- \frac{\ii \omega^{-1} u}{\sqrt{\det \left( \frac{1}{2 \pi \ii} A \right)}}  (A^{-1})^{i b} (A^{-1})^{j c} \frac{d}{d t} A_{b c} \Big|_{\gamma(t)} + \mathcal{O}(\omega^{-2}) \nonumber \\
    &= \ii \omega^{-1} \hat{\mathbb{E}} \frac{d}{dt} (A^{-1})^{i j} + \mathcal{O}(\omega^{-2}).
\end{align}
where we use $\frac{d}{dt}A^{-1}=-A^{-1}\cdot \frac{d A}{dt} \cdot A^{-1} $ and \cref{eq:E_hat}. Finally, we apply \cref{PropGoOverToHat} to the term $\mathbb{E} = \hat{\mathbb{E}} + \mathcal{O}(\omega^{-2})$ to obtain
\begin{equation}
    \dot{\bbQ}^{i j} = \ii \omega^{-1} \mathbb{E} \frac{d}{dt} (A^{-1})^{i j} + \mathcal{O}(\omega^{-2}).
\end{equation}
\end{proof}

\appendix

\section{The stationary phase approximation}
\label{app:Stationary}

We collect here some general results regarding the stationary phase approximation, which are used in many parts of the paper. The proof of the following theorem can be found in \cite[Sec. 7.7]{hormander1983analysis}.

\begin{theorem}[Theorem 7.7.5 from \cite{hormander1983analysis}]
\label{th:stationary}

Let $K \subset \mathbb{R}^n$ be a compact set, $X$ an open neighbourhood of $K$ and $k$ positive integer. If $q \in C^{2k}_0(K)$, $f \in C^{3k+1}(X)$ and $\Im f \geq 0$ in $X$, $\Im f(x_s) = 0$, $\nabla_a f(x_s) = 0$, $\det \nabla_a \nabla_b f(x_s) \neq 0$, $\nabla_a f \neq 0$ in $K \setminus \{x_s\}$ then

\begin{equation}
    \bigg| \int_{\R^n} q(x) e^{\ii \omega f(x)} \, d^n x - \frac{e^{\ii \omega f(x_s)}}{ \sqrt{ \det \left( \frac{\omega}{2 \pi \ii} A \right) } } \sum_{j<k} \omega^{-j} L_j q  (x_s)\bigg| \leq C \omega^{-k} \sum_{|\alpha|\leq 2 k} \sup |\mathfrak{D}^\alpha q|, \qquad \omega > 0.
\end{equation}
In the above equation, $A_{a b} = \nabla_a \nabla_b f (x_s)$, $\mathfrak{D}_a = - \ii \nabla_a$, and
\begin{subequations} \label{eq:L_def}
\begin{align}
    L_j q (x_s) &= \sum_{\nu - \mu = j} \sum_{2\nu \geq 3 \mu} \frac{1}{\ii^j 2^\nu \mu! \nu!}  \left[ -(A^{-1})^{a b} \nabla_a \nabla_b \right]^\nu (g^\mu q) (x_s), \\ 
    g(x) &= f(x) - f(x_s) - \frac{1}{2}  \nabla_a \nabla_b f(x_s) (x - x_s)^a (x - x_s)^b. \label{eq:g_Def}
\end{align}    
\end{subequations}    

\end{theorem}

Note that the Hessian of $f$ is denoted above as $\nabla \nabla f$, while in \cite{hormander1983analysis} this is denoted by $f''$. The first four terms in the above expansion are
\begin{subequations} \label{eq:stationary_expansion}
\begin{align}
    L_0 q (x_s) &= q(x_s), \\
    L_1 q (x_s) &= \frac{\ii}{2} (A^{-1})^{a b} \nabla_a \nabla_b q (x_s) - \frac{\ii}{8} (A^{-1})^{a b} (A^{-1})^{c d} \nabla_a \nabla_b \nabla_c \nabla_d (g q) (x_s)  \nonumber \\
    &\qquad+ \frac{\ii}{96} (A^{-1})^{a b} (A^{-1})^{c d} (A^{-1})^{i j} \nabla_a \nabla_b \nabla_c \nabla_d \nabla_i \nabla_j (g^2 q) (x_s), \label{eq:L1} \\
    L_2 q(x_s) &= - \frac{1}{2^2 0! 2!} \left[(A^{-1})^{a b} \nabla_a \nabla_b \right]^2 ( q) (x_s) +  \frac{1}{2^3 1! 3!} \left[(A^{-1})^{a b} \nabla_a \nabla_b \right]^3 (g q) (x_s) \nonumber\\
    &\qquad- \frac{1}{2^4 2! 4!} \left[(A^{-1})^{a b} \nabla_a \nabla_b \right]^4 (g^2 q) (x_s) + \frac{1}{2^5 3! 5!} \left[(A^{-1})^{a b} \nabla_a \nabla_b \right]^5 (g^3 q) (x_s) \nonumber \\
    &\qquad- \frac{1}{2^6 4! 6!} \left[(A^{-1})^{a b} \nabla_a \nabla_b \right]^6 (g^4 q) (x_s), \label{eq:L2} \\
    L_3 q(x_s) &= - \frac{\ii}{2^3 0! 3!} \left[(A^{-1})^{a b} \nabla_a \nabla_b \right]^3 ( q) (x_s) + \frac{\ii}{2^4 1! 4!} \left[(A^{-1})^{a b} \nabla_a \nabla_b \right]^4 (g q) (x_s) \nonumber\\
    &\qquad- \frac{\ii}{2^5 2! 5!} \left[(A^{-1})^{a b} \nabla_a \nabla_b \right]^5 (g^2 q) (x_s) + \frac{\ii}{2^6 3! 6!} \left[(A^{-1})^{a b} \nabla_a \nabla_b \right]^6 (g^3 q) (x_s) \nonumber \\
    &\qquad- \frac{\ii}{2^7 4! 7!} \left[(A^{-1})^{a b} \nabla_a \nabla_b \right]^7 (g^4 q) (x_s) + \frac{\ii}{2^8 5! 8!} \left[(A^{-1})^{a b} \nabla_a \nabla_b \right]^8 (g^5 q) (x_s) \nonumber\\
    &\qquad - \frac{\ii}{2^9 6! 9!} \left[(A^{-1})^{a b} \nabla_a \nabla_b \right]^9 (g^6 q) (x_s). \label{eq:L3}
\end{align}    
\end{subequations}
\begin{remark}\label{rem:g}
    
Note that at least $3$ derivatives need to hit $g$ to get a non-zero term when evaluated at $x = x_s$. In particular, based on this property and the symmetry of the matrix $(A^{-1})^{a b}$, we have
\begin{align}
    L_1 q (x_s) &= \frac{\ii}{2} (A^{-1})^{a b} \nabla_a \nabla_b q (x_s) - \frac{\ii}{2} (A^{-1})^{a b} (A^{-1})^{c d} (\nabla_a q )(x_s) ( \nabla_b \nabla_c \nabla_d g ) (x_s)  \nonumber \\
    &\qquad- \frac{\ii}{8} q(x_s) (A^{-1})^{a b} (A^{-1})^{c d} (\nabla_a  \nabla_b \nabla_c \nabla_d g ) (x_s)  \nonumber \\
    &\qquad+ \frac{\ii}{96} q(x_s) (A^{-1})^{a b} (A^{-1})^{c d} (A^{-1})^{i j} \nabla_a \nabla_b \nabla_c \nabla_d \nabla_i \nabla_j (g^2) (x_s) .
\end{align}
\end{remark}

\section{Additional results and useful relations} \label{AppAdditionalResults}

We gather here some useful results that are needed for the main part of the paper.

\begin{lemma} \label{Lemma_E/2}
Suppose $(\phi,\e_0,\e_1, \h_0,\h_1)$ satisfies the assumptions of \cref{prop:2sttranstoMax}. Then, the following identity holds:
\begin{align}
    \mu\h\cdot \overline{\h} + \omega^{-1} L_1 \big(\mu \h\cdot\overline{\h} \big) \big|_{\gamma} = \eps \e \cdot \overline{\e} + \omega^{-1} L_1 \big( \eps \e\cdot\overline{\e} \big) \big|_{\gamma} +\mathcal{O}(\omega^{-2}).
\end{align}
In particular, 
\begin{align}
    \mu \h_0\cdot \overline{\h}_0 \big|_{\gamma} &= \eps \e_0\cdot\overline{\e}_0 \big|_{\gamma}, \qquad
    \nabla_a \big( \mu \h_0\cdot \overline{\h}_0 \big) \big|_{\gamma}= \nabla_a \big( \eps \e_0\cdot\overline{\e}_0 \big) \big|_{\gamma}.
\end{align}
\end{lemma}

\begin{proof}
First, recall the relations that we will use for this proof:
\begin{subequations}
\begin{align}
    \sqrt{\mu} \h_0^i \big|_\gamma &= - \frac{1}{\sqrt{\mu} \dot{\phi}} \epsilon^{i j k} \e_{0 k} \nabla_j \phi \bigg|_\gamma , \label{eq:dh0} \\
    \nabla_a \big( \sqrt{\mu} \h_0^i \big) \big|_\gamma &= - \frac{1}{\sqrt{\mu} \dot{\phi}} \epsilon^{i j k} \bigg[ \nabla_a \big( \e_{0 k} \nabla_j \phi \big) - \frac{1}{\sqrt{\mu} \dot{\phi}} ( \nabla_j \phi ) \e_{0 k} \nabla_a \big( \sqrt{\mu} \dot{\phi} \big) \bigg] \bigg|_\gamma, \\
    \nabla_a \nabla_b \big( \sqrt{\mu} \h_0^i \big) \big|_\gamma &= - \frac{1}{\sqrt{\mu} \dot{\phi}} \epsilon^{i j k} \bigg[ \nabla_a \nabla_b \big( \e_{0 k} \nabla_k \phi \big) - \frac{2}{\sqrt{\mu} \dot{\phi}} \nabla_{(a} \big( \sqrt{\mu} \dot{\phi} \big) \nabla_{b)} \big( \e_{0 k} \nabla_j \phi \big) \nonumber\\
    &\qquad+ \frac{2}{\mu \dot{\phi}^2} (\nabla_j \phi) \e_{0 k} \nabla_{(a} \big( \sqrt{\mu} \dot{\phi} \big) \nabla_{b)} \big( \sqrt{\mu} \dot{\phi} \big) - \frac{1}{\sqrt{\mu} \dot{\phi}} (\nabla_j \phi) \e_{0 k} \nabla_a \nabla_b \big( \sqrt{\mu} \dot{\phi} \big)  \bigg] \bigg|_\gamma, \\
    \nabla_a \big( \sqrt{\mu} \dot{\phi} \big) \big|_\gamma &= \frac{1}{\eps \sqrt{\mu} \dot{\phi}} (\nabla^i \phi) \nabla_a \nabla_i \phi - \frac{\sqrt{\mu} \dot{\phi}}{2} \nabla_a \ln \eps \bigg|_\gamma, \\
    \nabla_a \nabla_b \big( \sqrt{\mu} \dot{\phi} \big) \big|_\gamma &= \sqrt{\mu} \nabla_a \nabla_b \dot{\phi} + \frac{\dot{\phi}}{2 \sqrt{\mu}} \nabla_a \nabla_b \sqrt{\mu} - \frac{\sqrt{\mu} \dot{\phi}}{4} (\nabla_a \ln \mu) (\nabla_b \ln \mu) \nonumber\\
    &\qquad+ \frac{1}{\eps \dot{\phi}} \big( \nabla_{(a} \ln \mu \big) (\nabla^i \phi) \nabla_{b)} \nabla_i \phi - \frac{\mu \dot{\phi}}{2} \big( \nabla_{(a} \ln \mu \big) \nabla_{b)} \ln \eps \bigg|_\gamma. 
\end{align}    
\end{subequations}
We will also use the constraints
\begin{subequations}
\begin{align}
    \e_0^i \nabla_i \phi \big|_\gamma &= 0, \label{eq:e0_constraint}\\
    \nabla_a \big( \e_0^i \nabla_i \phi \big) \big|_\gamma &= 0 \quad \Rightarrow \quad (\nabla_i \phi) \nabla_a \e_0^i \big|_\gamma= - \e_0^i \nabla_a \nabla_i \phi \big|_\gamma.
\end{align}
\end{subequations}
The initial expression can be expanded as
\begin{align} \label{eq:h_expansion}
    \mu\h\cdot \overline{\h} + \omega^{-1} L_1 \big(\mu \h\cdot\overline{\h} \big) \big|_{\gamma} &= \mu\h_0 \cdot \overline{\h}_0 + 2 \omega^{-1} \Re \big( \mu\h_1 \cdot \overline{\h}_0 \big) + \frac{\ii \omega^{-1}}{2} (A^{-1})^{a b} \nabla_a \nabla_b \big( \mu\h_0 \cdot \overline{\h}_0 \big) \nonumber \\
    &\qquad - \frac{\ii \omega^{-1}}{2} (A^{-1})^{a b} (A^{-1})^{c d} \big[\nabla_a \big( \mu\h_0 \cdot \overline{\h}_0 \big) \big] \nabla_b \nabla_c \nabla_d g  \nonumber \\
    &\qquad + \frac{\ii \omega^{-1}}{96} \mu\h_0 \cdot \overline{\h}_0 (A^{-1})^{a b} (A^{-1})^{c d} (A^{-1})^{i j} \nabla_a \nabla_b \nabla_c \nabla_d \nabla_i \nabla_j g^2 \Big|_{\gamma}.
\end{align}
For the first and last terms of the above expansion, we can use \cref{eq:dh0} to obtain
\begin{align}
    \mu \h_0\cdot \overline{\h}_0 \big|_{\gamma} &= \frac{1}{\mu \dot{\phi}^2}{\epsilon^{ij}}_k{{\epsilon_i}^{m}}_n\e_0^k\overline{\e}_0^n (\nabla_j \phi) (\nabla_m \overline{\phi}) \big|_{\gamma} = \frac{1}{\mu \dot{\phi}^2} \big(\delta^{jm}\delta_{kn}-\delta^{j}_n\delta^{m}_k \big) \e_0^k \overline{\e}_0^n (\nabla_{j}\phi) (\nabla_{m}\overline{\phi}) \big|_{\gamma} \nonumber\\
    &=\frac{1}{\mu \dot{\phi}^2} \Big[ \e_0\cdot\overline{\e}_0 (\nabla_i\phi) (\nabla^i\overline{\phi}) -\e_0^k (\nabla_k\overline{\phi}) \overline{\e}_0^m (\nabla_m\phi) \Big]\Big|_{\gamma} = \eps \e_0\cdot\overline{\e}_0 \Big|_{\gamma}.\label{eq:h0h0btoe0e0b}
\end{align}
We used the fact that $\nabla_i \phi|_{\gamma}$ is $\R$-valued and, therefore, we can use the Eikonal equation~\eqref{eq:Eikonal} and the constraint~\eqref{eq:e0_constraint}.

In the fourth term in \cref{eq:h_expansion} we have
\begin{align}
    \nabla_a \big( \mu\h_0 \cdot \overline{\h}_0 \big) \big|_\gamma &= 2 \Re \big[ \sqrt{\mu} \overline{\h}_{0 i} \nabla_a \big(\sqrt{\mu} \h_0^i \big) \big] \big|_\gamma \nonumber \\
    &= \frac{2}{\mu \dot{\phi}^2} \Re \bigg\{ \epsilon_{i j k} \epsilon^{i m n} (\nabla^j \phi) \overline{\e}_{0}^k \bigg[ \nabla_a \big( \e_{0 k} \nabla_j \phi \big) - \frac{1}{\sqrt{\mu} \dot{\phi}} (\nabla_j \phi) \e_{0 k} \nabla_a \Big( \sqrt{\mu} \dot{\phi} \Big) \bigg]  \bigg\} \bigg|_\gamma \nonumber \\
    &= \frac{2}{\mu \dot{\phi}^2} \Re \bigg[ |\nabla \phi|^2 \overline{\e}_0^k \nabla_a \e_{0 k} + \overline{\e}_0^k \e_{0 k} (\nabla^j \phi) \nabla_a \nabla_j \phi \nonumber\\
    &\qquad\qquad\qquad- \frac{|\nabla \phi|^2 \overline{\e}_0^k \e_{0 k}}{\n^2 \dot{\phi}^2} (\nabla^j \phi) \nabla_a \nabla_j \phi + |\nabla \phi|^2 \overline{\e}_0^k \e_{0 k} \nabla_a \ln \eps \bigg] \bigg|_\gamma  \nonumber \\
    &=\nabla_a \big( \eps \e_0 \cdot \overline{\e}_0 \big) \big|_\gamma.
\end{align}
The remaining terms in \cref{eq:h_expansion} are
\begin{align}
    &2 \omega^{-1} \Re \big( \mu\h_1 \cdot \overline{\h}_0 \big) + \frac{\ii \omega^{-1}}{2} (A^{-1})^{a b} \nabla_a \nabla_b \big( \mu\h_0 \cdot \overline{\h}_0 \big) \big|_\gamma = \nonumber\\
    &\quad= 2 \omega^{-1} \Re \bigg\{ \mu \h_1 \cdot \overline{\h}_0 + \frac{1}{4} \left(\nabla \nabla \Im \phi ^{-1} \right)^{a b} \Big[ \nabla_a (\sqrt{\mu} h_{0 i}) \nabla_b (\sqrt{\mu} \overline{h}_0^i) + \sqrt{\mu} \overline{h}_0^i \nabla_a \nabla_b (\sqrt{\mu} h_{0 i}) \Big] \bigg\} \bigg|_\gamma.
\end{align}
We calculate the three terms in the above equation separately. For the first term, we have
\begin{align}
     \mu \h_1 \cdot \overline{\h}_0 \big|_\gamma &= - \frac{\ii}{\mu \dot{\phi}^2} \epsilon_{i j k} (\nabla^j \phi) \overline{\e}_0^k \Big[ \mu \dot{\h}_0^i + \epsilon^{i m n} \big( \nabla_m e_{0 n} + \ii \e_{1 n} \nabla_m \phi \big) \Big] \Big|_\gamma \nonumber\\
     &= \frac{\ii}{\mu \dot{\phi}^2} \bigg[ \frac{1}{\dot{\phi}} \e_0 \cdot \overline{\e}_0 (\nabla^a \phi) \nabla_a \dot{\phi} - \frac{\ddot{\phi} |\nabla \phi|^2 }{\dot{\phi}^2} \e_0 \cdot \overline{\e}_0 + \frac{|\nabla \phi|^2}{\dot{\phi}} \overline{\e} \cdot \dot{\e} \nonumber\\
     &\qquad\qquad- \overline{\e}_0^b (\nabla^a \phi) \nabla_a \e_{0 b} + \overline{\e}_0^a (\nabla^b \phi) \nabla_a \e_{0 b} - \ii |\nabla \phi|^2 \e_1 \cdot \overline{\e}_0 \bigg] \bigg|_\gamma \nonumber \\
     &= \eps \e_1 \cdot \overline{\e}_0 - \frac{\ii \eps}{|\nabla \phi|^2} \overline{\e}_0^a \e_0^b \nabla_a \nabla_b \phi - \frac{\ii \eps \n}{|\nabla \phi|} \overline{\e}_0^b \big(\partial_t + \dot{\gamma}^a \nabla_a \big) \e_{0 b} \bigg|_\gamma \nonumber \\
     &= \eps \e_1 \cdot \overline{\e}_0 + \frac{\ii \eps \e_0 \cdot \overline{\e}_0}{2 |\nabla \phi|^2} \bigg[ \delta^{i j} - \frac{(\nabla^i \phi) (\nabla^j \phi)}{|\nabla \phi|^2} - \frac{2 \e_0^{(i} \overline{\e}_0^{j)}}{\e_0 \cdot \overline{\e}_0} \bigg] \nabla_i \nabla_j \phi + \frac{\ii \eps \e_0 \cdot \overline{\e}_0}{2 |\nabla \phi|^2} (\nabla^i \phi) \left( \nabla_i \ln \n - \nabla_i \ln \mu \right) \bigg|_\gamma.
\end{align}
Note that the last term in the above equation will vanish when taking the real part. The second term is
\begin{align}
    &\nabla_a \big(\sqrt{\mu} \h_{0 i} \big) \nabla_b \big( \sqrt{\mu} \overline{\h}_0^i \big) \big|_\gamma = \frac{\epsilon_{i j k} \epsilon^{i m n} }{\mu \dot{\phi}^2} \bigg[ \nabla_a \big( \e_0^k \nabla^j \phi \big) - \frac{1}{\sqrt{\mu} \dot{\phi}} (\nabla^j \phi) \e_0^k \nabla_a \big( \sqrt{\mu} \dot{\phi} \big) \bigg] \nonumber\\
    &\qquad\qquad\qquad\qquad\qquad\qquad\qquad\qquad \times \bigg[ \nabla_a \big( \overline{\e}_{0 n} \nabla_m \overline{\phi} \big) - \frac{1}{\sqrt{\mu} \dot{\overline{\phi}}} (\nabla_m \overline{\phi} ) \overline{\e}_{0 n} \nabla_a \big( \sqrt{\mu} \dot{\phi} \big) \bigg] \bigg|_\gamma \nonumber \\
    &\qquad\qquad= \nabla_a \big(\sqrt{\eps} \e_{0 i} \big) \nabla_b \big(\sqrt{\eps} \overline{\e}_0^i \big) + \frac{\eps \e_0 \cdot \overline{\e}_0}{|\nabla \phi|^2} \bigg[ \delta^{i j} - \frac{(\nabla^i \phi) (\nabla^j \phi)}{|\nabla \phi|^2} - \frac{2 \e_0^{(i} \overline{\e}_0^{j)}}{\e_0 \cdot \overline{\e}_0} \bigg] (\nabla_a \nabla_i \phi) (\nabla_b \nabla_j \overline{\phi}) \bigg|_\gamma.
\end{align}
Similarly, for the third term, we obtain
\begin{align}
    \sqrt{\mu} \overline{\h}_0^i \nabla_a \nabla_b \big( \sqrt{\mu} \h_{0 i} \big) \big|_\gamma &= \sqrt{\eps} \overline{\e}_0^i \nabla_a \nabla_b \big( \sqrt{\eps} \e_{0 i} \big) - \frac{\eps \e_0 \cdot \overline{\e}_0}{|\nabla \phi|^2} \bigg[ \delta^{i j} - \frac{(\nabla^i \phi) (\nabla^j \phi)}{|\nabla \phi|^2} - \frac{2 \e_0^{(i} \overline{\e}_0^{j)}}{\e_0 \cdot \overline{\e}_0} \bigg] (\nabla_a \nabla_i \phi) (\nabla_b \nabla_j \phi) \bigg|_\gamma.
\end{align}
Bringing these three terms together, we obtain
\begin{align}
    &2 \omega^{-1} \Re \big( \mu\h_1 \cdot \overline{\h}_0 \big) + \frac{\ii \omega^{-1}}{2} (A^{-1})^{a b} \nabla_a \nabla_b \big( \mu\h_0 \cdot \overline{\h}_0 \big) \Big|_\gamma = 2 \omega^{-1} \Re \big( \eps\e_1 \cdot \overline{\e}_0 \big) + \frac{\ii \omega^{-1}}{2} (A^{-1})^{a b} \nabla_a \nabla_b \big( \eps \e_0 \cdot \overline{\e}_0 \big) \nonumber\\
    &\qquad +2 \omega^{-1} \Re \bigg\{ \frac{\ii \eps \e_0 \cdot \overline{\e}_0}{2 |\nabla \phi|^2} \bigg[ \delta^{i j} - \frac{(\nabla^i \phi) (\nabla^j \phi)}{|\nabla \phi|^2} - \frac{2 \e_0^{(i} \overline{\e}_0^{j)}}{\e_0 \cdot \overline{\e}_0} \bigg] \nabla_i \nabla_j \phi \nonumber\\
    &\qquad\qquad\qquad + \frac{\eps \e_0 \cdot \overline{\e}_0}{4 |\nabla \phi|^2} \bigg[ \delta^{i j} - \frac{(\nabla^i \phi) (\nabla^j \phi)}{|\nabla \phi|^2} - \frac{2 \e_0^{(i} \overline{\e}_0^{j)}}{\e_0 \cdot \overline{\e}_0} \bigg]\left(\nabla \nabla \Im \phi ^{-1} \right)^{a b}  (\nabla_a \nabla_i \phi) \nabla_b \nabla_j (\overline{\phi} - \phi) \bigg\} \bigg|_\gamma.
\end{align}
In the above equation, we have $\nabla_b \nabla_j (\overline{\phi} - \phi) = -2\ii \nabla_b \nabla_j \Im \phi$, and the term in the curly brackets vanishes. This completes the proof.
\end{proof}

\begin{proposition}\label{prop:timeder}
 Suppose that the following equality holds on $\gamma$ to some degree $j$:
    \begin{align} \label{EqAB}
        D^{\alpha}A \big|_{\gamma}=D^{\alpha}B \big|_{\gamma} \qquad \forall |\alpha| \leq j.
    \end{align}
    Then, for all $|\beta|\leq j$, we can compute single time derivatives of $A$ along $\gamma$ via
    \begin{align}\partial_t D^{\beta} A \big|_{\gamma} =
    \begin{cases}
        \partial_t D^{\beta}B \big|_\gamma \qquad  &|\beta|\leq j-1,\\
        \partial_t D^{\beta}B + \dot{\gamma}^i\nabla_i D^{\beta}(B-A) \big|_{\gamma}\qquad &|\beta|=j.
    \end{cases}
    \end{align}
    Suppose $j\geq 2$. Then 
    \begin{align}
        \partial_t^2A \big|_{\gamma}=\partial_t^2B \big|_{\gamma}.
    \end{align}
\end{proposition} 
\begin{proof}  This follows from a careful but elementary computation
\begin{align}
    \partial_tA=\bigg(\partial_t-\frac{\nabla^k\phi}{\n^2\dot{\phi}}\nabla_k\bigg)A-\bigg(-\frac{\nabla^k\phi}{\n^2\dot{\phi}}\bigg)\nabla_kA.\label{eq:t_der_A}
\end{align}
Evaluating on $\gamma$ and using \cref{EqAB} gives
    \begin{align}
        \partial_t A \big|_{\gamma}&= \big( \partial_t+\dot{\gamma}^i\nabla_i \big) A -\dot{\gamma}^i \nabla_i A \big|_{\gamma} =\big(\partial_t+\dot{\gamma}^i\nabla_i \big) B -\dot{\gamma}^i \nabla_i B \big|_{\gamma} = \rd_tB\big|_\gamma.
    \end{align}
For second time derivatives, we note the formula, 
\begin{align}
    \partial_t^2A&=\bigg(\partial_t-\frac{\nabla^m\phi}{\n^2\dot{\phi}}\nabla_m\bigg)\bigg[\bigg(\partial_t-\frac{\nabla^k\phi}{\n^2\dot{\phi}}\nabla_k\bigg)A-\bigg(-\frac{\nabla^k\phi}{\n^2\dot{\phi}}\bigg)\nabla_kA\bigg] -\bigg(-\frac{\nabla^m\phi}{\n^2\dot{\phi}}\bigg)\partial_t\nabla_mA \nonumber\\
    &=\bigg(\partial_t-\frac{\nabla^m\phi}{\n^2\dot{\phi}}\nabla_m\bigg)\big(\partial_t-\frac{\nabla^k\phi}{\n^2\dot{\phi}}\nabla_k\bigg)A-\nabla_kA\bigg(\partial_t-\frac{\nabla^m\phi}{\n^2\dot{\phi}}\nabla_m\bigg)\bigg(-\frac{\nabla^k\phi}{\n^2\dot{\phi}}\bigg) \nonumber\\
    &\qquad -2\bigg(-\frac{\nabla^k\phi}{\n^2\dot{\phi}}\bigg)\bigg(\partial_t-\frac{\nabla^m\phi}{\n^2\dot{\phi}}\nabla_m\bigg)\nabla_kA+\bigg(-\frac{\nabla^m\phi}{\n^2\dot{\phi}}\bigg)\bigg(-\frac{\nabla^k\phi}{\n^2\dot{\phi}}\bigg)\nabla_k\nabla_mA,
\end{align}
where we use~\cref{eq:t_der_A}. Evaluating on $\gamma$ and using $D^{\alpha}A=D^{\alpha}B$ for $|\alpha|\leq 2$ gives
    \begin{align}
        \partial_t^2A \big|_\gamma &= \big(\partial_t+\dot{\gamma}^m\nabla_m\big) \big(\partial_t+\dot{\gamma}^k\nabla_k\big) A -2\dot{\gamma}^k \big(\partial_t+\dot{\gamma}^m\nabla_m\big) \nabla_k A+\dot{\gamma}^k\dot{\gamma}^m\nabla_k \nabla_m A\nonumber\\
        &\qquad-\big( \nabla_k A \big)\big(\partial_t+\dot{\gamma}^m\nabla_m\big) \bigg(-\frac{\nabla^k\phi}{\n^2\dot{\phi}}\bigg) \bigg|_\gamma  \nonumber\\
        &= \big(\partial_t+\dot{\gamma}^m\nabla_m\big) \big(\partial_t+\dot{\gamma}^k\nabla_k\big) B -2\dot{\gamma}^k \big(\partial_t+\dot{\gamma}^m\nabla_m \big) \nabla_k B +\dot{\gamma}^k\dot{\gamma}^m\nabla_k\nabla_m B\nonumber\\
        &\qquad-\big( \nabla_k B \big) \big(\partial_t+\dot{\gamma}^m\nabla_m\big) \bigg(-\frac{\nabla^k\phi}{\n^2\dot{\phi}}\bigg) \bigg|_\gamma = \rd_t^2 B\big|_\gamma ,
    \end{align}
    where we have used \cref{EqAB}.  
\end{proof}

\begin{proposition}\label{prop:2ndderphi}
 Suppose that the eikonal equation holds to degree $1$ on $\gamma$. Then, we have the following identities for $2^{\mathrm{nd}}$-derivatives of $\phi$:
 \begin{subequations}
 \begin{align}
        \ddot{\phi} \big|_{\gamma} &=\frac{\nabla^i\phi}{\n^2 \dot{\phi}}\nabla_i\dot{\phi}\Big|_{\gamma}\\
        \nabla_i\dot{\phi} \big|_{\gamma} &= \frac{(\nabla^j\phi)(\nabla_i\nabla_j{\phi})}{\n^2\dot{\phi}} - \dot{\phi} \frac{\nabla_i \n^2}{2\n^2}\Big|_{\gamma},\\
        \big(\partial_t+\dot{\gamma}^i\nabla_i\big) \dot{\phi} \big|_{\gamma} &=0,\\
        \big(\partial_t+\dot{\gamma}^i\nabla_i\big)\nabla_j{\phi} \big|_{\gamma}&=- \dot{\phi} \frac{\nabla_{j} \n^2}{\n^2} \Big|_\gamma.
    \end{align}    
 \end{subequations}
\end{proposition}

\begin{proof}
Using proposition~\ref{prop:timeder}, we can compute
\begin{subequations}
\begin{align}
    \partial_t \big(\nabla\phi\cdot\nabla\phi-\n^2\dot{\phi}^2\big) \big|_\gamma &= 0,\\
    \nabla_i \big(\nabla\phi\cdot\nabla\phi-\n^2\dot{\phi}^2 \big) \big|_\gamma &= 0,
\end{align}
\end{subequations}
which gives the first two results. The latter two are derived from the first two. 
\end{proof}

The classical version of Borel's lemma allows one to specify derivatives of a function at a point and extend it to a smooth function globally:
\begin{lemma}[Classical Borel's Lemma]
    Given, for each $n$-tuple $\alpha\in \N^n$, a constant $c_{\alpha}\in \R$. There exists a function $f\in C^{\infty}(\R^n)$ such that
    \begin{align}
        [D^{\alpha}f](0)=c_{\alpha}.
    \end{align}
\end{lemma}

However, we require a simpler version of the lemma:

\begin{lemma}[Borel's Lemma II]\label{lem:Borel}
   Let $0\leq N<\infty$ and $\gamma:\R\rightarrow \R^4$ be a smooth curve parametrised by $t\in \R$ such that $\gamma(t)=(t,\gammu(t))$. Given, for each $3$-tuple $\alpha\in \N^3$ with $|\alpha|\leq N$, a  $c_{\alpha}\in C^{\infty} (\R)$. There exists a function $f\in C^{\infty}(\R^{3+1})$ such that
    \begin{align}
        [D^{\alpha}f](t,\gammu(t))=c_{\alpha}(t).
    \end{align}
\end{lemma}
\begin{proof}
    Define 
    \begin{align}
        f(t,\x):= \sum_{|\alpha|\leq N}\frac{1}{\alpha!}c_{\alpha}(t)[\x-\gammu(t)]^{\alpha}.
    \end{align}
    This is clearly smooth and satisfies the required equality. 
\end{proof}

\section{Derivation of the Gaussian beam equations} \label{AppDerGB}

Here we define:
\begin{align} 
   \mathrm{Eikonal}[\alpha]&:= D^{\alpha} \Big( \nabla\phi \cdot \nabla \phi-\n^2\dot{\phi}^2 \Big) \Big|_{\gamma}\\
    (e_A^n-\mathrm{transport})[\alpha]&:=\Big( \partial_t+\dot{\gamma}^m\nabla_m \Big)D^{\alpha}\e^n_{A}  \Big|_{\gamma} - D^{\alpha} \bigg\{ \frac{1}{2\n^2\dot{\phi}} \bigg[ \e^n_{A} \Big( \Delta\phi - \n^2 \ddot{\phi} \Big) - \ii \Big( \Delta\e^n_{A-1}  - \n^2 \ddot{\e}^n_{A-1} \Big)\nonumber\\
    &\quad - \ii \e_{A-1}^m\nabla_n\nabla_m\ln\eps + \Big( \e^{m}_{A}\nabla^{n}\phi - \ii \nabla^n\e^m_{A-1}\Big) \nabla_m \ln \n^2 - \Big( \e^{n}_{A}\nabla^{m}\phi- \ii \nabla^m\e^n_{A-1}\Big) \nabla_m\ln\mu \bigg] \bigg\}\bigg|_{\gamma}\nonumber\nonumber\\
    &\quad-\sum_{0<\beta\leq \alpha}\binom{\alpha}{\beta} D^{\beta} \bigg(\frac{\nabla^m\phi}{\n^2\dot{\phi}} \bigg) \nabla_m D^{\alpha-\beta}\e_A^n \bigg|_{\gamma},\\
    (h_A^n-\mathrm{transport})[\alpha]&:=\Big( \partial_t+\dot{\gamma}^m\nabla_m \Big)D^{\alpha}\h^n_{A}  \Big|_{\gamma} - D^{\alpha} \bigg\{ \frac{1}{2\n^2\dot{\phi}} \bigg[ \h^n_{A} \Big( \Delta\phi - \n^2 \ddot{\phi} \Big) - \ii \Big( \Delta\h^n_{A-1}  - \n^2 \ddot{\h}^n_{A-1} \Big)\nonumber\\
    &\quad - \ii \h_{A-1}^m\nabla_n\nabla_m\ln\mu + \Big( \h^{m}_{A}\nabla^{n}\phi - \ii \nabla^n\h^m_{A-1}\Big) \nabla_m \ln \n^2 - \Big( \h^{n}_{A}\nabla^{m}\phi- \ii \nabla^m\h^n_{A-1}\Big) \nabla_m\ln\eps \bigg] \bigg\}\bigg|_{\gamma}\nonumber\nonumber\\
    &\quad-\sum_{0<\beta\leq \alpha}\binom{\alpha}{\beta} D^{\beta} \bigg(\frac{\nabla^m\phi}{\n^2\dot{\phi}} \bigg) \nabla_m D^{\alpha-\beta}\h_A^n \bigg|_{\gamma}.
    \end{align}
    The definition of the Eikonal equation \eqref{eq:Eikonal} and the $e_A$-transport equation \eqref{eq:eAtransport} now read as
    \begin{align}
        \mathrm{Eikonal}[\alpha]&=0\qquad \forall |\alpha|\leq j_{\phi}\\
        (e_A^n-\mathrm{transport})[\alpha]&=0\qquad \forall |\alpha|\leq j_{A}.
    \end{align}
Furthermore, we say that $\h_A$ satisfies the $\h_A$-transport equation along $\gamma$ to degree $j_{A}$ if
\begin{align}\label{eq:hAtransport}
        (h_A^n-\mathrm{transport})[\alpha]&=0\qquad \forall |\alpha|\leq j_{A}.
    \end{align}
The following algebraic relations between those expressions exist:

\begin{lemma}\label{lem:decompGF}
    The following expressions hold:
    \begin{subequations}
    \begin{align}
        &\bG_A^i\nabla_i\phi=\ii \div\bG_{A-1}-\dot{\phi}\varepsilon\bC_A+\ii \varepsilon\partial_t\bC_{A-1},\\
        &\bF_A^i\nabla_i\phi= \ii \div\bF_{A-1}-\dot{\phi}\mu\bK_A+\ii \mu\partial_t\bK_{A-1}, \label{EqUseForK}\\
    &\eps\dot{\phi} \bF_{A}^n-(\star\bG_{A})^{mn}\nabla_m\phi-\bK_{A}\nabla^n\phi+ \ii \nabla^n\bK_{A-1}+\eps\div \Big(\frac{\ii}{\eps}\star\bG_{A-1}^{n}\Big) -\ii \partial_t \bF_{A-1}^n \nonumber\\
    &\qquad=2\n^2\dot{\phi}\Big[(\h_{A-1}^n-\mathrm{transport})[0]\Big]-\ii\h_{A}^n\Big(\mathrm{Eikonal}[0]\Big),\\
    &\mu\dot{\phi} \bG_{A}^n-(\star\bF_{A})^{mn}\nabla_m\phi-\bC_{A}\nabla^n\phi+ \ii\nabla^n\bC_{A-1}+\mu\div\Big(\frac{\ii}{\mu}\star\bF_{A-1}^{n}\Big)-\ii\partial_t \bF_{A-1}^n \nonumber\\
    &\qquad=2\n^2 \dot{\phi}\Big[(\e_{A-1}^n-\mathrm{transport})[0]\Big]-\ii\e_{A}^n\Big(\mathrm{Eikonal}[0]\Big). \label{EqUseForG}
\end{align}
\end{subequations}
\end{lemma}
\begin{proof}
We begin by computing directly that
\begin{align}
    \bG_A^i\nabla_i\phi&=-\ii\eps \dot{\phi} \e^i_{A}\nabla_i\phi+{\epsilon^{ij}}_k\nabla_j\h^k_{A-1}\nabla_i\phi-\eps\dot{\e}^i_{A-1}\nabla_i\phi,
\end{align}
by the antisymmetry of $\epsilon^{ijk}$. We now compute from $\div\bG_{A-1}$:
\begin{align}
    {\epsilon^{ij}}_k\nabla_j\h_{A-1}^k\nabla_i\phi=\ii\div\bG_{A-1}-\eps \dot{\phi} \div\e_{A-1}-\eps\e_{A-1}^i\nabla_i\dot{\phi}-\e_{A-1}^i\dot{\phi} \nabla_i\eps +\ii\div(\eps\dot{\e}_{A-2}). 
\end{align}
Substituting in gives
\begin{align}
    \bG_A^i\nabla_i\phi&=\ii\div\bG_{A-1}-\Big[\div(\eps\e_{A-1})+\ii\eps \e^i_{A}\nabla_i\phi\Big]\dot{\phi}+\ii\partial_t\Big[\div(\eps\e_{A-2})+\ii\eps\e^i_{A-1}\nabla_i\phi\Big],
\end{align}
which gives the result. 

Taking the dual of $\bG_{A+1}$ gives, 
\begin{align}
\star\bG_{A+1}^{mn}=\big(\nabla^m\h^n_{A}+\ii \h^n_{A+1}\nabla^m\phi\big) -\big(\nabla^n\h^m_{A}+\ii\h^m_{A+1} \nabla^n\phi \big)-{\epsilon_{i}}^{mn}\eps\dot{\e}^i_{A}-\ii{\epsilon_{i}}^{mn}\e^i_{A+1}\eps \dot{\phi}.
\end{align}
Contracting with $\nabla_m\phi$ yields
\begin{align}
\star\bG_{A+1}^{mn}\nabla_m\phi&=\big(\nabla_m\phi\nabla^m\h^n_{A}+ \ii \h^n_{A+1}\nabla\phi\cdot\nabla\phi \big) -\big(\nabla_m\phi\nabla^n\h^m_{A}+\ii \h^m_{A+1}\nabla_m\phi \nabla^n\phi \big)+{\epsilon^{nm}}_{i}\nabla_m\phi\eps\dot{\e}^i_{A} \nonumber\\
&\qquad+\ii{\epsilon^{nm}}_{i}\e^i_{A+1}\nabla_m\phi \eps \dot{\phi}.
\end{align}
We use $\bF_A$ as
\begin{align}
\ii{\epsilon^{nj}}_k\e^k_{A+1}\nabla_j\phi=\bF_{A+1}^n-\mu\dot{\h}^n_{A}-\ii\mu \h^n_{A+1}\dot{\phi}-{\epsilon^{nj}}_k\nabla_j\e^k_{A}
\end{align}
to replace the last term above. This yields
\begin{align}
(\star\bG_{A+1})^{mn}\nabla_m\phi&=\nabla_m\phi\nabla^m\h^n_{A}+\ii \h^n_{A+1} \big(\nabla\phi\cdot\nabla\phi-\n^2 \dot{\phi}^2 \big)-\big(\nabla_m\phi\nabla^n\h^m_{A}+ \ii\h^m_{A+1}\nabla_m\phi \nabla^n\phi \big) \nonumber\\
&\qquad+{\epsilon^{nm}}_{i}\nabla_m\phi\eps\dot{\e}^i_{A}+\big(\bF_{A+1}^n-\mu\dot{\h}^n_{A}-{\epsilon^{nj}}_k\nabla_j\e^k_{A} \big) \eps \dot{\phi}.
\end{align}
We compute $\partial_t(-\ii\bF_A)$:
\begin{align}
    \partial_t(-\ii\bF_A^n)={\epsilon^{nj}}_k\dot{\e}^k_{A}\nabla_j\phi+{\epsilon^{nj}}_k\e^k_{A}\nabla_j\dot{\phi}+\mu \dot{\h}^n_{A}\dot{\phi}+\mu \h^n_{A}\ddot{\phi}-\ii{\epsilon^{nj}}_k\nabla_j\dot{\e}^k_{A-1}-\ii\mu\ddot{\h}^n_{A-1}.
\end{align}
This yields, when combined with $\nabla_m\phi(\star \bG_{A+1})^{mn}$, 
\begin{align}
& \partial_t(-\ii\bF_A^n)+\nabla_m\phi\nabla^m\h^n_{A}+\ii\h^n_{A+1} \big(\nabla\phi\cdot\nabla\phi-\n^2\dot{\phi}^2 \big) -\nabla_m\phi\nabla^n\h^m_{A}-\ii \h^m_{A+1}\nabla_m\phi \nabla^n\phi+\bF_{A+1}^n\eps\dot{\phi}\\
&\qquad-\mu\dot{\h}^n_{A} \eps \dot{\phi}-(\star\bG_{A+1})^{mn}\nabla_m\phi-\n^2 \dot{\h}^n_{A}\dot{\phi}-\n^2 \h^n_{A}\ddot{\phi}+\ii\eps\mu\ddot{\h}^n_{A-1} \nonumber\\
&=\eps\nabla_j \big({\epsilon^{nj}}_k\e^k_{A}\dot{\phi}\big)-\ii{\epsilon^{nj}}_k\eps\nabla_j\dot{\e}^k_{A-1}.
\end{align}
Returning to $\star\bG_A$, we now write
\begin{align}
{\epsilon^{nm}}_{i}\e^i_{A} \dot{\phi}=\frac{\ii}{\eps} \big(\nabla^m\h^n_{A-1}+\ii \h^n_{A}\nabla^m\phi \big) -\frac{\ii}{\eps} \big(\nabla^n\h^m_{A-1}+\ii\h^m_{A}\nabla^n\phi \big)-\ii{\epsilon_{i}}^{mn}\dot{\e}^i_{A-1}-\frac{\ii}{\eps}\star\bG_{A}^{mn}.
\end{align}
Taking a divergence then gives
\begin{align}
\eps \nabla_m\big({{\epsilon^{nm}}_{i}}\e^i_{A}\dot{\phi}\big)-  \ii{\epsilon_{i}}^{nm}\eps\nabla_m\dot{\e}^i_{A-1}&=-\ii(\nabla_m\ln\eps) \big(\nabla^m\h^n_{A-1}+\ii\h^n_{A}\nabla^m\phi-\nabla^n\h^m_{A-1}-\ii\h^m_{A}\nabla^n\phi \big) \nonumber\\
&\qquad+\ii\Delta\h^n_{A-1}-\h^n_{A}\Delta\phi-(\nabla_m\h^n_{A})\nabla^m\phi+(\nabla^n\phi)\div\h_{A} \nonumber\\
&\qquad-\ii\nabla^n \big(\bK_A-\h_{A-1}^i\nabla_i\ln\mu \big)-(\nabla^n \h^m_{A})\nabla_m\phi -\eps\div\Big(\frac{\ii}{\eps}\star\bG_{A}^{n}\Big).
\end{align}
This produces
\begin{align}
& \bF_{A+1}^n\eps\dot{\phi}-(\star\bG_{A+1})^{mn}\nabla_m\phi-\bK_{A+1}\nabla^n\phi+\ii\nabla^n\bK_A+\eps\div\Big(\frac{\ii}{\eps}\star\bG_{A}^{n}\Big)-\ii\partial_t \bF_A^n \nonumber\\
&=2\n^2\dot{\phi}\bigg(\partial_t-\frac{\nabla^m\phi}{\n^2\dot{\phi}} \nabla_m \bigg)\h^n_{A}-\h^n_{A}\Delta\phi+\ii\h_{A-1}^i\nabla^n\nabla_i\ln\mu+\n^2 \h^n_{A}\ddot{\phi}-\ii\n^2\ddot{\h}^n_{A-1}+\ii\Delta\h^n_{A-1} \nonumber\\
&\qquad-\ii(\nabla_m\ln\eps) \big(\nabla^m\h^n_{A-1}+\ii\h^n_{A}\nabla^m\phi\big)+\ii\big(\nabla^n\h_{A-1}^i\big)\nabla_i\ln \n^2-\h^m_{A}(\nabla^n\phi)\nabla_m\ln \n^2 \nonumber\\
&\qquad-\ii\h^n_{A+1}\big(\nabla\phi\cdot\nabla\phi-\n^2 \dot{\phi}^2\big).
\end{align}
\end{proof}

\begin{proposition}
 Let $N\geq1$. Suppose that for each $ 0\leq A\leq N-1$ our Gaussian beam approximation satisfies
    \begin{subequations}
    \begin{align}
       D^{\alpha}\bG_A \big|_{\gamma}&=0,\label{eq:dyneA}\\
       D^{\alpha}\bF_A \big|_{\gamma}&=0,\label{eq:dynhA}
    \end{align}
    \end{subequations}
    for all $0\leq|\alpha|\leq N-1-A$.
    Then, for each $ 0\leq A\leq N-1$
    \begin{subequations}
    \begin{align}
         D^{\alpha}\bC_A \big|_{\gamma}&=0,\label{eq:consteA}\\
         D^{\alpha}\bK_A \big|_{\gamma}&=0,\label{eq:consthA}
    \end{align}        
    \end{subequations}
    for all $0\leq|\alpha|\leq N-1-A$.
    Moreover, the Eikonal equation~\eqref{eq:Eikonal} vanishes to degree $N-1$ and, for each $0\leq A\leq N-1$, the $\e_A$-transport equation~\eqref{eq:eAtransport} vanishes to degree $N-2-A$. Finally, for each $0\leq A\leq N-1$, the $\h_A$-transport equation~\cref{eq:hAtransport} vanishes to degree $N-2-A$. 
\end{proposition}
\begin{proof}
We appeal to Lemma~\ref{lem:decompGF} and Proposition~\ref{prop:timeder} and proceed by induction with a careful counting of derivatives. 
\end{proof}

\section{Auxiliary computations} \label{app:D}

\subsection{Computation of initial average quantities and multipole moments}\label{sec:compute_initial_quantities_GB_ID}
\begin{proof}[Proof of~\cref{PropInitialValueAveragedQuantities}]
Consider the definitions of total energy, energy centroid, total linear momentum, total angular momentum, and quadrupole moment given in \cref{sec:def_average_quantities}. The energy density and Poynting vector of the initial data in \cref{DefGBID} are\footnote{The terms $\mathcal{O}_{L^1(\R^3)}(\omega^{-2})$ are for example obtained as follows: $$||\omega^{\frac{3}{4}} \eu_0 e^{i \omega \phu} \cdot \mathcal{O}_{L^2(\R^3)}(\omega^{-2}) ||_{L^1(\R^3)} \leq \underbrace{\Big( \int_{\R^3} \omega^{\frac{3}{2}} |\eu_0|^2 e^{2 \omega \Im \phu}\, d^3x \Big)^{\frac{1}{2}}}_{= \mathcal{O}(1)} \cdot \omega^{-2} \;,$$ where we have applied \cref{th:stationary} to the underbraced term.}
\begin{subequations}
\begin{align}
    \boldsymbol{\mathcal{E}} &= \omega^{\nicefrac{3}{2}} \uu  e^{-2 \omega \Im \phu} \nonumber\\
    &\qquad+ \frac{\omega^{\nicefrac{3}{2}}}{4} \Re \Big\{ \Big[ \varepsilon \eu_0 \cdot \eu_0 + \mu \hu_0 \cdot \hu_0 + 2 \omega^{-1} \big( \varepsilon \eu_0 \cdot \eu_1 + \mu \hu_0 \cdot \hu_1 \big)  \Big] e^{2 \ii  \omega \Re \phu} \Big\} e^{-2 \omega \Im \phu}  +\mathcal{O}_{L^1(\R^3)}(\omega^{-2}), \\
    \boldsymbol{\mathcal{S}} &= \omega^{\nicefrac{3}{2}} \vu  e^{-2\omega  \Im \phu} + \frac{\omega^{\nicefrac{3}{2}}\n^2}{2}  \Re \Big\{ \Big[ \eu_0 \times \hu_0 + \omega^{-1} \big( \eu_0 \times \hu_1 + \eu_1 \times \hu_0 \big)  \Big] e^{2 \ii \omega \Re \phu} \Big\} e^{-2\omega  \Im \phu} \nonumber \\
    &\qquad + \mathcal{O}_{L^1(\R^3)}(\omega^{-2}).
\end{align}    
\end{subequations}
The integrals of the terms in the above equations that are proportional to $e^{\pm 2 \ii \omega \Re \phi}$ decay to arbitrary high order in $\omega$ by \cite[Th. 7.7.1]{hormander1983analysis}. For the integrals of the remaining terms, we apply \cref{th:stationary}, which gives the expressions in \cref{eq:initial_quantities}.
In particular, we consider the following integrals in the context of Theorem~\ref{th:stationary}:
\begin{subequations}
\begin{align}
    \mathbb{E}(0) &= \bbE =  \int_{\R^3} \boldsymbol{\mathcal{E}} \, d^3x=\int_{\R^3} \omega^{\nicefrac{3}{2}} \uu  e^{-2 \omega \Im \phu} \, d^3x + \mathcal{O}(\omega^{-2}),\\
    \mathbb{X}^i(0) &= \frac{1}{\mathbb{E}} \int_{\R^3} x^i \boldsymbol{\mathcal{E}} \, d^3x=\frac{1}{\mathbb{E}} \int_{\R^3} x^i \omega^{\nicefrac{3}{2}} \uu  e^{-2 \omega \Im \phu} \, d^3x + \mathcal{O}(\omega^{-2}),\\
     \mathbb{Q}^{i j}(0) &= \int_{\R^3} r^i r^j \boldsymbol{\mathcal{E}} \, d^3x=\int_{\R^3} r^i r^j \omega^{\nicefrac{3}{2}} \uu  e^{-2 \omega \Im \phu} \, d^3x + \mathcal{O}(\omega^{-2}) ,\\
    \mathbb{P}_i(0) &= \int_{\R^3} \boldsymbol{\mathcal{S}}_i \, d^3x=\int_{\R^3}\omega^{\nicefrac{3}{2}} \vu _ie^{-2\omega  \Im \phu}d^3x + \mathcal{O}(\omega^{-2}), \\
    \mathbb{J}_i(0) &= \int_{\R^3} \epsilon_{ijk} r^j \boldsymbol{\mathcal{S}} \, d^3x=\int_{\R^3} \epsilon_{ijk} r^j \omega^{\nicefrac{3}{2}} \vu ^k e^{-2\omega  \Im \phu}d^3x + \mathcal{O}(\omega^{-2}),
\end{align}
\end{subequations}
where $r^i = x^i - \bbX^i(0)$. We now write $f=2i\Im\phu$ and recall $\phu \in C^\infty (\R^3, \C)$ with $\frac{1}{2}\Im f=\Im \phu \geq 0$ and $\frac{1}{2}\Im f|_{\x_0}=\Im \phu|_{\x_0} = 0$, $\frac{-\ii}{2}\nabla_i f|_{\x_0}=\nabla_i \Im\phu|_{\x_0} = 0$ for $i = 1,2,3$, $ \Im \nabla_i \nabla_j \phu|_{\x_0}$ is a positive definite matrix, and $ \Im \nabla \phu \neq 0$ in $\cl(\mathcal{K})\setminus \{\x_0\}$. So we can estimate the above integrals with the Theorem~\ref{th:stationary}. For $p=2$ in Theorem~\ref{th:stationary}, we obtain
\begin{subequations}
\begin{align}
    \mathbb{E}(0) &= \int_{\R^3} \omega^{\nicefrac{3}{2}} \uu  e^{i \omega f} \, d^3x + \mathcal{O}(\omega^{-2}) =\frac{\omega^{\nicefrac{3}{2}}e^{\ii \omega f(\x_0)}}{ \sqrt{ \det \left( \frac{\omega}{2 \pi \ii} \Ab\right) } } \Big[\uu (\x_0)+ \omega^{-1} L_1 \uu  (\x_0)\Big] + \mathcal{O}(\omega^{-2}) \nonumber \\
    &\qquad =\frac{\big[\uu (\x_0)+ \omega^{-1}(L_1 \uu)  (\x_0)\big]}{ \sqrt{ \det \left( \frac{\Ab}{2 \pi \ii} \right) } } + \mathcal{O}(\omega^{-2}),\\
    \mathbb{P}_i(0) &= \int_{\R^3}\omega^{\nicefrac{3}{2}} \vu _ie^{i\omega f}d^3x + \mathcal{O}(\omega^{-2}) =\frac{1}{ \sqrt{ \det \left( \frac{\Ab}{2 \pi \ii} \right) } } \Big[\vu_i(\x_0)+ \omega^{-1} (L_1 \vu_i)  (\x_0)\Big]+ \mathcal{O}(\omega^{-2}).
\end{align}
\end{subequations}
For the centre of energy we compute, 
\begin{align}
    \mathbb{X}^i(0) &=\frac{1}{\mathbb{E}} \int_{\R^3} x^i \omega^{\nicefrac{3}{2}} \uu  e^{-2 \omega \Im \phu} \, d^3x + \mathcal{O}(\omega^{-2})=\frac{1}{\mathbb{E}(0) \sqrt{ \det \left( \frac{\Ab}{2 \pi \ii} \right) } } \Big[\x_0^i\uu (\x_0)+ \omega^{-1} (L_1 \x^i \uu ) (\x_0)\Big]+ \mathcal{O}(\omega^{-2}) \nonumber\\
    &=\frac{\x_0^i\uu (\x_0)+ \omega^{-1} (L_1 \x^i \uu ) (\x_0)}{\uu (\x_0)+ \omega^{-1} (L_1 \uu  ) (\x_0)}+ \mathcal{O}(\omega^{-2})=\x_0^i+ \omega^{-1}\frac{ (L_1 \x^i \uu ) (\x_0)-\x_0^i (L_1  \uu ) (\x_0)}{\uu (\x_0)}+ \mathcal{O}(\omega^{-2})\;.
\end{align}
We can now compute, 
\begin{align}
    (L_1 \x^i\uu) (\x_0)-\x_0^i (L_1 \uu) (\x_0) &= \frac{\ii}{2} (\Ab^{-1})^{a i}  \nabla_a\uu+\frac{\ii}{2} (\Ab^{-1})^{i b} \nabla_b\uu- \frac{\ii}{2} (\Ab^{-1})^{i b} (A^{-1})^{c d} \uu ( \nabla_b \nabla_c \nabla_d g )\Big|_{\x_0} ,
\end{align}
where 
\begin{align}
    g(\x)&=2\ii\Im\phu(\x)-2\ii\Im\phu(\x_0)-\ii[\nabla_i\nabla_j\Im\phu](x_0)(\x-\x_0)^i(\x-\x_0)^j\nonumber\\
    &=2\ii\Im\phu(\x)-\ii[\nabla_i\nabla_j\Im\phu](x_0)(\x-\x_0)^i(\x-\x_0)^j. 
\end{align}
So we compute that
\begin{align}
     (\nabla_b \nabla_c \nabla_d g)(\x_0)&=2\ii (\nabla_b \nabla_c \nabla_d\Im\phu)(\x_0). 
\end{align}
This gives
\begin{align} \label{EqXmx}
    \mathbb{X}^i(0) &=\x_0^i+ \omega^{-1}\bigg[\frac{\ii (\Ab^{-1})^{ ia}  (\nabla_a\uu ) (\x_0)}{\uu (\x_0)}+ (\Ab^{-1})^{i b} (\Ab^{-1})^{c d}  (\nabla_b \nabla_c \nabla_d\Im\phu)(\x_0)\bigg]+ \mathcal{O}(\omega^{-2}).
\end{align}
For the quadrupole moment we find
\begin{align}
    \mathbb{Q}^{i j}(0) &= \int_{\R^3} r^i(0,x) r^j(0,x) \omega^{\nicefrac{3}{2}} \uu  e^{-2 \omega \Im \phu} \, d^3x + \mathcal{O}(\omega^{-2})\nonumber\\
    &=\frac{1}{ \sqrt{ \det \left( \frac{\Ab}{2 \pi \ii} \right) } } \Big\{r^i(0,\x_0) r^j(0,\x_0)\uu (x_0)+ \omega^{-1} L_1 [r^i(0,\x) r^j(0,\x) \uu ] (\x_0)\Big\}+ \mathcal{O}(\omega^{-2}) \nonumber\\
     &=\frac{1}{ \sqrt{ \det \left( \frac{\Ab}{2 \pi \ii} \right) } }  \omega^{-1} L_1 [r^i(0,\x) r^j(0,\x) \uu ] (\x_0) + \mathcal{O}(\omega^{-2}),
\end{align}
since $r^i(0,\x_0) = \x_0^i - \mathbb{X}^i(0)=\mathcal{O}(\omega^{-1})$ by \cref{EqXmx}. We now compute
\begin{align}
    L_1(r^ir^j \uu) (x_0) &= \frac{\ii}{2} (\Ab^{-1})^{a b} \nabla_a \nabla_b (r^ir^j \uu) (\x_0)+ \mathcal{O}(\omega^{-1})=\frac{\ii}{2} (\Ab^{-1})^{ij}\uu+ \mathcal{O}(\omega^{-1}),
\end{align}
where we use that $\nabla_ar^i=\delta^i_a$ and $r^i(0,\x_0)=\mathcal{O}(\omega^{-1})$.
Finally, for angular momentum we have
\begin{align}
    \mathbb{J}_i(0) &=\int_{\R^3} \epsilon_{ijk} r^j(0,x) \omega^{\nicefrac{3}{2}} \vu ^k e^{-2\omega  \Im \phu}d^3x + \mathcal{O}(\omega^{-2})\nonumber\\
    &=\frac{\epsilon_{ijk}}{ \sqrt{ \det \left( \frac{\Ab}{2 \pi \ii} \right) } } \Big[ r^j(0,\x_0)\vu^k(\x_0)+ \omega^{-1} L_1 [ r^j(0,\x)\vu^k ] (\x_0)\Big] + \mathcal{O}(\omega^{-2}).
\end{align}
We compute 
\begin{align}
    L_1 [ r^j(0,\x)\vu^k ](\x_0)&=\frac{\ii}{2} (\Ab^{-1})^{a b} \nabla_a \nabla_b [r^j(0,\x)\vu^k] \nonumber\\
    &\qquad- \frac{\ii}{2} (\Ab^{-1})^{a b} (\Ab^{-1})^{c d} \big\{\nabla_a [r^j(0,\x)\vu^k] \big\} ( \nabla_b \nabla_c \nabla_d g )\Big|_{\x_0}+ \mathcal{O}(\omega^{-2}) \nonumber\\
    &=\ii (\Ab^{-1})^{j a} (\nabla_a\vu^k) (\x_0) + (\Ab^{-1})^{j b} (\Ab^{-1})^{c d} \vu^k(\x_0) ( \nabla_b \nabla_c \nabla_d \Im\phu ) (\x_0),
\end{align}
where we used $\nabla_ar^i=\delta^i_a$ and $r^i(0,x_0)=\mathcal{O}(\omega^{-1})$. This completes the proof.
\end{proof}

\subsection{Computation for circularly polarised initial data} \label{Appendix:ComputePCircID}
We recall that
\begin{align}
    \mathbb{P}_i(0) &=\frac{1}{ \sqrt{ \det \left( \frac{\Ab }{2 \pi \ii} \right) } } \Big[\vu_i(\x_0)+ \omega^{-1} (L_1 \vu_i)  (\x_0)\Big] + \mathcal{O}(\omega^{-2}),
\end{align}
with
\begin{align}
    \vu &= \frac{\n^2}{2} \Re \big( \eu_0 \times \overline{\hu}_0 \big) + \frac{\omega^{-1} \n^2}{2} \Re \big( \eu_0 \times \overline{\hu}_1 + \eu_1 \times \overline{\hu}_0 \big).
\end{align}
Therefore,
\begin{align}
    \mathbb{P}(0) &=\frac{1}{ \sqrt{ \det \left( \frac{\Ab}{2 \pi \ii} \right) } } \Big\{\frac{\n^2}{2} \Re \big( \eu_0 \times \overline{\hu}_0 \big) + \frac{\omega^{-1} \n^2}{2} \Re \big( \eu_0 \times \overline{\hu}_1 + \eu_1 \times \overline{\hu}_0 \big)+ \omega^{-1} L_1 \Big[\frac{\n^2}{2} \Re \big( \eu_0 \times \overline{\hu}_0 \big)\Big]  \Big\}\Big|_{\x_0} \nonumber \\ &\qquad + \mathcal{O}(\omega^{-2}).
\end{align}
In~\cref{th:stationary} applied to the current setting, 
\begin{align}
    g(x) &= 2\ii\Im\phu(x) - 2\ii \Im\phu(x_0) - \ii  \nabla_j \nabla_k \Im\phu|_{\x_0}(x - x_0)^j (x - x_0)^k.
\end{align}
Note that by imposing $\nabla_i\Im\phu|_{\x_0}=0$ and $D^{\alpha}\phu|_{\x_0}=0$ for $3\leq |\alpha|\leq 4$ we have $D^{\alpha}g|_{\x_0}=0$ for all $|\alpha|\leq 4$. This fact reduces $L_1u(\x_0)$ to
\begin{align}
    L_1 q (\x_0) &= \frac{\ii}{2} (\Ab ^{-1})^{a b} \nabla_a \nabla_b q (\x_0).
\end{align}
Finally, we recall
 \begin{align}
            D^{\alpha}\hu^i_{0}\bigg|_{\x_0} &= D^{\alpha}\bigg[-\frac{1}{\mu \dot{\phu}} \epsilon\indices{^{i j}_k} \eu^k_{0}\nabla_j\phu  \bigg]\bigg|_{\x_0} &\forall |\alpha|\leq 5,\\ 
            D^{\alpha}\hu^i_{1}\bigg|_{\x_0} &= D^{\alpha} \bigg[ -\frac{1}{\mu \dot{\phu}} \epsilon\indices{^{i j}_k} \eu^k_{1} \nabla_j\phu + \frac{\ii}{\mu \dot{\phu}}\epsilon\indices{^{i j}_k} \nabla_j\eu^k_{0} + \frac{\ii}{\n^2\dot{\phu}^2} \nabla^j\phu \nabla_{j}\hu^i_{0}+ \frac{\ii}{2\n^2 \dot{\phu}^2}\nabla^{i}\phu \hu^{m}_{0} \nabla_m \ln\n^2 \nonumber\\
            &\qquad+\frac{\ii}{2\n^2 \dot{\phu}^2}\hu^{i}_{0} \Big( \Delta\phu-\n^2 \ddot{\phu}  -  \nabla^{m}\phu \nabla_m\ln\eps \Big) \bigg]\bigg|_{\x_0} &\forall |\alpha|\leq 3,  
        \end{align}
        where we can compute $\dot{\phu}$ and its derivatives from the formula $\dot{\phu}=-\frac{\sqrt{\nabla_i\phu\nabla^i\phu}}{\n}$.
Note that we can use these values for $\hu_0$ directly in all these expressions, since they hold to degree $5$ and we have at most $2$ derivatives on $\hu_0$ appearing in $L_1\vu$. 

We now compute the leading order contribution in $\omega$:
\begin{align}
    (\eu_0\times \overline{\hu}_0)^m=-\frac{1}{\mu \overline{\dot{\phu}}}  (\nabla^m\overline{\phu} \eu_0\cdot\overline{\eu}_0-\eu_0^j\nabla_j\overline{\phu} \overline{\eu}^m_{0}).
\end{align}
Taking the real part gives
\begin{align}
    \Re[ (\eu_0\times \overline{\hu}_0)^m]=-\frac{1}{2\mu \overline{\dot{\phu}}}  (\nabla^m\overline{\phu} \eu_0\cdot\overline{\eu}_0-\eu_0^j\nabla_j\overline{\phu} \overline{\eu}^m_{0})-\frac{1}{2\mu \dot{\phu}}  (\nabla^m\phu \eu_0\cdot\overline{\eu}_0-\overline{\eu}_0^j\nabla_j{\phu} \eu^m_{0}).
\end{align}
Evaluating at $\x_0$ and using $\Im\phu|_{\x_0}=0$ and the constraints gives
\begin{align}
    \Re \Big[\frac{\n^2}{2} (\eu_0\times \overline{\hu}_0)^m \Big]\Big|_{x_0}=\frac{\eps \n}{2\mu |\nabla\phu|}  (\nabla^m{\phu}) \eu_0\cdot\overline{\eu}_0\Big|_{x_0}=\frac{\mathfrak{a}\eps\n}{2|\nabla\phu|} \nabla^m{\phu}\Big|_{x_0},
\end{align}
which completes the leading order computation. 

Moving onto the $\omega^{-1}$ contribution, we now note the following expressions:
\begin{subequations} \label{eq:aux_relations}
\begin{align}
\Ab_{ij}|_{x_0}&=\nabla_i\nabla_j[\phu-\overline{\phu}]|_{x_0}=2\nabla_{ij}\phu|_{x_0}=-2\nabla_{ij}\overline{\phu}|_{x_0},\label{eq:A}\\
\dot{\phu}&=-\frac{|\nabla\phu|}{\n}\Big|_{x_0},\label{eq:dotphu}\\
    \nabla_p\dot{\phu}|_{x_0}&=-\frac{1}{\n|\nabla\phu|}\nabla^q\phu\nabla_{q}\nabla_p\phu\Big|_{x_0}=-\frac{1}{2\n|\nabla\phu|}\nabla^q\phu \Ab_{pq}\Big|_{x_0},\label{eq:ddotphu}\\
    \nabla_p\overline{\dot{\phu}}|_{x_0}&=\overline{\nabla_p\dot{\phu}}|_{x_0}=-\frac{1}{\n\overline{|\nabla\phu|}}\nabla^q\overline{\phu} \nabla_q\nabla_p\overline{\phu}\Big|_{x_0}=\frac{1}{2\n|\nabla\phu|}\nabla^q\phu \Ab_{pq}\Big|_{x_0},\label{eq:ddotphuc}\\
    \ddot{\phu}|_{x_0}&=\frac{\nabla^j\phu}{\n^2|\nabla\phu|^2}\nabla^p\phu\nabla_p\nabla_j\phu\Big|_{x_0}=\frac{\nabla^j\phu}{2\n^2|\nabla\phu|^2}\nabla^p\phu \Ab_{pj}\Big|_{x_0},\label{eq:ddotdotphu}\\
    \nabla_q\nabla_p\dot{\phu}|_{x_0}&=-\frac{1}{4\n|\nabla\phu|}\Ab_{q}^m\Ab_{m}^p+\frac{1}{4\n|\nabla\phu|^3}\nabla^n\phu\nabla^m\phu \Ab_{nq}\Ab_{mp}\Big|_{x_0},\label{eq:dddotphu}\\
    \nabla_q\nabla_p\overline{\dot{\phu}}|_{x_0}&=-\frac{1}{4\n|\nabla\phu|}\Ab_{q}^m\Ab_{m}^p+\frac{1}{4\n|\nabla\phu|^3}\nabla^n\phu\nabla^m\phu \Ab_{nq}\Ab_{mp}\Big|_{x_0},\label{eq:dddotphuc}
\end{align}
\end{subequations}
We now compute (ignoring derivatives of $\eps$ and $\mu$ since our medium is nearly homogeneous),
\begin{align}
   &(\overline{\eu}_0\times \hu_{1})^m\bigg|_{\x_0} \nonumber \\
   &\quad=(\overline{\eu}_0)_k\bigg[\frac{\ii}{\mu \dot{\phu}} \nabla^m\eu^k_{0} -\frac{1}{\mu \dot{\phu}}  \eu^k_{1} \nabla^m\phu  - \frac{\ii}{\n^2\mu\dot{\phu}^2} \nabla^p\phu \nabla_{p}\bigg(\frac{1}{ \dot{\phu}} \eu^k_{0}\nabla^m\phu  \bigg)-\frac{\ii\eu^k_{0}\nabla^m\phu  }{2\mu\n^2 \dot{\phu}^3}  \Big( \Delta\phu-\n^2 \ddot{\phu}\Big) \bigg]\bigg|_{\x_0} \nonumber\\
   &\quad-(\overline{\eu}_0)^j\bigg[ \frac{\ii}{\mu \dot{\phu}} \nabla_j\eu^m_{0} -\frac{1}{\mu \dot{\phu}}  \eu^m_{1} \nabla_j\phu - \frac{\ii}{\mu\n^2\dot{\phu}^2} \nabla^p\phu \nabla_{p}\bigg(\frac{1}{ \dot{\phu}} \eu^m_{0}\nabla_j\phu  \bigg)-\frac{\ii\eu^m_{0}\nabla_j\phu   }{2\mu\n^2 \dot{\phu}^3} \Big( \Delta\phu-\n^2 \ddot{\phu}  \Big) \bigg]\bigg|_{\x_0}. 
\end{align}
We recall $\nabla\Im\phu|_{\x_0}=0$ and the constraint $\eu_0^i\nabla_i\phu=0$ as well as the relations from~\eqref{EqIDCircCon} that
\begin{subequations}\label{eq:de0}
\begin{align}
    \nabla_a \eu_0^i \big|_{\x_0} &=- \frac{1}{|\nabla \phu|^2} \Big( \eu_0^b \nabla_a \nabla_b \phu \Big) \nabla^i \phu \Big|_{\x_0}= - \frac{1}{2|\nabla \phu|^2}  \eu_0^c \Ab_{ac}\nabla^i \phu \Big|_{\x_0},\\
    \nabla_a \overline{\eu}_0^i \big|_{\x_0} &=  \frac{1}{2|\nabla \phu|^2}  \overline{\eu}_0^c \Ab_{ac}\nabla^i \phu \Big|_{\x_0}
\end{align}    
\end{subequations}
to produce
\begin{align}
   &(\overline{\eu}_0\times \hu_{1}+\eu_0\times \overline{\hu}_{1})^m\bigg|_{\x_0} \nonumber\\
   &=\frac{\ii \overline{\eu}_0\cdot\eu_0}{\mu\n^2\dot{\phu}^2}\bigg\{ \nabla^p\phu \nabla_{p}\overline{\bigg(\frac{1}{ \dot{\phu}} \nabla^m\phu  \bigg)} - \nabla^p\phu \nabla_{p}\bigg(\frac{1}{ \dot{\phu}} \nabla^m\phu  \bigg)-\frac{1}{ 2\dot{\phu}} \nabla^m\phu   \Big[ \Delta(\phu-\overline{\phu})-\n^2 (\ddot{\phu}  -\ddot{\overline{\phu}}) \Big]\bigg\} \bigg|_{\x_0} \nonumber\\
   &+\frac{\ii\overline{\eu}_0^j\eu_0^b}{\mu\n^2\dot{\phu}^3}\nabla_j\nabla_b(\phu-\overline{\phu}) \nabla^m\phu-\frac{\ii}{\mu\n^2\dot{\phu}^3} \eu_0^j \overline{\eu}^m_{0} \nabla^p\phu \nabla_{p}\nabla_j\overline{\phu}+ \frac{\ii}{\mu\n^2\dot{\phu}^3} \eu_0^j \overline{\eu}^m_{0}\nabla^p\phu \nabla_{p}\nabla_j\phu \bigg|_{\x_0}. 
\end{align}
Using the prescribed value of $\eu_1$ from \cref{EqIDCircCon}, we now compute 
\begin{align}
    (\eu_1\times \overline{\hu}_0)^m|_{x_0}=\ii \div\eu_0\frac{1}{\mu \dot{\phu}}  \overline{\eu}^m_{0}\Big|_{x_0}= -\frac{\ii}{\mu \dot{\phu} |\nabla \phu|^2} \overline{\eu}^m_{0} \Big( \eu_0^b \nabla_a \nabla_b \phu \Big) \nabla^a \phu\Big|_{x_0}=-\frac{\ii}{\mu \n^2 \dot{\phu}^3}  \overline{\eu}^m_{0} \Big( \eu_0^b \nabla_a \nabla_b \phu \Big) \nabla^a \phu\Big|_{x_0}.
\end{align}
Combining these results gives
\begin{align}
    &2\Re(\eu_0\times \overline{\hu}_{1}+\eu_1\times \overline{\hu}_0)^m\bigg|_{\x_0}\nonumber \\ 
    &=\frac{\ii \overline{\eu}_0\cdot\eu_0}{\mu\n^2\dot{\phu}^2} \bigg\{ \frac{1}{ \dot{\phu}^2} \nabla^p\phu \nabla_{p}\dot{\phu}\nabla^m\phu-\frac{1}{ \dot{\phu}^2} \nabla^p\phu \nabla_{p}\overline{\dot{\phu}}\nabla^m\phu +\frac{1}{\dot{\phu}}\nabla^p\phu\nabla_p\nabla^m(\overline{\phu}-\phu) \nonumber\\
    &\qquad\qquad\qquad-\frac{1}{ 2\dot{\phu}} \nabla^m\phu   \Big[ \Delta[\phu-\overline{\phu}]-\n^2 (\ddot{\phu}  -\ddot{\overline{\phu}}) \Big]\bigg\} +\frac{\ii}{\mu \n^2 \dot{\phu}^3} \overline{\eu}_0^j\eu_0^b \nabla_j\nabla_b(\phu-\overline{\phu})\nabla^m\phu \bigg|_{\x_0}.
\end{align}
Using the relations \eqref{eq:aux_relations},
we obtain
\begin{align}
    2\Re(\eu_0\times \overline{\hu}_{1}+\eu_1\times \overline{\hu}_0)^m\bigg|_{\x_0} 
    &=\frac{\ii\n\overline{\eu}_0\cdot\eu_0}{\mu|\nabla\phu|^3}\bigg(\nabla^p\phu \Ab_p^m -\frac{3}{ 2|\nabla\phu|^2} \nabla^p\phu \nabla^q\phu\nabla^m\phu \Ab_{pq} +\frac{1}{ 2} \nabla^m\phu \Ab_p^p\bigg)\nonumber \\
    &\qquad-\frac{\ii\n(\overline{\eu}_0)^j\eu_0^i}{\mu |\nabla\phu|^3}\Ab_{ij}\nabla^m\phu \bigg|_{\x_0}.
\end{align}
We now compute $2$ derivatives of
\begin{align}
    \frac{1}{2}\n^2(\eu_0\times \overline{\hu}_0)^m=-\frac{\eps}{2 \overline{\dot{\phu}}}  (\nabla^m\overline{\phu} \eu_0\cdot\overline{\eu}_0-\eu_0^j\nabla_j\overline{\phu} \overline{\eu}^m_{0}).
\end{align}
Similarly, ignoring derivatives of $\eps$ gives
\begin{align}
   \nabla_a\Big[ \frac{1}{2}\n^2(\eu_0\times \overline{\hu}_0)^m\Big]&=\frac{\eps}{2 \overline{\dot{\phu}}^2}\nabla_a\overline{\dot{\phu}}  (\nabla^m\overline{\phu} \eu_0\cdot\overline{\eu}_0-\eu_0^j\nabla_j\overline{\phu} \overline{\eu}^m_{0}) -\frac{\eps}{2 \overline{\dot{\phu}}}  (\nabla_a\nabla^m\overline{\phu} \eu_0\cdot\overline{\eu}_0-\eu_0^j\nabla_a\nabla_j\overline{\phu} \overline{\eu}^m_{0}) \nonumber\\
   &-\frac{\eps}{2 \overline{\dot{\phu}}}  (\nabla^m\overline{\phu} \nabla_a\eu_0\cdot\overline{\eu}_0-\nabla_a\eu_0^j\nabla_j\overline{\phu} \overline{\eu}^m_{0})-\frac{\eps}{2 \overline{\dot{\phu}}}  (\nabla^m\overline{\phu} \eu_0\cdot\nabla_a\overline{\eu}_0-\eu_0^j\nabla_j\overline{\phu} \nabla_a\overline{\eu}^m_{0}) .
\end{align}
Taking another derivative and evaluating at $\x_0$, ignoring third derivatives of $\phu$ and derivatives of $\eps$ gives
\begin{align}
   \nabla_b\nabla_a&\Big[ \frac{1}{2}\n^2(\eu_0\times \overline{\hu}_0)^m\Big]\Big|_{x_0} =\frac{\eps}{2 \overline{\dot{\phu}}}  (\nabla_b\eu_0^j\nabla_a\overline{\eu}^m_{0}+\nabla_a\eu_0^j\nabla_b\overline{\eu}^m_{0})\nabla_j\overline{\phu} -\frac{\eps}{2 \overline{\dot{\phu}}}  ( \nabla_a\eu_0\cdot\nabla_b\overline{\eu}_0+\nabla_b\eu_0\cdot\nabla_a\overline{\eu}_0)\nabla^m\overline{\phu} \nonumber\\
   &+\frac{\eps}{2 \overline{\dot{\phu}}}  \eu_0^j\Big(\nabla_a\nabla_j\overline{\phu} \nabla_b\overline{\eu}^m_{0}+\nabla_b\nabla_j\overline{\phu} \nabla_a\overline{\eu}^m_{0}\Big) +\frac{\eps}{2 \overline{\dot{\phu}}}  \overline{\eu}^m_{0}\Big(\nabla_a\eu_0^j\nabla_b\nabla_j\overline{\phu} +\nabla_b\eu_0^j\nabla_a\nabla_j\overline{\phu} \Big)\nonumber \\
   &+\frac{\eps}{2 \overline{\dot{\phu}}^2}\nabla_b\overline{\dot{\phu}}\Big(  \nabla_a\nabla^m\overline{\phu} \eu_0\cdot\overline{\eu}_0-   \eu_0^j\nabla_a\nabla_j\overline{\phu} \overline{\eu}^m_{0}-  \nabla_a\eu_0^j\nabla_j\overline{\phu} \overline{\eu}^m_{0}\Big)\nonumber \\
   &+\frac{\eps}{2 \overline{\dot{\phu}}^2}\nabla_a\overline{\dot{\phu}}\Big(  \nabla_b\nabla^m\overline{\phu} \eu_0\cdot\overline{\eu}_0-   \eu_0^j\nabla_b\nabla_j\overline{\phu} \overline{\eu}^m_{0}-  \nabla_b\eu_0^j\nabla_j\overline{\phu} \overline{\eu}^m_{0}\Big)\nonumber \\
   &+\Big(\frac{\eps}{2 \overline{\dot{\phu}}^2}\nabla_b\nabla_a\overline{\dot{\phu}}-\frac{\eps}{ \overline{\dot{\phu}}^3}\nabla_b\overline{\dot{\phu}}\nabla_a\overline{\dot{\phu}} \Big) \nabla^m\overline{\phu} \eu_0\cdot\overline{\eu}_0+\frac{\eps}{2 \overline{\dot{\phu}}}  \nabla_b\nabla_a\eu_0^j\nabla_j\overline{\phu} \overline{\eu}^m_{0}\Big|_{x_0}. 
\end{align}
Tracing with $\Ab ^{-1}$ and using \cref{eq:A} gives
\begin{align}
   (\Ab ^{-1})^{ab}\nabla_b\nabla_a\Big[ \frac{1}{2}\n^2(\eu_0\times \overline{\hu}_0)^m\Big]\Big|_{x_0}&=\frac{\eps}{ \overline{\dot{\phu}}} (\Ab ^{-1})^{ab}\nabla_b\eu_0^j\nabla_a\overline{\eu}^m_{0}\nabla_j\overline{\phu} -\frac{\eps}{ \overline{\dot{\phu}}}(\Ab ^{-1})^{ab} \nabla_a\eu_0\cdot\nabla_b\overline{\eu}_0\nabla^m\overline{\phu} \nonumber\\
   &-\frac{\eps}{ 2\overline{\dot{\phu}}}  \eu_0^j\nabla_j\overline{\eu}^m_{0} -\frac{\eps}{ 2\overline{\dot{\phu}}}  \overline{\eu}^m_{0}\nabla_j\eu_0^j+\frac{\eps}{2 \overline{\dot{\phu}}}  (\Ab ^{-1})^{ab}\nabla_b\nabla_a\eu_0^j\nabla_j\overline{\phu} \overline{\eu}^m_{0}\nonumber \\
   &+\frac{\eps}{ \overline{\dot{\phu}}^2} \Big[  -\frac{1}{2}\nabla^m\overline{\dot{\phu}}\eu_0\cdot\overline{\eu}_0+  \frac{1}{2}\eu_0^j\nabla_j\overline{\dot{\phu}} \overline{\eu}^m_{0}-  (\Ab ^{-1})^{ab}\nabla_b\overline{\dot{\phu}}\nabla_a\eu_0^j\nabla_j\overline{\phu} \overline{\eu}^m_{0}\Big]\nonumber \\
   &+(\Ab ^{-1})^{ab}\Big(\frac{\eps}{2 \overline{\dot{\phu}}^2}\nabla_b\nabla_a\overline{\dot{\phu}}-\frac{\eps}{ \overline{\dot{\phu}}^3}\nabla_b\overline{\dot{\phu}}\nabla_a\overline{\dot{\phu}} \Big) \nabla^m\overline{\phu} \eu_0\cdot\overline{\eu}_0\Big|_{x_0}. 
\end{align}

Using the relations in \cref{eq:aux_relations,eq:de0} and
\begin{subequations}
\begin{align}
    \nabla_a \nabla_b \eu_0^i \big|_{\x_0} &= - \frac{1}{|\nabla \phu|^2} \Big[ (\nabla_{a} \eu_0^c \nabla_{b} +\nabla_{b} \eu_0^c \nabla_{a})\nabla_c \phu \Big] \nabla^i \phu \Big|_{\x_0}=  \frac{\eu_0^c}{4|\nabla \phu|^4} (\Ab_{ac} \Ab_{bj}+\Ab_{bc} \Ab_{aj})\nabla^j\phu \nabla^i \phu\Big|_{\x_0},\\
    \nabla_a \nabla_b \overline{\eu}_0^i \big|_{\x_0} &=  \frac{\overline{\eu}_0^c}{4|\nabla \phu|^4} (\Ab_{ac} \Ab_{bj}+\Ab_{bc} \Ab_{aj})\nabla^j\phu \nabla^i \phu\Big|_{\x_0},
\end{align}    
\end{subequations}
we obtain
\begin{align}
   (\Ab ^{-1})^{ab}\nabla_b\nabla_a\Big[ \frac{1}{2}\n^2(\eu_0\times \overline{\hu}_0)^m\Big]\Big|_{x_0}&=\frac{\eps\n}{ 4|\nabla \phu|}  \eu_0^j  \overline{\eu}_0^c \Ab_{jc}\nabla^m \phu -\frac{\eps\n}{ 4|\nabla\phu|^3}\nabla^q\phu \Ab_{q}^m\eu_0\cdot\overline{\eu}_0 \nonumber\\
   &\qquad+\frac{\eps\n}{8|\nabla\phu|^3}\Big(\frac{3}{| \nabla\phu|^2}\nabla^q\phu \Ab_{bq}\nabla^b\phu-\Ab_{p}^p\Big) \nabla^m{\phu} \eu_0\cdot\overline{\eu}_0\Big|_{x_0}. 
\end{align}
So, 
\begin{align}
   (\Ab ^{-1})^{ab}\nabla_b\nabla_a&\Big[ \frac{1}{2}\n^2(\overline{\eu}_0\times {\hu}_0)^m+\frac{1}{2}\n^2(\eu_0\times \overline{\hu}_0)^m\Big]\Big|_{x_0} =\frac{\eps\n\nabla^m \phu}{ 2|\nabla \phu|}  \eu_0^j  \overline{\eu}_0^c \Ab_{jc} -\frac{\eps\n}{ 2|\nabla\phu|^3}\nabla^q\phu \Ab_{q}^m\eu_0\cdot\overline{\eu}_0 \nonumber\\
   &\qquad+\frac{\eps\n}{4|\nabla\phu|^3}\Big(\frac{3}{| \nabla\phu|^2}\nabla^q\phu \Ab_{bq}\nabla^b\phu-\Ab_{p}^p\Big) \nabla^m{\phu} \eu_0\cdot\overline{\eu}_0\Big|_{x_0}. 
\end{align}
Combining the $\omega^{-1}$ terms gives
\begin{align}
    &\frac{\n^2}{2}\Re(\eu_0\times \overline{\hu}_{1}+\eu_1\times \overline{\hu}_0)^m+\frac{\ii}{2} (\Ab ^{-1})^{a b} \nabla_a \nabla_b \Big[\frac{\n^2}{2} \Re \big( \eu_0 \times \overline{\hu}_0 \big)\Big]\bigg|_{\x_0} \nonumber\\ 
    &\quad=\frac{\ii\eps\n \overline{\eu}_0\cdot\eu_0}{8|\nabla\phu|^3} \bigg(\frac{1}{ 2}  \Ab_p^p -\frac{3}{ 2|\nabla\phu|^2} \nabla^p\phu \nabla^q\phu \Ab_{pq} \bigg)\nabla^m\phu-\frac{\ii\eps\n \overline{\eu}_0^j\eu_0^i}{8 |\nabla\phu|^3}  \Ab_{ij}\nabla^m\phu+\frac{\ii\eps\n \overline{\eu}_0\cdot\eu_0}{8|\nabla\phu|^3}\nabla^p\phu \Ab_p^m \bigg|_{\x_0}. 
\end{align} 
Projecting onto the orthonormal frame,
and noting $\eu_0\cdot\overline{\eu}_0=\mathfrak{a}^2$, we can then write
\begin{align}
    &\frac{\n^2}{2}\Re(\eu_0\overline{\hu}_{1}+\eu_1\times \overline{\hu}_0)^m+\frac{\ii}{2} (\Ab ^{-1})^{a b} \nabla_a \nabla_b \Big[\frac{\n^2}{2} \Re \big( \eu_0 \times \overline{\hu}_0 \big)\Big]\bigg|_{\x_0}\nonumber \\
    &\qquad=\frac{\ii\eps\n\mathfrak{a}^2}{16|\nabla\phu|^2}\bigg[  \Ab_p^p -\frac{\nabla^p\phu}{|\nabla\phu|} \frac{\nabla^q\phu}{|\nabla\phu|} \Ab_{pq} -(X^iX^j+Y^iY^j)\Ab_{ij}\bigg]\frac{\nabla_m\phu}{|\nabla\phu|} \nonumber\\
    &\qquad\qquad+\frac{\ii\eps\n\mathfrak{a}^2}{8|\nabla\phu|^2}\frac{\nabla^p\phu}{|\nabla\phu|} \Ab_{pq}(Y^qY^m+X^qX^m) \bigg|_{\x_0}. 
\end{align}
Since the trace is basis independent, the term in the first line now vanishes. 

\pagebreak

\printbibliography

\end{document}